\documentclass[review]{elsarticle}
\usepackage{srcltx}
\usepackage{eurosym}
\usepackage{mathtools}
\usepackage{amsmath}
\usepackage{amsfonts}
\usepackage{amssymb}
\usepackage{amsthm}
\usepackage{graphicx}
\usepackage{mathrsfs}
\usepackage{xcolor}
\usepackage{exscale}
\usepackage{latexsym}
\usepackage[colorlinks,plainpages=true,pdfpagelabels,hypertexnames=true,colorlinks=true,pdfstartview=FitV,linkcolor=blue,citecolor=red,urlcolor=black]{hyperref}
\PassOptionsToPackage{unicode}{hyperref}
\PassOptionsToPackage{naturalnames}{hyperref}
\usepackage{enumerate}
\usepackage[shortlabels]{enumitem}
\usepackage{bookmark}
\usepackage{wasysym}
\usepackage{esint}
\usepackage[ddmmyyyy]{datetime}
\usepackage[margin=2.2cm]{geometry}

\numberwithin{equation}{section}
\everymath{\displaystyle}
\newtheorem{theorem}{Theorem}[section]
\newtheorem{lemma}[theorem]{Lemma}
\newtheorem{corollary}[theorem]{Corollary}

\newtheorem{definition}[theorem]{Definition}% Use {\rm ...}
\newtheorem{remark}[theorem]{Remark}        % Use {\rm ...}
\newtheorem{assumption}[theorem]{Assumption}  
\numberwithin{equation}{section}
\def\aa{\mathcal{A}}

\newcommand{\lamot}{\La_0,\La_1}
\newcommand{\vwspace}{L^{\pp-\be}(-T,T; W_0^{1,\pp-\be}(\Om))}

\newcommand{\pa}{\partial}
\newcommand{\ddt}[1]{\frac{d#1}{dt}}
\newcommand{\dds}[1]{\frac{d#1}{ds}}
\newcommand{\vlh}{\lsbt{v}{\la,h}}
\newcommand{\vl}{\lsbt{v}{\la}}
\newcommand{\elam}{\lsbo{E}{\lambda}}
\newcommand{\bM}{{\bf M}}
\newcommand{\mm}{\mathcal{M}}
\newcommand{\pp}{p(\cdot)}

\newcommand{\sss}{s(\cdot)}

\newcommand{\modp}{\om_{\pp}}

\newcommand{\plog}{p^{\pm}_{\log}}

\newcommand{\slog}{s^{\pm}_{\log}}
\newcommand{\logh}{\log^{\pm}}
\newcommand{\tQ}{\tilde{Q}}

\newcommand{\tx}{\tilde{x}}

\newcommand{\tlt}{\tilde{t}}

\newcommand{\tz}{\tilde{z}}
\newcommand{\hz}{\hat{z}}
\newcommand{\lz}{\overline{z}}
\makeatletter
\newcommand{\vo}{\vec{o}\@ifnextchar{^}{\,}{}}
\makeatother
\def\Yint#1{\mathchoice
    {\YYint\displaystyle\textstyle{#1}}%
    {\YYint\textstyle\scriptstyle{#1}}%
    {\YYint\scriptstyle\scriptscriptstyle{#1}}%
    {\YYint\scriptscriptstyle\scriptscriptstyle{#1}}%
      \!\iint}
\def\YYint#1#2#3{{\setbox0=\hbox{$#1{#2#3}{\iint}$}
    \vcenter{\hbox{$#2#3$}}\kern-.50\wd0}}
\def\longdash{-\mkern-9.5mu-} %USE THIS IF \usepackage{fourier} IS NOT USED. 
\def\tiltlongdash{\rotatebox[origin=c]{18}{$\longdash$}}
\def\fiint{\Yint\tiltlongdash}
\def\Xint#1{\mathchoice
    {\XXint\displaystyle\textstyle{#1}}%
    {\XXint\textstyle\scriptstyle{#1}}%
    {\XXint\scriptstyle\scriptscriptstyle{#1}}%
    {\XXint\scriptscriptstyle\scriptscriptstyle{#1}}%
      \!\int}
\def\XXint#1#2#3{{\setbox0=\hbox{$#1{#2#3}{\int}$}
    \vcenter{\hbox{$#2#3$}}\kern-.50\wd0}}
\def\hlongdash{-\mkern-13.5mu-}
\def\tilthlongdash{\rotatebox[origin=c]{18}{$\hlongdash$}}
\def\hint{\Xint\tilthlongdash}
\makeatletter
\def\namedlabel#1#2{\begingroup
   \def\@currentlabel{#2}%
   \label{#1}\endgroup
}
\makeatother
\makeatletter
\newcommand{\rmh}[1]{\mathpalette{\raisem@th{#1}}}
\newcommand{\raisem@th}[3]{\hspace*{-1pt}\raisebox{#1}{$#2#3$}}
\makeatother

\newcommand{\lsb}[2]{#1_{\rmh{-3pt}{#2}}}
\newcommand{\lsbo}[2]{#1_{\rmh{-1pt}{#2}}}
\newcommand{\lsbt}[2]{#1_{\rmh{-2pt}{#2}}}
\newcommand{\redref}[2]{\texorpdfstring{\protect\hyperlink{#1}{\textcolor{black}{(}\textcolor{red}{#2}\textcolor{black}{)}}}{}}
\newcommand{\redlabel}[2]{\hypertarget{#1}{\textcolor{black}{(}\textcolor{red}{#2}\textcolor{black}{)}}}
\newcounter{desccount}
\newcommand{\descitem}[2]{\item[#1] \refstepcounter{desccount}\label{#2}}
\newcommand{\descref}[2]{\hyperref[#1]{\textcolor{black}{(}\textcolor{blue}{\bf #2}\textcolor{black}{)}}}
\newcommand{\dref}[2]{\hyperref[#1]{\textcolor{black}{(}\textcolor{blue}{\bf #2}\textcolor{black}{)}}}
\newcommand{\tuh}{\tilde{u}_h}

\newcommand{\tp}{\tilde{p}}
\newcommand{\tq}{\tilde{q}}

\newcommand{\tu}{\tilde{u}}

\newcommand{\mfz}{\mathfrak{z}}
\newcommand{\mft}{\mathfrak{t}}

\newcommand{\mfi}{\mathfrak{X}}
\newcommand{\mfq}{\mathfrak{Q}}

\newcommand{\mcc}{\mathcal{C}}
\newcommand{\mcf}{\mathcal{F}}
\newcommand{\mcg}{\mathcal{G}}
\newcommand{\mch}{\mathcal{H}}
\newcommand{\mci}{\mathcal{I}}

\newcommand{\mbfq}{{\bf \mathfrak{Q}}}
\newcommand\RR{\mathbb{R}}

\newcommand\NN{\mathbb{N}}
\newcommand{\al}{\alpha}
\newcommand{\be}{\beta}
\newcommand{\ga}{\gamma}
\newcommand{\de}{\delta}
\newcommand{\ve}{\varepsilon}

\newcommand{\tht}{\theta}

\newcommand{\sig}{\sigma}

\newcommand{\om}{\omega}
\newcommand{\la}{\lambda}
\newcommand{\vt}{\vartheta}
\newcommand{\Th}{\Theta}

\newcommand{\Om}{\Omega}

\newcommand{\La}{\Lambda}
\newcommand{\tTh}{{\Upsilon}}
\DeclareMathOperator{\dv}{div}
\DeclareMathOperator{\spt}{spt}

\DeclareMathOperator{\diam}{diam}

\newcommand{\iprod}[2]{\langle #1 \ ,  #2\rangle}

\newcommand{\abs}[1]{\left| #1\right|}

\newcommand{\mgh}[1]{\left\{ #1\right\}}

\newcommand{\lbr}[1][(]{\left#1}
\newcommand{\rbr}[1][)]{\right#1}
\newcommand{\avgs}[2]{\lsbo{\lbr #1 \rbr}{#2}}

\newcommand{\txt}[1]{\qquad \text{#1} \qquad}
\newcommand{\qnot}{Q_{\rho}^{\al_0}(z_0)}
\newcommand{\qone}{Q_{\rho_a}^{\al_0}(z_0)}

\newcommand{\qthr}{Q_{\rho_a+\frac23(\rho_b - \rho_a)}^{\al_0}(z_0)}
\newcommand{\qfur}{Q_{\rho_b}^{\al_0}(z_0)}
\newcommand{\qfurr}{Q_{2\rho_b}^{\al_0}(z_0)}
\newcommand{\qtoo}{Q_{2\rho}^{\al_0}(z_0)}
\newcommand{\qfor}{Q_{4\rho}^{\al_0}(z_0)}
\newcommand{\qfiv}{Q_{8\rho}^{\al_0}(z_0)}
\newcommand{\qfve}{Q_{16\rho}^{\al_0}(z_0)}

\newcommand{\qzero}{Q_{\rho_{z_0}}^{\al_0}(z_0)}
\newcommand{\qzeros}{Q_{2\rho_{z_0}}^{\al_0}(z_0)}
\newcommand{\qzerof}{Q_{4\rho_{z_0}}^{\al_0}(z_0)}
\newcommand{\qzerofi}{Q_{4\rho_{z_i}}^{\al_0}(z_i)}
\newcommand{\qzeross}{Q_{4\mfi\rho_{z_0}}^{\al_0}(z_0)}
\newcommand{\qzerossi}{Q_{4\mfi\rho_{z_i}}^{\al_0}(z_i)}

\newcommand{\bnot}{B_{\rho}^{\al_0}(x_0)}
\newcommand{\btoo}{B_{2\rho}^{\al_0}(x_0)}
\newcommand{\bone}{B_{\rho_a}^{\al_0}(x_0)}

\newcommand{\bfur}{B_{\rho_b}^{\al_0}(x_0)}
\newcommand{\bfve}{B_{16\rho}^{\al_0}(x_0)}

\newcommand{\itoo}{I_{2\rho}^{\al_0}(t_0)}
\newcommand{\ione}{I_{\rho_a}^{\al_0}(t_0)}

\newcommand{\qthrs}{Q_{\rho_2}^{\al_0}(z_0)}

\newcommand{\bthrs}{B_{\rho_2}^{\al_0}(x_0)}

\newcommand{\itwos}{I_{\rho_1}^{\al_0}(t_0)}

\newcommand{\musym}{\lbr \frac{16\rho}{\rho_b-\rho_a} \rbr}
\newcommand{\htq}{\hat{Q}}

\newcommand{\lnq}{\overline{Q}}

\newcommand{\htb}{\hat{B}}

\newcommand{\lnb}{\overline{B}}

\newcommand{\hti}{\hat{I}}

\newcommand{\lni}{\overline{I}}

\newcommand{\htx}{\hat{x}}

\newcommand{\lnx}{\overline{x}}

\newcommand{\httm}{\hat{t}}

\newcommand{\lntm}{\overline{t}}

\newcommand{\htr}{\hat{r}}

\newcommand{\rzb}{\rho_{\overline{z}}}
\newcommand{\rzh}{\rho_{\hat{z}}}

\newcommand{\scaletexp}[2]{-1+d}

\newcommand{\scalexexpn}[2]{-\frac{n}{p(#2)}+\frac{nd}{2}}

\newcommand{\scalex}[2]{#1^{-\frac{1}{p(#2)}+\frac{d}{2}}}
\newcommand{\scalet}[2]{#1^{-1+d}}

\newcommand{\scalexn}[2]{#1^{-\frac{n}{p(#2)}+\frac{nd}{2}}}

\newcommand{\nscalex}[2]{#1^{\frac{1}{p(#2)}-\frac{d}{2}}}
\newcommand{\nscalet}[2]{#1^{1-d}}
\newcounter{whitney}
\refstepcounter{whitney}

\newcounter{ineqcounter}
\refstepcounter{ineqcounter}
\makeatletter
\g@addto@macro\normalsize{%
  \setlength\abovedisplayskip{3pt}
  \setlength\belowdisplayskip{3pt}
  \setlength\abovedisplayshortskip{1pt}
  \setlength\belowdisplayshortskip{3pt}
}
\makeatletter
\def\ps@pprintTitle{%
 \let\@oddhead\@empty
 \let\@evenhead\@empty
 \def\@oddfoot{}%
 \let\@evenfoot\@oddfoot}
\makeatother
\begin{document}

\begin{frontmatter}

\title{Interior and boundary higher integrability of very weak solutions for quasilinear parabolic equations with variable exponents.}

\author[myaddress]{Karthik Adimurthi\corref{mycorrespondingauthor}\tnoteref{thanksfirstauthor}}
\cortext[mycorrespondingauthor]{Corresponding author}
\ead{karthikaditi@gmail.com and kadimurthi@snu.ac.kr}
\tnotetext[thanksfirstauthor]{Supported by the National Research Foundation of Korea grant NRF-2015R1A4A1041675.}

\author[myaddress,myaddresstwo]{Sun-Sig Byun\tnoteref{thankssecondauthor}}
\ead{byun@snu.ac.kr}
\tnotetext[thankssecondauthor]{Supported by the National Research Foundation of Korea grant  NRF-2017R1A2B2003877.}

\author[myaddress]{Jehan Oh\tnoteref{thanksthirdauthor}}
\ead{ojhan0306@snu.ac.kr}
% \tnotetext[thanksthirdauthor]{Supported by the National Research Foundation of Korea grant  NRF-????. }

\address[myaddress]{Department of Mathematical Sciences, Seoul National University, Seoul 08826, Korea.}
\address[myaddresstwo]{Research Institute of Mathematics, Seoul National University, Seoul 08826, Korea.}

\begin{abstract}
We prove boundary higher integrability for the (spatial) gradient of \emph{very weak} solutions of quasilinear parabolic equations of the form 
\[
\left\{
 \begin{array}{ll}
u_t - \dv  \aa(x,t,\nabla u) = 0 &\quad  \text{on} \ \Om \times (-T,T), \\
u = 0 &\quad  \text{on} \ \partial \Om \times (-T,T),
 \end{array}
\right.
\]
 where the non-linear structure $\mathcal{A}(x, t,\nabla u)$ is modelled after the variable exponent $p(x,t)$-Laplace operator given by $|\nabla u|^{p(x,t)-2} \nabla u$.   To this end, we prove that the gradients satisfy a reverse H\"older inequality near the boundary by constructing a suitable test function which is Lipschitz continuous and  preserves the boundary values. In the interior case, such a result was proved in \cite{bogelein2014very} provided $p(x,t) \geq \mathfrak{p}^- \geq 2$ holds  and was then extended to the singular case  $\frac{2n}{n+2}< \mathfrak{p}^-\leq p(x,t)\leq \mathfrak{p}^+ \leq 2$  in \cite{li2017very}. This restriction was necessary because the intrinsic scalings for quasilinear parabolic problems are different in the case $\mathfrak{p}^+ \leq 2$ and $\mathfrak{p}^-\geq 2$. 

In this paper, we develop a new unified intrinsic scaling, using which, we are able to extend the results of \cite{bogelein2014very,li2017very} to the full range $\frac{2n}{n+2} < \mathfrak{p}^- \leq  p(x,t)\leq \mathfrak{p}^+<\infty$ and also obtain analogous results upto the boundary. \emph{The main novelty of this paper is that our methods are able to handle both the singular case and degenerate case simultaneously.} To simplify the exposition, we will only prove the higher integrability result near the boundary, provided the domain $\Om$ satisfies a uniform measure density condition. Our techniques are also applicable to higher order equations as well as systems.

\end{abstract}


\end{frontmatter}

% \begin{doublespacing}
\tableofcontents

\section{Introduction}

In this paper, we consider the boundary regularity  of quasilinear parabolic equations of the form 
\begin{equation}
\label{main}
 \left\{
 \begin{array}{ll}
u_t - \dv  \aa(x,t,\nabla u) = 0 &\quad \text{on} \ \Om \times (-T,T), \\
u = 0 & \quad \text{on} \ \partial \Om \times (-T,T),
 \end{array}
\right.
\end{equation}
where the non-linear structure $\mathcal{A}(x, t,\nabla u)$ is modelled after the variable exponent $p(x,t)$-Laplace operator given by $|\nabla u|^{p(x,t)-2} \nabla u$ and the domain $\Om$ could potentially have non-smooth boundary (see Subsection \ref{domain_structure} and Subsection \ref{operator_structure} for the precise assumptions).
% modelled after the well studied $p(x,t)$-Laplacian operator in a bounded domain $\Om\subset \RR^n$,  potentially with non-smooth boundary $\pa \Om$.

\emph{Weak solution} $u$ of \eqref{main} is in the space $ L^2(-T,T; L^2(\Om)) \cap L^{\pp}(-T,T; W_0^{1,\pp}(\Om))$ which allows one to use $u$ as a test function. But from the definition of weak solution, we see that the expression (see Definition \ref{very_weak_solution}) makes sense if we only assume $u \in L^2(-T,T; L^2(\Om)) \cap L^{s(\cdot)}(-T,T; W_0^{1,s(\cdot)}(\Om))$ for some $s(\cdot) > \pp-1$. But under this milder notion of solution called \emph{very weak solution}, we lose the ability to use $u$ as a test function.

In the constant exponent case $\pp \equiv p$, a suitable test function in the interior case was constructed in \cite{KL}  by  modifying $u$ on a \emph{bad} set and  partial interior higher integrability results were obtained in \cite{KL1,KL}. By suitably adapting the techniques from \cite{KL1}, in \cite{Par,Mik1}, the boundary higher integrability for \emph{weak} solutions were proved for domains satisfying a uniform thickness condition measured with respect to a suitable capacity. Under the same boundary regularity assumption from \cite{Mik1,Par}, in \cite{AB2}, the authors were able to extend the techniques of \cite{KL} to prove boundary higher integrability of \emph{very weak solutions}.

Unlike the constant exponent case, there is a genuine difficulty when trying to prove the higher integrability for gradients of \emph{very weak solutions} solving \eqref{main}. In the constant exponent case, the standard idea is to consider intrinsically scaled cylinders of the form
\begin{equation*}
Q_{R, \la^{2-p}R^2} \  \text{in the case}\  p \geq 2 \txt{or}
Q_{\la^{\frac{p-2}{2}}R, R^2} \  \text{in the case}\  p \leq 2,
\end{equation*}
for some $\la >0$ and radius $R>0$. This approach was used in \cite{bogelein2014very} where they considered cylinders of the form 
\begin{equation}
\label{shrink_one}
Q_{R, \la^{\frac{2-p(x,t)}{p(x,t)}}R^2}(x,t) \qquad   \text{in the case}\  2\leq \mathfrak{p}^- \leq \pp \leq \mathfrak{p}^+ < \infty,
\end{equation}
and in \cite{li2017very}, they considered cylinders of the form 
\begin{equation}
\label{shrink_two}
Q_{\la^{\frac{p(x,t)-2}{2p(x,t)}}R, R^2}(x,t) \qquad   \text{in the case}\  \frac{2n}{n+2} < \mathfrak{p}^- \leq \pp \leq \mathfrak{p}^+ \leq 2.
\end{equation}

One of the reason that there is a difference in the intrinsic scaling is because, the intrinsic cylinders in \eqref{shrink_one} and \eqref{shrink_two} shrink as $\la \nearrow \infty$.
If we were to consider the situation  where $\pp$ is allowed to cross the exponent $2$, i.e., suppose $\pp\leq 2$ in a subregion and $\pp \geq 2$ in the remaining part, then the techniques of \cite{bogelein2014very,li2017very} fail and thus, we are forced to come up with a different approach which should be capable of handling both the singular case and the degenerate case simultaneously. 
The main contribution of this paper is to  overcome the aforementioned difficulty and develop a \emph{unified intrinsic scaling} using which, we are able to prove \emph{interior and boundary} higher integrablity for \emph{very weak} solutions of \eqref{main} in the full range $\frac{2n}{n+2} < \mathfrak{p}^- \leq \pp \leq \mathfrak{p}^+ < \infty$.  \emph{As far as we know, this is the first technique that  can deal with both the singular case and degenerate case simultaneously.}

In our approach, we consider intrinsic cylinders of the form $Q_{\scalex{\la}{x,t}R,\scalet{\la}{z} R^2}(x,t)$ for a suitably chosen constant $d < \min \left\{ 1, \frac{2}{\mathfrak{p}^+} \right\}$.  It is easy to observe that by the restriction on the exponent $d$, the intrinsic cylinders $Q_{\scalex{\la}{x,t}R,\scalet{\la}{x,t} R^2}(x,t)$ shrink in both space and time simultaneously as $\la \nearrow \infty$. Two very important results that we need to prove are \emph{Vitaly-type}  and \emph{Whitney-type} covering lemmas over coverings based on these new intrinsically scaled cylinders (see Lemma \ref{lemma_vitali} and Lemma \ref{whitney_covering}). Once we have these covering lemmas, we can follow the strategy developed in \cite{KL,AB2} combined with the localization techniques developed in \cite{bogelein2014very,li2017very} and prove the desired boundary higher integrability result for \emph{very weak solutions}. {The localization techniques was first developed to prove higher integrability for \emph{weak solutions} in \cite{acerbi2004regularity} and suitably adapted to the setting of \emph{very weak solutions} in \cite{bogelein2014very,li2017very}.} The interior higher integrability estimates follow in an analogous way.

The plan of the paper is as follows: In Section \ref{section_preliminaries}, we describe all the structural assumptions and notations, in Section \ref{main_theorem}, we will state the main theorem that will be proved, in Section \ref{section_two_one}, we shall collect several useful lemmas,  in Section \ref{section_three}, we prove the required covering lemmas, in Section \ref{section_four}, we construct the Lipschitz function and prove its necessary properties, in Section \ref{section_five}, we use the previously obtained results to prove a Caccioppoli-type inequality, in Section \ref{section_six}, we prove a reverse H\"older type inequality and finally in Section \ref{section_seven}, we prove the main theorem.
 
 \subsection*{Acknowledgement} 
%  The first author would was partially supported by National Research Foundation of Korea grant NRF-2015R1A4A1041675, the second author was partially supported by National Research Foundation of Korea grant NRF-2017R1A2B2003877.

 The first author would like to thank Verena B\"ogelein, Ugo Gianazza, Juha Kinnunen and Qifan Li  for many helpful discussions. 

The first and second author would  like to thank the organisers of the conference \emph{Recent developments in Nonlinear Partial Differential Equations and Applications - NPDE2017} held at TIFR-CAM, Bangalore where part of this work was done.

 \section{Preliminaries}
 \label{section_preliminaries}
 In this section, we shall collect all the structure assumptions as well as recall several useful lemmas that are already available in existing literature. 
 
 \subsection{Parabolic metrics}
 Let us first collect a few metric's on $\RR^{n+1}$ that will be used throughout the paper:
 \begin{definition}
 \label{parabolic_metric}
 We define the parabolic metric $d_p$ on $\RR^{n+1}$ as follows: Let $z_1 = (x_1,t_1)$ and $z_2 = (x_2,t_2)$ be any two points on $\RR^{n+1}$, then 
 \begin{equation*}
 \label{par_met}
 d_p(z_1,z_2) := \max \mgh{|x_1-x_2|, \sqrt{|t_1-t_2|}}.
 \end{equation*}
 \end{definition}

 Since we  use intrinsically scaled cylinders where the scaling depends on the center of the cylinder, we will also need to consider the following localized parabolic metric:
 \begin{definition}
 \label{loc_parabolic_metric}
 Given any $z= (x,t) \in \RR^{n+1}$ and any $\mu > 0$, $d > 0$, we define the localized parabolic metric $d_z^{\mu,d}$ as follows: Let $z_1 = (x_1,t_1)$ and $z_2 = (x_2,t_2)$ be any two points on $\RR^{n+1}$, then 
 \begin{equation}
 \label{loc_par_met}
 d_z^{\mu,d}(z_1,z_2) := \max \mgh{\nscalex{\mu}{z}|x_1-x_2|, \sqrt{\nscalet{\mu}{z}|t_1-t_2|}}.
 \end{equation}
 \end{definition}

 \subsection{Structure of the domain}
 \label{domain_structure}
 Let us now introduce the assumption satisfied by the domain $\Om$.
 \begin{assumption}
 \label{uniform_measure}
 Let ${\Om}\subset \RR^n$ be a bounded domain which  satisfies a uniform measure density condition with constant $m_{\ve}>0$, i.e.,  for any $y \in \pa {\Om}$ and any $r > 0$, the following holds:
 \begin{equation*}
 |{\Om}^c \cap B_r(y) | \geq m_{\ve} |B_r(y)|.
 \end{equation*}
 \end{assumption}

 \begin{remark}
 In the constant exponent case, the domain $\Om^c$ was assumed to be uniformly thick where the thickness was measured with respect to a suitable capacity (see \cite{AB2}). Analogous to that theory, we can weaken Assumption \ref{uniform_measure} and instead assume that $\Om^c$ is also uniformly thick with the thickness measured with respect to a suitable capacity. While this improvement is important, we will refrain from pursuing that here and instead use a weaker condition as given in Assumption \ref{uniform_measure}. This will enable us to simplify many of the calculations while still retaining all the complexities of the problem.
 \end{remark}

 \subsection{Structure of the operator}
 \label{operator_structure}
 We shall now describe the assumptions on the nonlinear structure in \eqref{main}. We assume $\aa(x,t,\nabla u)$ is a Carath\'eodory function, i.e., we have $(x,t) \mapsto \aa(x,t,\zeta)$  is measurable for every $\zeta \in \RR^n$ and 
$\zeta \mapsto \aa(x,t,\zeta)$ is continuous for almost every  $(x,t) \in \Om \times (-T,T)$.

We further assume that for a.e. $(x,t) \in \Om \times (-T,T)$ and for any $\zeta \in \RR^n$, there exist two positive constants $\lamot$, such that  the following bounds are satisfied   by the nonlinear structure:
\begin{gather}
 \iprod{\aa(x,t,\zeta)}{\zeta} \geq \La_0 |\zeta |^{p(x,t)}  - h_1  \txt{and} |\aa(x,t,\zeta)| \leq \La_1  \lbr 1+ |\zeta|\rbr^{p(x,t)-1}, \label{abounded}
%  |\bb(x,t,\zeta)| \leq \La_2  |\zeta|^{p-1} + h_3 \label{bbounded}.
\end{gather}
where, $h_1\in \RR$ is a fixed constant.
 \emph{As in the constant exponent case (see \cite{KL,AB2}),  we do not make any  assumptions regarding the smoothness of $\aa(x,t,\zeta)$ with respect to $x,t,\zeta$}.

 \subsection{Structure of the variable exponent}

\begin{definition}
\label{definition_p_log}
 We say that a bounded measurable function $\pp : \RR^{n+1} \rightarrow \RR$  belongs to the $\log$-H\"older class $\logh$, if the following conditions are satisfied:
 \begin{itemize}
  \item There exist constants $\mathfrak{p}^-$ and $\mathfrak{p}^+$ such that $1< \mathfrak{p}^- \leq p(z) \leq \mathfrak{p}^+ < \infty$ for every $z \in \RR^{n+1}$.
  \item $ |p(z_1)  - p(z_2)| \leq \frac{L}{- \log |z_1-z_2|}$ holds  for every $ z_1,z_2 \in \RR^{n+1}$ with $ d_p(z_1,z_2) \leq \frac12 $ and for some  $L>0$.
 \end{itemize}
 
\end{definition}

\begin{remark}\label{remark_def_p_log} We remark that  $\pp$ is log-H\"{o}lder continuous in $\RR^{n+1}$ if and only if there is a non-decreasing continuous function ${\modp} : [0,\infty) \rightarrow [0,\infty)$ such that 
\begin{itemize}
 \item $\lim_{r\rightarrow 0} \modp(r) = 0$ and $|p(z_1)- p(z_2)| \leq \modp(d_p(z_1,z_2))$ for every $z_1,z_2 \in \RR^{n+1}$.
 \item $\modp(r) \log \lbr \frac{1}{r} \rbr \leq {L}$ holds for all $ 0< r \leq \frac12.$
\end{itemize}
The function $\modp(r)$ is called the modulus of continuity of the variable exponent $\pp$. 
 \end{remark}

  \subsection{Definition of very weak solution}

There is a well known difficulty in defining the notion of solution for \eqref{main} due to a lack of time derivative of $u$. To overcome this, one can either use Steklov average or convolution in time. In this paper, we shall use the former approach (see also \cite[Page 20, Eqn (2.5)]{DiB1} for further details).

Let us first define Steklov average as follows: let $h \in (0, 2T)$ be any positive number, then we define
\begin{equation}\label{stek1}
  u_{h}(\cdot,t) := \left\{ \begin{array}{ll}
                              \hint_t^{t+h} u(\cdot, \tau) \ d\tau \quad & t\in (-T,T-h), \\
                              0 & \text{else}.
                             \end{array}\right.
 \end{equation}
 
 \begin{definition}[Very weak solution] 
\label{very_weak_solution}
%  Define the function
 
%  and 
% \[
%  \def\arraystretch{2.2}
%   u_{\tilde{h}}(\cdot,t) := \left\{ \begin{array}{ll}
%                               \hint_{t-\tilde{h}}^{t} u(\cdot, \tau) \ d\tau & t\in (-T+\tilde{h},T) \\
%                               0 & \text{else}
%                              \end{array}\right.
%  \]
 Let $ \be \in (0,1)$ and $h \in (0,2T)$ be given and suppose that $\pp(1-\be) > 1$.  A very weak solution of \eqref{main} is a function $u \in L^2(-T,T; L^2(\Om)) \cap L^{\pp(1-\be)}(-T,T; W_0^{1,\pp(1-\be)}(\Om))$ such that
 \begin{equation}
 \label{def_weak_solution}
  \int_{\Om \times \{t\}} \frac{d [u]_{h}}{dt} \phi + \iprod{[\aa(x,t,\nabla u)]_{h}}{\nabla \phi} \, dx = 0 \qquad \text{for all} \ \, -T < t < T-h,
 \end{equation}
holds for any $\phi \in W_0^{1,\frac{\pp(1-\be)}{\pp(1-\be)-1}}(\Om) \cap L^{\infty}(\Om)$.
 
\end{definition}

 \subsection{Maximal Function}

For any $f \in L^1(\RR^{n+1})$, let us now define the strong maximal function in $\RR^{n+1}$ as follows:
\begin{equation}
 \label{not5.5}
 \mm(|f|)(x,t) := \sup_{\tQ \ni(x,t)} \fiint_{\tQ} |f(y,s)| \ dy \ ds
\end{equation}
where the supremum is  taken over all parabolic cylinders $\tQ_{a,b}$ with $a,b \in \RR^+$ such that $(x,t)\in \tQ_{a,b}$. An application of the Hardy-Littlewood maximal theorem in $x-$ and $t-$ directions shows that the Hardy-Littlewood maximal theorem still holds for this type of maximal function (see \cite[Lemma 7.9]{Gary} for details):
\begin{lemma}
\label{max_bnd}
 If $f \in L^1(\RR^{n+1})$, then for any $\al >0 $, there holds
 \[
  |\{ z \in \RR^{n+1} : \mm(|f|)(z) > \al\}| \leq \frac{5^{n+2}}{\al} \|f\|_{L^1(\RR^{n+1})}.
 \]

 and if $f \in L^{\vartheta}(\RR^{n+1})$ for some $1 < \vartheta \leq \infty$, then there holds
 \[
  \| \mm(|f|) \|_{L^{\vartheta}(\RR^{n+1})} \apprle_{(n,\vartheta)} \| f \|_{L^{\vartheta}(\RR^{n+1})}.
 \]

\end{lemma}
 
 \subsection{Notation}
We shall clarify all the notation that will be used in this paper.
\begin{enumerate}[(i)]
\item\label{not0} We shall fix a point $z_0 = (x_0,t_0) \in \pa \Om \times (-T,T)$. 
 \item\label{not1} We shall use $\nabla$ to denote derivatives with respect the space variable $x$.
\item\label{not2} We shall sometimes alternate between using $\ddt{f}$, $\pa_t f$ and $f'$ to denote the time derivative of a function $f$.
 \item\label{not3} We shall use $D$ to denote the derivative with respect to both the space variable $x$ and time variable $t$ in $\RR^{n+1}$. 
 \item\label{not3-1} In what follows, we shall always assume the following bounds are applicable for the variable exponent $\pp$:
 \begin{equation*}
 \label{bound_pp}
 \frac{2n}{n+2} < \mathfrak{p}^- \leq \pp \leq \mathfrak{p}^+ < \infty.
 \end{equation*}

 \item\label{not4}  Let $z_0 = (x_0,t_0) \in \RR^{n+1}$ be a point, $d >0$ be a fixed constant and $\rho, s >0$ be two given parameters and let $\la \in (0,\infty)$. We shall use the following notations:
 \begin{equation*}\label{notation_space_time}
\def\arraystretch{1.5}
 \begin{array}{ll}
 I_s(t_0) := (t_0 - s^2, t_0+s^2) \subset \RR,
& \quad Q_{\rho,s}(z_0) := B_{\rho}(x_0) \times I_{s}(t_0) \subset \RR^{n+1},\\ 
 I_s^{\la}(t_0) := (t_0 - \scalet{\la}{z_0}s^2, t_0+\scalet{\la}{z_0}s^2) \subset \RR,
& \quad B_{\rho}^{\la}(x_0) := B_{\scalex{\la}{z_0}\rho}(x_0) \subset \RR^n, \\
 \mcc^{\la}_{\rho}(x_0) := \lbr \Om \cap B_{\scalet{\la}{z_0}\rho}(x_0) \rbr \times \RR \subset \RR^{n+1}, & \quad Q^{\la}_{\rho,s}(z_0) := B_{\scalex{\la}{z_0}\rho}(x_0) \times I^{\la}_{s}(t_0) \subset \RR^{n+1},\\  
 Q_{\rho}^{\la} (z_0) := Q_{\rho, \rho^2}^{\la} (z_0). & \quad
%  \al Q^{\la}_{\rho,s}(z_0) := B_{\al \rho}(x_0) \times \tm_{\al^2s}(t_0)  \subset \RR^{n+1},
%  &\quad \mch_s(t_0) := \RR^n \times \tm_s(t_0) \subset \RR^{n+1},\\ 
%   \mch_s^{\la}(t_0) := \RR^n \times I_{\la^{-p(z_0)+d} s}(t_0) \subset \RR^{n+1},
%  & \quad\mcc^{\la}_{\rho}(x_0) := \Om \cap B_{\la^{-1+\frac{d}{2}}\rho}(x_0) \times \RR \subset \RR^{n+1},\\
% \avgs{Q_{\rho,s}(z_0)}{f} = \avgf{Q_{\rho,s}(z_0)}{f} \ dz \\
% \Om_{\rho,s}(z_0) := \Om \cap B_{\rho}(x_0) \times \tm_s(t_0) \subset \RR^{n+1}, 
% & \quad\Om_{\rho}(x_0) := \Om \cap B_{\rho}(x_0) \subset \RR^{n}.
 \end{array}
\end{equation*}

\emph{Note that in the above notation, we have dropped writing the exponent $d$ because this constant will be universally fixed in \eqref{exp_d}.}

\item\label{not10} Once we have fixed $d$ in \eqref{exp_d}, we will use the short form to denote $d_z^{\mu} := d_z^{\mu,d}$.

\item\label{not5} We shall use $\int$ to denote the integral with respect to either space variable or time variable and use $\iint$ to denote the integral with respect to both space and time variables simultaneously. 

Analogously, we will use $\fint$ and $\fiint$ to denote the average integrals as defined below: for any set $A \times B \subset \RR^n \times \RR$, we define
\begin{gather*}
% \begin{array}{c}
\avgs{f}{A}:= \fint_A f(x) \ dx = \frac{1}{|A|} \int_A f(x) \ dx,\\
\avgs{f}{A\times B}:=\fiint_{A\times B} f(x,t) \ dx \ dt = \frac{1}{|A\times B|} \iint_{A\times B} f(x,t) \ dx \ dt.
% \end{array}
\end{gather*}

\item\label{not6} Given any positive function $\mu$, we shall denote $\avgs{f}{\mu} := \int f\frac{\mu}{\|\mu\|_{L^1}}dm$ where the domain of integration is the domain of definition of $\mu$ and $dm$ denotes the associated measure.

% {\color{red}\item\label{not5.5} We will use $\mm$ to denote the usual Hardy-Littlewood Maximal function defined for any $f \in L^1(\RR^{n+1})$ by
% \[
% \mm(|f|)(x,t) := \sup_{\rho>0,\theta>0} \fiint_{Q_{\rho,\theta}(x,t)} |f(\mfz)| d\mfz.
% \]
% }
\item\label{not8} In what follows, $m_{\ve}$ will denote the constants arising from the assumption that  $\Om^c$ satisfies a uniform measure density condition (see Assumption \ref{uniform_measure}).

\item We will use the notation $\plog$ to denote the constants $\mathfrak{p}^+,\mathfrak{p}^-$ and the constant coming from the $\log$-H\"older continuity of the variable exponent.

\item\label{not9} We will obtain an $\be_0 = \be_0 (n,\plog,\lamot,,m_{\ve}) \in (0,1)$ in Definition \ref{restriction_be_0} such that all the estimates hold for any $\be \in (0,\be_0)$.

\item\label{not11} We will use the notation $\apprle_{(a,b,\ldots)}$ to denote an inequality with a constant depending on $a,b,\ldots$.

\item We will use $\hat{r},\hat{Q},\ldots$ to denotes objects scaled by the constant $\hat{c}$ given by \descref{W4}{W4}.

\item We use $\mm$ to denote the Hardy-Littlewood Maximal function defined in \eqref{not5.5} and $\tilde{\mm}$ to denote the truncated maximal function as given in \eqref{def_mtilde}.
\end{enumerate}

 \section{Main Theorem}
 \label{main_theorem}
 
We now state the main theorem that will be proved. 
\begin{theorem}
\label{thm_main}
Let $\pp : \RR^{n+1} \rightarrow \RR$ be a variable exponent as in Definition \ref{definition_p_log} and $\Om$ be a bounded domain satisfying Assumption \ref{uniform_measure}.
Then there exists a constant $\be_0 = \be_0(n,\plog ,\La_0, \La_1,m_e) \in (0,1)$ such that the following holds: for any $\be \in (0,\be_0)$ and \emph{any very weak solution} $u \in L^2(-T,T; L^2(\Om)) \cap L^{\pp(1-\be)}(-T,T; W_0^{1,\pp(1-\be)}(\Om))$ of \eqref{main}, we have the improved integrability
$$
|\nabla u| \in L^{\pp}(\overline{\Om} \times (-T,T)).
$$
Moreover, for $\bM \geq 1$ there exists a radius $\rho_0 = \rho_0(n,\La_1,\bM)>0$ such that if
$$
\iint_{\Om_T} \lbr|u| +  |\nabla u|+1\rbr^{p(z)(1-\be)}  \, dz \leq \bM,
$$
then for any parabolic cylinder $Q_{2r}(\mathfrak{z}_0) \subset \RR^{n}\times (-T,T)$  with $\mfz_0 \in \pa \Om \times (-T,T)$ and $r \in (0,\rho_0]$, there holds
\begin{equation*}
\label{mainthm_estimate}
\fiint_{Q_r(\mathfrak{z}_0)} |\nabla u|^{p(z)} \, dz \apprle_{(n, \plog,\La_0, \La_1,m_e)} \left[ \left( \fiint_{Q_{2r}(\mathfrak{z}_0)} |\nabla u|^{p(z)(1-\be)} \, dz \right)^{1+\frac{\be}{-\frac{n}{p(\mathfrak{z}_0)}+\frac{(n+2)d}{2}-\be}} + 1 \right],
\end{equation*}
where $d \in (0,1)$ is the constant chosen in \eqref{exp_d}.
\end{theorem}

\begin{remark} The above theorem is stated only at the boundary, i.e., when $\mfz_0 \in \pa \Om \times (-T,T)$. The interior higher integrability result for \emph{very weak solution} follows similarly by combining our unified intrinsic scaling with the techniques of \cite{bogelein2014very,li2017very}.\end{remark}

 \begin{remark}
 In Theorem \ref{thm_main}, we extend the solution by $0$ outside $\Om$, i.e., we set 
 \[
 u = 0 \txt{on} \Om^c \times (-T,T).
 \]

 \end{remark}

 \section{Some useful lemmas}
 \label{section_two_one}
 
 Let us first recall a well known parabolic type Poincar\'e's inequality (see \cite[Lemma 2.13]{AB2} for a proof):
  \begin{lemma}
 \label{lemma_crucial_1}
 Let $f \in L^{\vt} (-T,T; W^{1,\vt}(\Om))$ with $\vt \in [1,\infty)$  and suppose that $B_{r} \Subset \Om$ be compactly contained ball of radius $r>0$. Let $I \subset (-T,T)$ be a time interval  and $\rho(x,t) \in L^1(B_r \times I)$ be any positive function such that $\|\rho\|_{L^{\infty}(B_r\times I)} \apprle_n \frac{\|\rho\|_{L^1(B_r\times I)}}{|B_r\times I|} $ and $\mu(x) \in C_c^{\infty}(B_r)$ such that $\int_{B_r} \mu(x) \ dx = 1$ with $|\mu| \leq \frac{C{(n)}}{r^n}$ and $|\nabla \mu| \leq  \frac{C(n)}{r^{n+1}}$. Then there holds
 \begin{equation*}
 \begin{array}{ll}
  \fiint_{B_r \times I} \left|\frac{f - \avgs{f}{\rho}}{r}\right|^{\vt} \ dz & \apprle_{(n,\vt)} \fiint_{B_r \times I} |\nabla f|^{\vt} \ dz + \sup_{t_1,t_2 \in I} \left| \frac{\avgs{f}{\mu}(t_2) - \avgs{f}{\mu}(t_1)}{r} \right|^{\vt},
  \end{array}
 \end{equation*}
where $\avgs{f}{\rho}:= \iint_{B_r\times I} f(z) \frac{\rho(z)}{\|\rho\|_{L^1(B_r\times I)}} \ dz $ and $\avgs{f}{\mu}(t_i) := \int_{B_r} f(x,t_i) \mu(x) \ dx$ for $i = 1,2$. 
\end{lemma}

The following crucial lemma will be used throughout the paper:
\begin{lemma}
\label{lemma_crucial_2}
 Suppose that $u \in L^2(-T,T;L^2(\Om)) \cap \vwspace$ is a very weak solution of \eqref{main} for some $0 \leq \be \leq \min\{1,\mathfrak{p}^--1\}$. Let $\mathcal{B} \subset \Om$ be a compactly contained region and $(t_1,t_2) \subset (-T, T-h)$ for some $h \in (0,T)$ be a time interval. Let $\phi(x) \in C_c^{\infty}(\mathcal{B})$, $\varphi(t) \in C_c^{\infty}(t_1,t_2)$ be two non-negative functions and $[u]_h$ be the Steklov average as defined in \eqref{stek1}. Then   the following estimate holds:
 \begin{equation*}
  \label{lemma_crucial_2_est}
  \begin{array}{ll}
  |\avgs{{[u]_h\varphi}}{\phi} (t_2) - \avgs{{[u]_h\varphi}}{\phi}(t_1)| & \leq \La_1 \|\nabla \phi\|_{L^{\infty}{(\mathcal{B})}} \|\varphi\|_{L^{\infty}(t_1,t_2)} \iint_{\mathcal{B} \times (t_1,t_2)} {[(1+|\nabla u|)^{p(z)-1} ]_h} \, dz \\
%   & \qquad + \La_2\|\phi\|_{L^{\infty}{(\mathcal{B})}} \|\varphi\|_{L^{\infty}(t_1,t_2)} \iint_{\mathcal{B} \times (t_1,t_2)} {[|\nabla u|^{p-1} +|h_3| ]_h} \, dz \\
   & \qquad +  \|\phi\|_{L^{\infty}{(\mathcal{B})}} \|\varphi'\|_{L^{\infty}(t_1,t_2)} \iint_{\mathcal{B} \times (t_1,t_2)} |[u]_h| \, dz.
  \end{array}
 \end{equation*}
\end{lemma}
\begin{proof}
 Let us use $\phi(x)\varphi(t)$ as a test function in \eqref{def_weak_solution} to get 
 \[
  \int_{\Om\times\{t\}} \lbr[[]\ddt{[u]_h} (x,t) \phi(x)\varphi(t)  +  \iprod{[\aa(x,t,\nabla u)]_h}{\nabla \phi}(x,t)\varphi(t) \rbr[]]\, dx =0.
 \]
 Using the Fundamental theorem of calculus, we get
 \begin{equation*}
  \begin{array}{ll}
   |\avgs{{[u]_h\varphi}}{\phi} (t_2) - \avgs{{[u]_h\varphi}}{\phi}(t_1)|%& = \left|\int_{\mathcal{B}_r} [u]_h(x,t_2) \phi(x) \ dx - \int_{\mathcal{B}_r} [u]_h(x,t_1) \phi(x) \, dx \right|\\
%    & = \left|\int_{t_1}^{t_2} \frac{d}{dt} \lbr \int_{\mathcal{B}_r} [u]_h(x,t) \phi(x) \ dx \rbr \, dt\right|\\
%    &= \left|\iint_{\mathcal{B} \times (t_1,t_2)} \frac{d}{dt} \lbr[(] [u]_h(x,t) \phi(x)\varphi(t) \rbr[)] \, dz \right|\\
   & \leq \left|\iint_{\mathcal{B} \times (t_1,t_2)} \iprod{[\aa(x,t,\nabla u)]_h}{\nabla \phi}(x,t) \varphi(t)\, dz \right| \\
%     & \qquad + \left|\iint_{\mathcal{B} \times (t_1,t_2)} \bb(x,t,\nabla u)\phi(x,t) \varphi(t)\, dz \right| \\
   & \qquad + \left| \iint_{\mathcal{B} \times (t_1,t_2)} [u]_h(x,t) \phi(x) \ddt{\varphi(t)} \, dz \right|  \\
   & \overset{\redlabel{lemma3.4.1}{a}}{\leq} \La_1\|\nabla \phi\|_{L^{\infty}{(\mathcal{B})}} \|\varphi\|_{L^{\infty}(t_1,t_2)} \iint_{\mathcal{B} \times (t_1,t_2)} {[(1+|\nabla u|)^{p(z)-1}]_h} \, dz \\
%    & \qquad + \La_2\|\phi\|_{L^{\infty}{(\mathcal{B})}} \|\varphi\|_{L^{\infty}(t_1,t_2)} \iint_{\mathcal{B} \times (t_1,t_2)} {[|\nabla u|^{p-1} +|h_3| ]_h} \, dz \\
   & \qquad + \|\phi\|_{L^{\infty}{(\mathcal{B})}} \|\varphi'\|_{L^{\infty}(t_1,t_2)} \iint_{\mathcal{B} \times (t_1,t_2)} |[u]_h| \, dz.
  \end{array}
 \end{equation*}
 To obtain \redref{lemma3.4.1}{a} above, we made use of \eqref{abounded}. This completes the proof.
\end{proof}

We next state Gagliardo-Nirenberg's inequality (see \cite[Lemma 2.6]{AB2} for the proof).
\begin{lemma}
\label{lemma_crucial_3}
Let $B_{\rho}(x_0) \subset \RR^n$ with $0 < \rho \leq 1$, and let $1 \leq \sig, \ga_1, \ga_2 < \infty$, $\theta \in (0,1)$.
Suppose that
\begin{equation*}
-\frac{n}{\sig} \leq \theta \left( 1-\frac{n}{\ga_1} \right) - (1-\theta)\frac{n}{\ga_2}.
\end{equation*}
For any $v \in W^{1,\ga_1}(B_{\rho}(x_0))$ with
$$\left| \{x \in B_{\rho}(x_0) : v(x)=0 \} \right| \geq c_0 |B_{\rho}(x_0)|$$
for some universal $c_0>0$, then the following estimate holds
\begin{equation*}
\label{lemma_crucial_3_est}
\fint_{B_{\rho}(x_0)} \left| \frac{v}{\rho} \right|^{\sig} dx \apprle_{(n,\sig,c_0)} \left( \fint_{B_{\rho}(x_0)} |\nabla v|^{\ga_1} \, dx \right)^{\frac{\theta \sig}{\ga_1}} \left( \fint_{B_{\rho}(x_0)} \left| \frac{v}{\rho} \right|^{\ga_2} dx \right)^{\frac{(1-\theta) \sig}{\ga_2}}.
\end{equation*}
\end{lemma}

We will also need the following  Poincar\'e's inequality proved in \cite[Theorem 4.13]{adimurthi2017sharp}.
\begin{theorem}
 \label{measure_density_poincare}
 Let $\sss \in \logh$ and let $\bM_p\geq 1$, $\varepsilon \in (0,1)$ be given constants. Define  ${\bf R}_p := \min \left\{\frac{1}{2\bM_p}, \frac1{|B_1|^{\frac{1}{n}}}, \frac12 \right\}$.  For any $\phi \in W^{1,\pp}(B_{2r})$ with $2r < {\bf R}_p$ satisfying 
 \begin{gather*}
 |\{ N(\phi)\}| := |\{ x \in B_r : \phi(x) =0\}|> \varepsilon |B_r|\txt{and} \int_{B_{2r}} |\nabla \phi (x)|^{s(x)} \ dx + 1 \leq \bM_p, 
 \end{gather*}
 then there holds
 \[
  \int_{B_r} \lbr \frac{|\phi|}{\diam(B_{r})}\rbr^{s(x)} \ dx \apprle_{(n,\slog)}  \int_{B_{2r}} |\nabla \phi(x)|^{s(x)} \ dx + |B_{r}|.
 \]
% with $C = C({\plog},n,\varepsilon)$. % provided $\int_{B_{2R_0}} |f(x)|^{p(x)} w(x) \ dx +1\leq M_1$. 
\end{theorem}

% We note that Theorem \ref{measure_density_poincare} is slightly different than the one proved in \cite[Theorem 4.13]{adimurthi2017sharp}. In order to obtain this improvement where the domain $B_r$ is the same on both sides of the inequality, we can repeat the arguments in the proof of \cite[Theorem 4.13]{adimurthi2017sharp} and combine them with the technical Lemma from \cite[Lemma 3.4]{han2011elliptic}.

We conclude this section with a standard iteration lemma, see \cite[Lemma 6.1]{Giusti}.
\begin{lemma}
\label{iter_lemma}
Let $0< r< R<\infty$ be given and let $h : [r,R] \to \RR$ be a non-negative and bounded function. Furthermore, let $\theta \in (0,1)$ and $A,B,\gamma_1,\gamma_2 \geq 0$ be fixed constants and 
suppose that
$$
h(\rho_1) \leq \theta h(\rho_2) + \frac{A}{(\rho_2-\rho_1)^{\gamma_1}} + \frac{B}{(\rho_2-\rho_1)^{\gamma_2}},
$$
holds for all $r \leq \rho_1 < \rho_2 \leq R$, then the following conclusion holds:
% Then there exists a positive constant $c = c$ such that
$$
h(r) \apprle_{(\theta,\gamma_1,\gamma_2)} \frac{A}{(R-r)^{\gamma_1}} + \frac{B}{(R-r)^{\gamma_2}}.
$$
\end{lemma}

% \hrule \hrule \hrule \hrule

 \section{Covering lemmas}
 \label{section_three}
 Henceforth, we shall assume there exists a \emph{very weak} solution $u \in  L^{\pp(1-\be)} ( -T,T; W_0^{\pp(1-\be)}(\Om))$ for some $\be \in (0,\be_0)$ with $\be_0 = \be_0(n,\lamot,\plog,m_e)$ to be suitably chosen (see Definition \ref{restriction_be_0}). For this choice of $\be$, let us now fix the following universal constant
\begin{equation}
\label{def_M_0}
\bM_0 := \iint_{\Om_T} \lbr|u| +  |\nabla u|+1\rbr^{p(z)(1-\be)}  \, dz.
\end{equation}

\begin{definition}
\label{restriction_rho_0}
We are going to fix a radius $\rho_0$ which satisfies the following bounds:
\begin{description}[leftmargin=\parindent,labelindent=\parindent]
\descitem{Restriction (i): }{rnone} We will assume that $\rho_0 \leq \frac{1}{1024\bM_0}$, where $\bM_0$ is from \eqref{def_M_0}. This restriction also ensures that we can apply \cite[Theorem 4.13]{adimurthi2017sharp}.
\descitem{Restriction (ii):}{rtwo} We will assume that 
\begin{equation}
\label{def_de_small}
p^+_{Q_{128\rho_0}} - p^-_{Q_{128\rho_0}} \leq \modp(512 \rho_0) \leq \frac12  \min \left\{\frac{1}{n+2}, \frac{1}{4} \right\}.
\end{equation}
\descitem{Restriction (iii):}{rthree} We will assume $\rho_0$ small such that 
\[
\modp(\rho_0) \leq \frac{L}{\log \lbr \frac{1}{\rho_0} \rbr} \leq \frac{1}{n}.
\]

\end{description}
\end{definition}

\begin{remark}
Let us consider the fixed  base cylinder of the form:
\begin{equation}
\label{def_q_0}
Q_0:= Q_{\rho}^{\al_0} (z_0) \qquad  \text{for some fixed} \ \, \al_0 \geq 1 \ \, \text{and} \ \, 128 \rho \leq \rho_0.
\end{equation}
\end{remark}

We now consider the following universal cylinder for $\rho_0$ as defined in Definition \ref{restriction_rho_0}:
\begin{equation*}
\label{large_cylinder}
\mbfq:= Q_{128\rho_0}(z_0) := B_{128\rho_0}(x_0) \times \lbr t_0 - (128\rho_0)^2,t_0 + (128\rho_0)^2\rbr.
\end{equation*}
% where $d$ is from \eqref{exp_d}. 

\begin{remark}
\label{p_bound}
We shall henceforth denote $p^- := p^-_{\mbfq}$ and $p^+:= p^+_{\mbfq}$. This is possible since we will show that all the cylinders considered henceforth will be contained inside $\mbfq$.
\end{remark}

\begin{definition}
Let us choose an exponent $d$ satisfying the bound
\begin{equation}
\label{exp_d}
\frac{2n}{(n+2)p^-} < d < \min\left\{ \frac{2}{p^+}, 1\right\}.
\end{equation}
This choice is possible since we see from \eqref{def_de_small} that  $1-\frac{p^-}{p^+} <\frac{1}{(n+2)p^+} < \frac{1}{n+2}< \frac{2}{n+2}$. 
\end{definition}

In the case of constant exponent, i.e., $\pp \equiv p$,  we can take $d = \min\left\{ \frac{2}{p}, 1\right\}$ and in each of those choices, we recover the usual parabolic scaling considered for singular and degenerate equations. In a sense, our choice of $d$ helps us to obtain an intermediate scaling which lies between the standard intrinsic scaling for singular and degenerate equations. An important fact to note is that all the subsequent estimates obtained will be independent of $d$.

\begin{definition}
\label{restriction_be_0}
We are going to collect all the restrictions on $\be_0$:
\begin{description}[leftmargin=\parindent,labelindent=\parindent]
\descitem{Restriction (i):}{beone} We will take $0 < \be_0 < 1 - \frac{1}{q} \leq \frac{1}{p^+}$ where $q$ is a constant that satisfies $1 < q < \frac{p^-}{p^+-1}$ (see Remark \ref{p_bound} for definition of $p^+$ and $p^-$). This choice is possible due to \eqref{def_de_small}. 
\descitem{Restriction (ii):}{btwo} We will further assume $2\be_0 \leq \frac{d(n+2)}{2} - \frac{n}{p^-}$.
\descitem{Restriction (iii):}{bthree} We will assume $\be_0 \leq \frac12$ holds.
\descitem{Restriction (iv):}{four} We will assume $\mathfrak{p}^- - \be_0 > \frac{2n}{n+2}$ holds.
\end{description}
\end{definition}

\begin{definition}
\label{first_bound_al_0}
Throughout this section, we shall assume  the following holds:
\begin{equation}
\label{est_bnd_al_0}
{\al_0^{1-\be}} \leq  \fiint_{Q_{\rho}^{\al_0}} \lbr |\nabla u| + 1\rbr^{p(z)(1-\be)} \, dz  \txt{and} \fiint_{Q_{16\rho}^{\al_0}} \lbr |\nabla u| + 1\rbr^{p(z)(1-\be)} \, dz \leq  \al_0^{1-\be}.
\end{equation}

\end{definition}

\begin{definition}
\label{cylinders}
Let us now fix two radii $\rho_a$ and $\rho_b$ such that 
\begin{equation}
\label{rho_def}
\rho \leq \rho_a < \rho_1:=\rho_a + \frac13 (\rho_b -\rho_a) < \rho_2:=\rho_a + \frac23 (\rho_b -\rho_a) < \rho_b \leq 16\rho,
\end{equation}
and consider the following chain of cylinders:
\[
Q_{\rho}^{\al_0} (z_0) \subset Q_{\rho_a}^{\al_0} (z_0) \subset Q_{\rho_1}^{\al_0} (z_0) \subset Q_{\rho_2}^{\al_0} (z_0) \subset Q_{\rho_b}^{\al_0} (z_0) \subset Q_{16\rho}^{\al_0} (z_0)\subset Q_{32\rho}^{\al_0} (z_0).
\]
Note that since $128 \rho \leq \rho_0$ and $\al_0 \geq 1$, it is clear that $Q_{32\rho}^{\al_0}(z_0) \subset \mbfq$.
\end{definition}

Consider the following cut-off functions:
\begin{equation}
\label{cut_off_function}
\begin{array}{cccc}
\eta=\eta(x) \in C_c^{\infty}\lbr B_{\rho_2}^{\al_0}(x_0)\rbr, & \eta \equiv 1 \ \, \text{on} \ B_{\rho_1}^{\al_0}(x_0), & 0 \leq \eta \leq 1, & |\nabla \eta| \leq \frac{c}{\scalex{\al_0}{z_0} (\rho_b-\rho_a)}, \\
 \zeta=\zeta(t) \in C_c^{\infty}\lbr I_{\rho_2}^{\al_0}(t_0)\rbr, & \zeta \equiv 1 \ \, \text{on} \ I_{\rho_1}^{\al_0}(t_0), & 0 \leq \zeta \leq 1, & |\pa_t \zeta| \leq \frac{c}{\scalet{\al_0}{z_0}(\rho_b^2-\rho_a^2)}.
\end{array}
\end{equation}

Let us now define the truncated solution to be 
\begin{equation*}
\label{trunc_solution}
\tuh(x,t) := [u]_h(x,t) \eta(x) \zeta(t),
\end{equation*}
where $[u]_h$ denotes the Steklov average as defined in \eqref{stek1}. The boundary condition $u = 0 $ on $\pa_p\Om_T$ implies that the truncated function has support $\spt(\tuh) \subset (\Om \cap B_{\rho_2}^{\al_0}(x_0)) \times \RR$.  
% While $\tuh$ has the right support, it does not posess sufficient regularity to be used as a test function. Hence we shall use the method of Lipschitz truncation.

From the first restriction in Definition \ref{restriction_be_0} which gives the choice of exponent $q$, let us now define the following function for some $\be \in (0,\be_0)$:
\begin{equation*}
\label{g_def}
g(z): = % \max \left\{
% \begin{array}{l}
\mm\lbr \left[|\nabla u| + 1 +\frac{|u|}{\scalex{\al_0}{z_0}\rho}\right]^{\frac{\pp}{q}} \lsb{\chi}{\qfur} \rbr^{q(1-\be)}(z).
% \mm\lbr \left[|\nabla \tu|+1\right]^{\frac{\pp}{q}} \lsb{\chi}{\qfur} \rbr^{q(1-\be)}(z),\\ 
% ,\mm\lbr \left[\frac{|u|}{\scalex{\al_0}{z_0}\rho}\right]^{\frac{\pp}{q}} \lsb{\chi}{\qfur} \rbr^{q(1-\be)}(z).
% \end{array}\right\}
\end{equation*}
Recall that $\mm$ is the Hardy-Littlewood Maximal function defined in \eqref{not5.5}.

Given the function $g(z)$, we define the good set for some $\la \geq c_e \al_0$ (here $c_e$ is a universal constant to be fixed in Lemma \ref{fix_ce}):
\begin{equation}
\label{elam}
\elam := \{ z \in \RR^{n+1} : g(z) \leq \la^{1-\be} \}.
\end{equation}

We shall use the following notation $\elam(t) := \{ (\tx,\tlt) \in \elam: \tlt=t\}$ to denote a time slice of the \emph{good set}. 

\begin{lemma}\label{fix_ce}
There exists a constant $c_e= c_e(\plog,n)$ such that if $\la \geq c_e \al_0$ for $\al_0$ satisfying \eqref{est_bnd_al_0}, then $\elam \neq \emptyset$. 
\end{lemma}
\begin{proof}
To show this, we proceed by contradiction.
Suppose $\elam = \emptyset$ for some $\la >0$, then we must have
\begin{equation}
\label{bond-1}\begin{array}{rcl}
\la^{1-\be} |\qfur| & \leq &\iint_{\qfur} g(z) \, dz \\
 & \overset{\text{Lemma \ref{max_bnd}}}{\apprle}& \iint_{\qfur} \left[|\nabla u|+1\right]^{\pp(1-\be)}  + \left[\frac{|u|}{\scalex{\al_0}{z_0}\rho}\right]^{\pp(1-\be)} dz \\
% \nonumber & \apprle \iint_{\qfur} \left[|\nabla u|+1\right]^{\pp(1-\be)}  + \left[\frac{|u|}{\scalex{\al_0}{z_0}\rho}\right]^{\pp(1-\be)} dz \\
 & =& J_1 + J_2.
\end{array}\end{equation}
Let us now estimate each of the above terms:
\begin{description}
\item[Estimate for $J_1$:] Using \eqref{est_bnd_al_0}, we get
\begin{equation}
\label{bond-J1}
\iint_{\qfur} \left[|\nabla u|+1\right]^{\pp(1-\be)} \, dz \leq \iint_{\qfve} \left[|\nabla u|+1\right]^{\pp(1-\be)} \, dz \leq \al_0^{1-\be} |\qfve|.
\end{equation}

% \item[Estimate for $J_2$:] From Definition \ref{restriction_rho_0}, we see that \cite[Theorem 4.13]{adimurthi2017sharp} becomes applicable (note that $\rho_2 - \rho_1 \leq 16\rho$), thus we get
% \begin{equation}
% \label{bond-J2}
% \begin{array}{ll}
% \iint_{\qfur} \left[\frac{|u|\lsb{\chi}{\qthrs\setminus\qtwos}}{\scalex{\al_0}{z_0}(\rho_2-\rho_1)}\right]^{\pp(1-\be)}  & \leq \lbr \frac{16\rho}{\rho_2-\rho_1} \rbr^{p^+(1-\be)}\iint_{\qfve}\left[\frac{|u|}{\scalex{\al_0}{z_0}16\rho}\right]^{\pp(1-\be)} \\
% & \apprle \lbr \frac{16\rho}{\rho_2-\rho_1} \rbr^{p^+(1-\be)} \iint_{\qfve} \lbr |\nabla u|+1\rbr^{\pp(1-\be)} \ dz \\
% & \overset{\eqref{est_bnd_al_0}}{\apprle} \lbr \frac{16\rho}{\rho_2-\rho_1} \rbr^{p^+(1-\be)}\al_0^{1-\be}|\qfve|.
% \end{array}
% \end{equation}

\item[Estimate for $J_2$:] We can estimate $J_2$ as follows:
\begin{equation}
\label{bond-J3} 
\begin{array}{rcl}
 \iint_{\qfur} \left[\frac{|u|}{\scalex{\al_0}{z_0}\rho}\right]^{\pp(1-\be)} dz & \leq & \iint_{\qfve}\left[\frac{|u|}{\scalex{\al_0}{z_0}\rho}\right]^{\pp(1-\be)} dz \\
 & \overset{\text{Theorem \ref{measure_density_poincare}}}{\apprle} & \iint_{\qfve} \left[|\nabla u| + 1\right]^{\pp(1-\be)} dz \\
& \overset{\eqref{est_bnd_al_0}}{\apprle} & \al_0^{1-\be} |\qfve|. 
\end{array}
\end{equation}
\end{description}

Combining the estimates \eqref{bond-J1}, \eqref{bond-J3} and \eqref{bond-1}, we see that 
\begin{equation}
\label{bond-2}
\la^{1-\be} |\qnot| \leq \la^{1-\be} |\qfur| \apprle_{(n,\plog,m_e)} \al_0^{1-\be} |\qfve|.
\end{equation}
Hence, if we denote 
\begin{equation*}\label{choice_ce}c_e := \lbr  C_{(n,\plog,m_e)} \frac{|\qfve|}{|\qnot|} \rbr^{\frac{1}{1-\be}}, \end{equation*} where $C_{(n,\plog,m_e)}$ is the constant in \eqref{bond-2}, then for all $\la > c_e \al_0$, the estimate \eqref{bond-2} fails which shows that $\elam \neq \emptyset$ whenever $\la > c_e \al_0$. Note that since $\be \leq \be_0 \leq \frac12$ (see Definition \ref{restriction_be_0}) , the constant $c_e$ is independent of $\be$. This proves the lemma. 

% \textcolor{blue}{\color{blue}
% In the case, we do not have the bound for $J_2$, we will then get 
% \begin{equation}
% \la^{1-\be} |\qnot|  \leq C_{(n,\plog,m_e)} \al_0^{1-\be} |\qfve| + |\qfve| \fiint_{\qfur} \left[\frac{|u|\lsb{\chi}{\qthrs\setminus\qtwos}}{\scalex{\al_0}{z_0}(\rho_2-\rho_1)}\right]^{\pp(1-\be)}
% \end{equation}
% \color{blue} Let $\tilde{\la} := \al_0 + \lbr \fiint_{\qfur} \left[\frac{|u|\lsb{\chi}{\qthrs\setminus\qtwos}}{\scalex{\al_0}{z_0}(\rho_2-\rho_1)}\right]^{\pp(1-\be)}\rbr^{\frac{1}{1-\be}}$, then 
% $\la^{1-\be} |\qnot| \leq C_{(n,\plog,m_e)} \tilde{\la}^{1-\be} |\qfve|$ and the lemma can hold again. 
% MAYBE, WE MAY NOT EVEN NEED $J_2$ TERM AT ALL??
% }
\end{proof}

\begin{lemma}
With $\al_0$ as in \eqref{est_bnd_al_0} and any $\rho$ satisfying $128\rho \leq \rho_0$, we have 
\begin{equation}
\label{bnd_al_0_rho}
\al_0^{p^+_{\qfve} - p^-_{\qfve}} \leq C_{(n,\plog,m_e)} \txt{and} \rho^{\pm \abs{p^+_{\qfve} - p^-_{\qfve}}} \leq C_{(n,\plog,m_e)}. 
\end{equation}
\end{lemma}
\begin{proof} Let us prove each of the bounds as follows:
\begin{description}
\item[Bound for $\rho$:] Since $\pp \in \plog$, we have from \eqref{parabolic_metric} that
\begin{align}
\nonumber p^+_{\qfve} - p^-_{\qfve} & \leq \modp \lbr \max\left\{ \scalex{\al_0}{z_0}16\rho, \sqrt{\scalet{\al_0}{z_0} (16\rho)^2}  \right\} \rbr \\
\label{2.2.5} & \leq \modp (32\rho)  \leq 32 \modp(\rho).
\end{align}
Note that to obtain \eqref{2.2.5}, we have made use of \eqref{exp_d} along with the fact that $\al_0 \geq 1$ to get $\scalex{\al_0}{z_0} \leq 1$ and $\scalet{\al_0}{z_0} \leq 1$.

Since we have $\rho_0 \leq 1$, we must also have $\rho \leq 1$. Thus the bound $\rho^{p^+_{\qfve} - p^-_{\qfve}} \leq 1$ is trivial. On the other hand, to bound $\rho^{-(p^+_{\qfve} - p^-_{\qfve})}$, we use the fact that $\modp$ is $\log$-H\"older continuous to get
\begin{equation*}
\rho^{-(p^+_{\qfve} - p^-_{\qfve})} \leq \rho^{-32\modp(\rho)} \leq e^{32\modp(\rho) \log \frac{1}{\rho}} \overset{\text{Remark \ref{remark_def_p_log}}}{\leq} e^{32L}.
\end{equation*}

This proves the second bound.
\item[Bound for $\al_0$:] From \eqref{est_bnd_al_0}, we see that 
\begin{equation*}
\al_0^{1-\be} \leq \frac{1}{\scalexn{\al_0}{z_0} \rho^{n+2} \scalet{\al_0}{z_0}} \iint_{\qnot} (|\nabla u| + 1)^{p(z)-\ve} \, dz.
\end{equation*}

Making use of \eqref{def_M_0}, we get
\begin{equation*}
\al_0^{1-\be \scalexexpn{\al_0}{z_0} \scaletexp{\al_0}{z_0}} \leq \frac{\bM_0}{\rho^{n+2}}.
\end{equation*}

From \eqref{exp_d}, we see that $\frac{d(n+2)}{2} - \frac{n}{p^-} > 0$ and from Definition \ref{restriction_be_0}, we have 
\begin{equation}\label{tildeve} \tilde{\ve}:= 1-\be -\frac{n}{p^-}+\frac{nd}{2}-1+d \geq \frac{d(n+2)}{4}-\frac{n}{2p^-} > 0.\end{equation}
Along with the fact that $\rho_0 \apprle \frac{1}{\bM_0}$ (see Definition \ref{restriction_rho_0}), we get
\begin{equation}
\label{2.2.11}
\al_0 \leq \lbr \frac{1}{\rho^{n+3}}\rbr^{\frac{1}{\tilde{\ve}}}. 
\end{equation}

Using \eqref{2.2.5} into \eqref{2.2.11}, we get
\begin{equation*}
\al_0^{p^+_{\qfve} - p^-_{\qfve}} \leq \rho^{-\frac{32(n+3)}{\tilde{\ve}} \om(\rho)} \leq e^{\frac{32(n+3)}{\tilde{\ve}} \om(\rho) \log \frac{1}{\rho}} \overset{\text{Remark \ref{remark_def_p_log}}}{\leq} C_{(n,\plog,m_e)}.
\end{equation*}

\end{description}

\end{proof}

\begin{lemma}
\label{lemma_vitali}
For every $z \in \qfur\setminus\elam$, consider the parabolic cylinders of the form 
\begin{equation*}
\label{q_rho_z}
Q_{\rho_z}^{\la}(z) := B_{\scalex{\la}{z}\rho_z}(x) \times (t - \scalet{\la}{z}\rho_z^2,t + \scalet{\la}{z}\rho_z^2)
 \end{equation*}
 where \begin{equation}
 \label{2.2.19}
 \rho_z := d^{\la}_z(z,\elam) := \inf_{\mfz \in \elam} d^{\la}_z(z,\mfz).
 \end{equation}
 Let $\mft \in (0,1)$ be a given constant and consider the open covering of $\qfur \setminus \elam$ given by 
 \begin{equation}
 \label{covering_F}
 \mcf := \left\{ Q_{\mft \rho_z}^{\la}(z)\right\}_{z \in \qfur \setminus \elam}.
 \end{equation}

 Then there exists a universal constant $\mfi = \mfi_{(n,\plog,m_e)}\geq 9$ and a countable disjoint sub-collection $ \mcg := \{Q_{\rho_i}^{\la}(z_i)\}_{i \in \NN}\subset \mcf$  such that there holds
 \begin{equation*}
 \bigcup_{\mcf}Q_{\rho_z}^{\la}(z) \subset \bigcup_{\mcg} Q_{\mfi \rho_{z_i}}^{\la} (z_i).
 \end{equation*}
\end{lemma}

 \begin{proof}
 From \eqref{covering_F} and the fact that $\qfur \subset \qfve$, we see that 
 \begin{equation}
 \label{obse-1}
 \begin{array}{c}
 \scalex{\la}{z}\mft \rho_z \leq \scalex{\la}{z} \rho_z \leq 32 \scalex{\al_0}{z_0}\rho \leq 32\rho  \leq \frac{\rho_0}{4}, \\
 \scalet{\la}{z} (\mft\rho_z)^2  \leq \scalet{\la}{z} (\rho_z)^2 \leq \scalet{\al_0}{z_0} 2 (16 \rho)^2 \leq 512 \rho \leq 4 \rho_0.
 \end{array}
%  \end{array}
 \end{equation}
 Note that to obtain \eqref{obse-1}, we have used  $\rho_0 \leq 1$. Also from \eqref{obse-1}, it is easy to observe  that $Q_{\rho_z}^{\la} \subset \mfq$ for all $z \in \qfur \setminus \elam$.

 From \eqref{obse-1}, we see that $\sup_{z \in \qfur\setminus \elam} \rho_z := R < \infty$ and hence for any $j\in \NN$, consider the following sub-collection of \eqref{covering_F}:
 \begin{equation*}
 \mcf_j := \left\{ Q_{\rho_z}^{\la}(z) \in \mcf : \frac{R}{2^{j}} < \rho_z \leq \frac{R}{2^{j-1}} \right\}.
 \end{equation*}
Let us now extract a countable collection that satisfies the conclusions of this lemma.  Assume that the countable disjoint sub-collection $\mcg_1 \subset \mcf, \mcg_2\subset \mcf, \ldots, \mcg_{k-1}\subset \mcf$ has been chosen, we will now inductively choose $\mcg_k$ to be the maximal disjoint sub-collection from the set
 \begin{equation*}
 \left\{ Q \in \mathcal{F} : Q \cap \htq = \emptyset,  \, \forall \htq \in \bigcup_{j=1}^{k-1} \mcg_j\right\}.
 \end{equation*}

Now let us define 
 \begin{equation}
 \label{def_mcg}
 \mcg := \bigcup_{j=1}^{\infty} \mcg_j. 
 \end{equation}
It is easy to see that  $\mcg$ constructed in \eqref{def_mcg} is a disjoint countable sub-collection  of $\mcf$.

In order to prove the Vitali covering lemma, we need to show the following two properties hold:
\begin{itemize}
\item For any $\lnq \in \mcf$, there exists an $\htq\in \mcg$ such that $\lnq \cap \htq \neq \emptyset$ and 
\item There exists a universal constant $\mfi = \mfi_{(n,\plog,m_e)}$ such that $\lnq \subset \mfi \htq$ holds. 
\end{itemize}

Let us now fix  $\lnq \in \mcf$, then it must necessarily belong to $\mcf_k$ for some $k \in \NN$.  By the maximality of $\mcg_k$, there exists $\htq \in \bigcup_{j=1}^k \mcg_j$ such that $\lnq \cap \htq \neq \emptyset$. This proves the first assertion of the Vitali Covering.

Let us now prove the second assertion of the Vitali covering lemma.  Recall that $\lnq \in \mcf$ and $\htq \in \mcg$ are two fixed cylinders such that $\lnq \cap \htq \neq \emptyset$. 
We shall use the notation
\begin{equation*}
\label{notation_barzhatz}
\begin{array}{l}
\lnq = \lnb \times \lni = Q_{\rzb}^{\la}(\lz)= B_{\scalex{\la}{\lz}\rzb}(\lnx) \times (\lntm - \scalet{\la}{\lz} \rzb^2, \lntm + \scalet{\la}{\lz} \rzb^2),\\
\htq = \htb \times \hti = Q_{\rzh}^{\la}(\hz)= B_{\scalex{\la}{\hz}\rzh}(\htx) \times (\httm - \scalet{\la}{\hz} \rzh^2, \httm + \scalet{\la}{\hz} \rzh^2).
\end{array}
\end{equation*}
From \eqref{covering_F}, we also know that $\lz, \hz \in \qfur \setminus \elam$.  From the choice of $\htq$ and $\lnq$, we have
\begin{equation}
\label{rhobarzhatz}
\rzb  \leq \frac{R}{2^{k-1}} = \frac{2R}{2^k} \leq 2 \rzh.
\end{equation}

We shall now show $\lnb \subset \mfi \htb$ and $\lni \subset \mfi \hti$ for some $\mfi = \mfi_{(n,\plog,m_e)}$:
\begin{description}
\item[Inclusion $\lnb \subset \mfi \htb$:] It would suffice to show the following holds:
\begin{equation}
\label{2.2.27}
\scalex{\la}{\lz}\rzb \leq \mfi \scalex{\la}{\hz} \rzh.
\end{equation}

After using \eqref{rhobarzhatz}, proving \eqref{2.2.27} reduces to showing 
\begin{equation}
\label{2.2.28}
\la^{p(\lz) - p(\hz)} \leq C_{(n,\plog,m_e)}.
\end{equation}

Since $\la \geq 1$, if $p(\lz) \leq p(\hz)$, then we trivially have $\la^{p(\lz) - p(\hz)} \leq 1$.  Hence without loss of generality, we shall assume $p(\lz) \geq p(\hz)$.
Since we have $\lnq \cap \htq \neq \emptyset$, there exists a point $\tz \in \lnq \cap \htq$ which we use to obtain
\begin{equation}
\label{2.2.29}
\la^{p(\lz) - p(\hz)} = \la^{p(\lz) - p(\tz)} \la^{p(\tz) - p(\hz)}.
\end{equation}
% 
% From  \eqref{2.2.29}, we see that in order to prove \eqref{2.2.28}, we will have to bound $\la^{p(\lz) - p(\tz)}$ for any $\lz_1, \lz_2 \in \lnq$.
% \begin{equation}
% \label{2.2.30}
% \la^{p(\lz_1) - p(\lz_2)} \leq C_{(n,\plog,m_e)} \qquad \forall \ \lz_1, \lz_2 \in \lnq.
% \end{equation}
We bound $\la^{p(\lz) - p(\tz)}$ as follows:
\begin{equation}
\label{2.2.31}
% \begin{array}{ll}
p(\lz) - p(\tz)  \leq \modp (d_p(\lz,\tz)) \leq \modp(2 \scalex{\la}{\lz} \rzb +  2 \scalet{\la}{\lz} \rzb) \leq 4 \modp(\rzb).
% & \leq \om \lbr \la^{\frac{d}{2}-\frac{1}{p(\bar{z})}}\rho_{\bar{z}} + \la^{-1+d}\rho_{\bar{z}}\rbr\\
% & \leq \om (2 \rho_{\bar{z}}) \\
% & \leq 2 \om(\rho_{\bar{z}}). 
% \end{array}
\end{equation}
To obtain the above bound, we have used $\la \geq 1$ combined with \eqref{exp_d}.

On $\lnq$, we have $g(\cdot) > \la^{1-\be}$, which combined with Lemma \ref{max_bnd} {and} Theorem \ref{measure_density_poincare} gives
% \begin{equation}
% \label{2.2.32}
% \la^{1-\be} \leq \fiint_{\lnq} g(z) \ dz.
% \end{equation}
% 
% Using the strong Maximal function bound, we get
\begin{equation}
\label{2.2.33}
\begin{array}{rcl}
\la^{1-\be} |\lnq| & \leq &\iint_{\lnq} g(z) \, dz \\
% &  \leq \iint_{\RR^{n+1}} \mm(|\nabla u|^q \lsb{\chi}{\qfur})^{\frac{p^--\ve}{q}}(z) \ dz + \mm(|\nabla \tu|^q \lsb{\chi}{8Q_0})^{\frac{p^--\ve}{q}}(z) \, dz \\
& \overset{\text{Lemma \ref{max_bnd}}}{\apprle} & \iint_{\qfur} \left[|\nabla u|+1\right]^{\pp(1-\be)}   + \left[\frac{|u|}{\scalex{\al_0}{z_0}\rho}\right]^{\pp(1-\be)} \, dz \\
& \overset{\text{Theorem \ref{measure_density_poincare}}}{\apprle} & \iint_{\qfurr} \left[|\nabla u|+1\right]^{\pp(1-\be)} \, dz\\
% & \qquad \qquad + \iint_{8Q_0} \left|\frac{u}{\al_0^{-\frac{1}{(n+2)p(z_0)}}\rho} \right|^{p^--\ve} dz \\
% & \apprle \iint_{8Q_0} |\nabla u|^{p^--\ve} \lsb{\chi}{8Q_0} \, dz +  + \iint_{8Q_0} \left|\nabla u\right|^{p^--\ve} dz \\
% & \apprle \iint_{8Q_0} \lbr |\nabla u|+ 1\rbr ^{p(z)-ve} \lsb{\chi}{8Q_0} dz  \leq M.
& \overset{\eqref{def_M_0}}{\leq} & C_{(n,\plog,m_e)} \bM_0.
\end{array}
\end{equation}

We also have 
\begin{equation}
\label{2.2.34}
|\bar{Q}| = \la^{\scalexexpn{\la}{\lz}} \rzb^{n+2} \scalet{\la}{\lz}.
\end{equation}

Combining \eqref{2.2.34} and \eqref{2.2.33}, we get
\begin{equation*}
\label{2.2.35}
\la^{1-\be + \frac{nd}{2} - \frac{n}{p^-} -1+d} \leq \la^{1-\be + \frac{nd}{2} - \frac{n}{p(\bar{z})} -1+d} \leq \frac{\bM_0}{\rho_{\bar{z}}^{n+2}}.
\end{equation*}

Similarly to the calculation in \eqref{tildeve}, we obtain
\begin{equation}
\label{2.2.36.1}
\la \leq \lbr \frac{\bM_0}{\rzb^{n+2}} \rbr^{\frac{1}{\tilde{\ve}}}. 
\end{equation}

Combining \eqref{2.2.31} and \eqref{2.2.36.1} gives
\begin{equation}
\label{2.2.36}
\la^{p(\lz) - p(\tz)} \leq \bM_0^{\frac{p(\lz) - p(\tz)}{\tilde{\ve}}}\lbr \frac{1}{\rzb} \rbr^{\frac{4\modp(\rzb)}{\tilde{\ve}}} \leq C_{(n,\plog,m_e)}\bM_0^{p(\bar{z}) - p(\tilde{z})}.
\end{equation}

Analogous to  \eqref{2.2.36}, there also holds
\begin{equation}
\label{2.2.37}
\la^{ p(\tz)-p(\hz)}  \leq C_{(n,\plog,m_e)}\bM_0^{ p(\tz)-p(\hz)}.
\end{equation}

Combining \eqref{2.2.36},\eqref{2.2.37} and \eqref{2.2.29}, we get
\begin{align}
\nonumber \la^{p(\bar{z}) - p(\hat{z})} & = \la^{p(\bar{z}) - p(\tilde{z})} \la^{p(\tilde{z}) - p(\hat{z})} \\
\nonumber & \leq C_{(n,\plog,m_e)}\bM_0^{p(\bar{z}) - p(\tilde{z})} \bM_0^{ p(\tilde{z})-p(\hat{z})} \\
\label{2.2.38} & = C_{(n,\plog,m_e)} \bM_0^{p(\bar{z}) -p(\hat{z})}.
\end{align}

Recall from the covering \eqref{covering_F}, we have $\lz,\hz \in \qfur$, which implies 
\begin{equation*}
\label{2.2.39}
{p(\lz) -p(\hz)} \leq p^+_{\qfve} - p^-_{\qfve} \overset{\eqref{2.2.5}}{\leq} 64 \modp (\rho).
\end{equation*}

Since we have restricted $\rho \apprle \rho_0 \apprle \frac{1}{\bM_0}$, we obtain
\begin{equation}
\label{2.2.40}
\bM_0^{p(\bar{z}) -p(\hat{z})} \leq \rho^{-16\om(\rho)} \leq C_{(n,\plog,m_e)}.
\end{equation}

Combining \eqref{2.2.40} and \eqref{2.2.38} implies the bound \eqref{2.2.28}.  This concludes the proof of the \emph{space inclusion}.

  \item[Inclusion $\lni \subset \mfi \hti$:] In order to show $\lni \subset \mfi \hti$, we need to show the following holds:
\begin{equation*}
\label{2.2.41}
\scalet{\la}{\rzb}\rzb^2 \leq \mfi \scalet{\la}{\rzh}\rzh^2.
\end{equation*}

But from \eqref{rhobarzhatz}, it is easy to get
\begin{equation*}
\label{2.2.42}
\scalet{\la}{\rzb}\rzb^2 \leq 4 \scalet{\la}{\rzh}\rzh^2 \quad \Longrightarrow \quad \lni \subset 9 \hti.
\end{equation*}

\end{description}

This completes the proof of the Vitali type lemma. 
 \end{proof}

  \begin{lemma}
 \label{whitney_covering}
 Let $\mcf$ be a covering of $\qfur \setminus \elam$ given by the cylinders:
 \begin{equation*}
\mcf := \left\{Q_{\frac{\de}{\mfi}\rho_z}^{\la}(z)\right\}_{z \in \qfur \setminus \elam},
 \end{equation*}
 with $\de= \frac{1}{4\hat{c}}$, where $\hat{c}$ is from \descref{W4}{W4} and $\mfi$ is the constant from Lemma \ref{lemma_vitali}.
 
 Subordinate to the covering $\mcf$, there exists a countable sub-collection 
 $
\mcg = \left\{ Q_{\de \rho_{z_i}}^{\la}(z_i)\right\}_{i \in \NN} = \{ Q_i\}_{i \in \NN} 
 $
 such that the following holds:
 
\begin{description}
\descitem{(W1)}{W1} $\qfur \setminus \elam \subset \bigcup_{i \in \NN}Q_i $.
\descitem{(W2)}{W2} Each point $z \in \qfur \setminus \elam$ belongs to at most $C_{(n,\plog,m_e)}$ cylinders of the form $2Q_i$. 
\descitem{(W3)}{W3} There exists a constant $C=C_{(n,\plog,m_e)}$ such that for any two cylinders $Q_i$ and $Q_j$ with $2Q_i \cap 2Q_j \neq \emptyset$, there holds
\begin{equation*}
|B_i| \leq C |B_j| \leq C |B_i| \txt{and} |I_i| \leq C |I_j| \leq C |I_i|.
\end{equation*}
In particular, there holds $|Q_i| \approx_{(n,\plog,m_e)} |Q_j|$.
\descitem{(W4)}{W4} There exists a constant $\hat{c} = \hat{c}_{(n,\plog,m_e)}\geq 9$ such that for all $i \in \NN$, there holds:
\begin{equation*}
\hat{c} Q_i \subset \RR^{n+1} \setminus \elam \txt{and} 8\hat{c} Q_i \cap \elam \neq \emptyset.
\end{equation*}
\descitem{(W5)}{W5} For the constant $\hat{c}$ from above, there holds $2Q_i \cap 2Q_j \neq \emptyset$ implies $2Q_i \subset \hat{c}Q_j$. 
\end{description}
\end{lemma}

 \begin{proof}
 Applying Lemma \ref{lemma_vitali} to the covering $\mcf$, we obtain a countable disjoint collection $\mch = \{Q_i\}_{i\in\NN}$ such that 
  \begin{equation*}
  \label{inclusion_whitney}
  \qfur \setminus \elam \subset \bigcup_{i \in \NN} \mfi Q_i.
  \end{equation*}
Let us define
 \begin{equation*}
 \label{def_mcg_whitney}
 \mcg:= \left\{ \mfi Q_i =Q_{\de\rho_{z_i}}^{\la}(z) :  Q_i \in \mch \right\}.
 \end{equation*}
We shall now show that the countable collection $\mcg$ satisfies all the properties:
\begin{description}
\item[Proof of \descref{W1}{W1}:] Using Lemma \ref{lemma_vitali}, we constructed $\mcg$ and this property is automatically satisfied. 
\item[Proof of \descref{W3}{W3}:] Let $Q_i$ and $Q_j$ be two cylinders such that $2Q_i \cap 2Q_j \neq \emptyset$. Note that the centers of $Q_i$ and $Q_j$ denoted by $z_i, z_j \in \qfur$. Following the procedure in obtaining \eqref{2.2.38}, we get
\begin{equation}
\label{2.2.26-1}
\la^{p^+_{2Q_i} - p^-_{2Q_i}} \leq C_{(n,\plog,m_e)} M^{p^+_{2Q_i} - p^-_{2Q_i}}.
\end{equation}

Since $\de \in \lbr 0,\frac14\rbr$, we have the following easy calculation:
\begin{equation}
\label{2.2.27-1}
\begin{array}{rcl}
{p^+_{2Q_i} - p^-_{2Q_i}} & \leq & \modp \left( 2 \scalex{\la}{z_i}\de \rho_{z_i} + 2\sqrt{\scalet{\la}{z_i} (\de \rho_{z_i})^2} \right) \\
& \overset{\eqref{obse-1}}{\leq}&  \modp  \lbr  16 \scalex{\al_0}{z_0} \rho + \sqrt{ \scalet{\al_0}{z_0}  (16  \rho)^2 } \rbr \\
& \leq & 32 \modp (\rho).
\end{array}
\end{equation}

Using Definition \ref{restriction_rho_0} and \eqref{def_q_0} followed by combining \eqref{2.2.26-1} and \eqref{2.2.27-1}, we get
\begin{equation}
\label{2.2.28-1}
\la^{p^+_{2Q_i} - p^-_{2Q_i}} \leq C_{(n,\plog,m_e)}.
\end{equation}

Since we have $2Q_i \cap 2Q_j \neq \emptyset$, we can pick some $\tz \in 2Q_i \cap 2Q_j$ and using \eqref{2.2.28-1}, we get
\begin{equation}
\label{2.2.29-1}
\la^{p(z_i) - p(z_j)} = \la^{p(z_i) - p(\tz)} \la^{p(\tz) - p(z_j)} \leq \la^{p^+_{2Q_i} - p^-_{2Q_i}} \la^{p^+_{2Q_j} - p^-_{2Q_j}} \overset{\eqref{2.2.28-1}}{\leq} C_{(n,\plog,m_e)}.
\end{equation}

Using \eqref{2.2.29-1} and \eqref{2.2.19}, we can compare the parabolic metrics $d_{z_i}^{\la}$ and $d_{z_j}^{\la}$ (see \eqref{loc_par_met} and \ref{not10})  as follows: Let $z_i, z_2 \in \RR^{n+1}$ be two arbitrary points:
\begin{equation}
\label{2.2.30-1}
\begin{array}{rcl}
d_{z_i}(z_1,z_2) & := &\max \left\{ \scalex{\la}{z_i}|x_1-x_2|, \sqrt{\scalet{\la}{z_i} |t_1-t_2|}\right\}\\
& = & \max \left\{ \scalex{\la}{z_i}|x_1-x_2|\la^{\frac{1}{p(z_i)}-\frac{1}{p(z_j)}}, \sqrt{\scalet{\la}{z_i} |t_1-t_2|}\right\}\\
& \overset{\eqref{2.2.29-1}}{\leq} & C_{(n,\plog,m_e)} \max \left\{ \scalex{\la}{z_i}|x_1-x_2|, \sqrt{\scalet{\la}{z_i} |t_1-t_2|}\right\}\\
&=: & C_{(n,\plog,m_e)} d_{z_j}(z_1,z_2).
\end{array}
\end{equation}

Let us denote $r_i := \de \rho_{z_i}$ and let $\hz_j \in \elam$ be such that $d_{z_j}(z_j,\hz_j) = d_{z_j}(z_j,\elam)$  (possible since $\elam$ is a closed set) and let $\tz \in 2Q_i \cap 2Q_j$ be as before.   Then we have
\begin{equation}
\label{2.2.31-1}
\begin{array}{rcl}
r_i = \de d_{z_i}(z_i,\elam)  &\leq & \de d_{z_i}(z_i,\hat{z}_j) \leq \de \lbr d_{z_i}(z_i,\tz) + d_{z_i}(\tz, z_j) + d_{z_i}(z_j,\hat{z}_j) \rbr\\
& \overset{\eqref{2.2.30-1}}{\leq} & \de \lbr 2r_i + Cd_{z_j}(\tz, z_j) + Cd_{z_j}(z_j,\hat{z}_j) \rbr\\ 
& \leq &\de \lbr 2r_i + 2Cr_j +  \frac{2C}{\de} r_j \rbr \\
& \leq & \frac12 r_i + 3C r_j.
\end{array}
\end{equation}
To obtain  the last inequality of \eqref{2.2.31-1}, we used the fact that  $\de \in \lbr 0,\frac14\rbr$.  This implies there exists a constant $C=C_{(n,\plog,m_e)}$ such that 
\begin{equation}\label{2.2.32-1}r_i \leq C r_j.\end{equation}

\begin{description}[leftmargin=*]
\item[Comparison of  $|B_i|$ and $|B_j|$:] We proceed as follows:
\begin{equation*}
\label{2.2.33-1}
\begin{array}{rcl}
|B_i|  & = & \la^{\frac{nd}{2}-\frac{n}{p(z_i)}} r_i^n = \la^{\frac{nd}{2}-\frac{n}{p(z_j)}} r_i^n \la^{\frac{n}{p(z_j)}-\frac{n}{p(z_i)}}\\
% & \overset{\eqref{2.2.28-1}}{\leq}& C_{(n,\plog,m_e)} \la^{\frac{nd}{2}-\frac{n}{p(z_j)}} r_i^n \\
& \overset{\eqref{2.2.28-1},\eqref{2.2.32-1}}{\leq}& C_{(n,\plog,m_e)} \la^{\frac{nd}{2}-\frac{n}{p(z_j)}} r_j^n  = C_{(n,\plog,m_e)} |B_j|.
\end{array}
\end{equation*}

\item[Comparison of  $|I_i|$ and $|I_j|$:]
\begin{equation*}
\label{2.2.34-1}
% \begin{array}{ll}
|I_i|  = \la^{-1+d} r_i^2 
 \overset{\eqref{2.2.32-1}}{\leq} C_{(n,\plog,m_e)} \la^{-1+d} r_j^2 = C_{(n,\plog,m_e)} |I_j|. 
% \end{array}
\end{equation*}

\end{description}

\item[Proof of \descref{W5}{W5}:]  Let $\mfz_i \in 2Q_i$ be any point and let  $\tz \in 2Q_i \cap 2Q_j\neq \emptyset$ be as from before.  We now get
\begin{equation*}
\label{2.2.35-1}
\begin{array}{rcl}
d_{z_j}(z_j,\mfz_i) & \leq & d_{z_j}(z_j,\tz)+d_{z_j}(\tz,z_i)+d_{z_j}(z_i,\mfz_i) \\
& \overset{\eqref{2.2.30-1}}{\leq}& 2r_j + c d_{z_i}(\tz,z_i)+cd_{z_i}(z_i,\tz_i) \\
& \leq & 2r_j + 4cr_i \\
& \overset{\eqref{2.2.32-1}}{\leq} & \hat{c} r_j. 
\end{array}
\end{equation*}
This implies $2Q_i \subset \hat{c} Q_j$. \emph{This is where we obtain the constant $\hat{c}$}.

\item[Proof of \descref{W4}{W4}:] At this point, let us choose $\de = \frac{1}{4\hat{c}}$. Then we see that 
\[\hat{c} Q_i = Q_{\frac{\rho_{z_i}}{4}}^{\la}(z_i)\subset \RR^{n+1} \setminus \elam \txt{and} 8\hat{c}Q_i \cap \elam \neq \emptyset, \] since $\rho_{z_i} = d^{\la}_{z_i}(z_i,\elam)$. 

\item[Proof of \descref{W2}{W2}:]  Let us fix $z \in \qfur \setminus \elam$ and define the index set
\begin{equation}\label{2.2.36-1}\mci_z:= \{ i \in \NN : z \in 2Q_i\}.\end{equation}

We need to show that $\# I_z \leq C_{(n,\plog,m_e)} < \infty$.  Let us fix some $i \in \mci_z$, then for any $k \in \mci_z$, using \descref{W3}{W3}, we have $|B_k| \geq \frac{1}{C_{(n,\plog,m_e)}} |B_i|$ which implies $\#\mci_z < \infty$. 
% \begin{equation}
% \label{2.2.37-1}
% |B_k| \geq \frac{1}{c} |B_i|.
% \end{equation}
% Thus \eqref{2.2.37-1} implies $\# I_z < \infty$.
Now we shall proceed with showing  $\#\mci_z \leq C_{(n,\plog,m_e)}$, i.e., a uniform bound exists.  Since $\#\mci_z <\infty$, there exists  $i_0 \in \mci_z$ such that 
\begin{equation*}
\label{2.2.38-1}
\min_{i \in \mci_z} |Q_i|  = |Q_{i_0}|.
\end{equation*}
 Moreover by \descref{W5}{W5}, we know that $Q_i \subset \hat{c} Q_{i_0}$ for any $i \in \mci_z$.  Taking into account that $\mfi^{-1} Q_i$ are disjoint, we get 
 \begin{equation}
 \label{2.2.39-1}
 |\mfi^{-1} Q_{i_0}| \# \mci_z  \leq |\hat{c} Q_{i_0}|.
 \end{equation}

 Thus from \eqref{2.2.39-1}, and using the fact that $\mfi = \mfi_{(n,\plog,m_e)}$ and $\hat{c} = \hat{c}_{(n,\plog,m_e)}$, we get \[\#\mci_z \leq (\hat{c} \mfi)^{n+2} \leq C_{(n,\plog,m_e)},\] which proves the desired assertion.

\end{description}
 
 \end{proof}

 \begin{lemma}
 \label{partition_unity}
 Subordinate to the covering $\mcg$ obtained in Lemma \ref{whitney_covering}, we obtain a partition of unity $\{ \psi\}_{i=1}^{\infty}$ on $\RR^{n+1} \setminus \elam$ that satisfies the following properties:
 \begin{itemize}
 \item $\sum_{i=1}^{\infty} \psi_i(z) = 1$ for all $z \in \qfur \setminus\elam$.
 \item $\psi_i \in C_c^{\infty}(2Q_i)$.
 \item $\|\psi_i\|_{\infty} + \scalex{\la}{z_i} r_i \| \nabla \psi_i\|_{\infty} + \scalet{\la}{z_i} r_i^2 \| \pa_t \psi_i\|_{\infty} \leq C_{(n,\plog,m_e)}$, where we have used the notation $r_i := \de \rho_{z_i}$ which is the parabolic radius of $Q_i$ with respect to the metric $d_{z_i}^{\la}$.
 \item $\psi_i \geq C_{(n,\plog,m_e)}$ on $Q_i$. 
 \end{itemize}

 \end{lemma}
 
  \section{Method of Lipschitz truncation}
 \label{section_four}

 From Lemma \ref{whitney_covering}, let us define  the following enlarged cylinders (recall $\hat{c}$ is from \descref{W4}{W4})
 \begin{equation}
 \label{def_qihat}
 \htq_i := \hat{c}Q_i = Q_{\hat{r}_i}^{\la}(z_i), \txt{where} \hat{r}_i := \hat{c} r_i.
 \end{equation}
% where $\hat{r}_i = \hat{c} r_i$. 

We shall also use the notation
\begin{equation}
\label{mcii}
\mci(i) := \{ j \in \NN : \spt(\psi_i) \cap \spt(\psi_j) \neq \emptyset\} \txt{and} \mci_z := \{ j \in \NN: z \in \spt(\psi_j) \}.
\end{equation}

 Let us now construct the Lipschitz truncation function:
 \begin{equation}
 \label{lipschitz_function}
 \vlh(z) = \tuh(z) - \sum_{i} \psi_i(z) \lbr \tuh(z) - \tuh^i\rbr,
 \end{equation}
 where 
 \begin{equation}
 \label{def_tuh}
 \tuh^i = \left\{
 \begin{array}{ll}
 \fiint_{2Q_i} \tuh(z) \ dz & \text{if} \ 2Q_i \subset (\Om \cap B_{\rho_2}^{\al_0}(x_0)) \times \RR, \\
 0 & \text{else}. 
 \end{array}\right.
 \end{equation}

 From construction in \eqref{lipschitz_function} and \eqref{def_tuh}, we have
 \begin{equation*}
 \spt(\vlh) \subset (\Om \cap B_{\rho_b}^{\al_0}(x_0)) \times \RR. 
 \end{equation*}

 We see that $\vlh$ has the right support for the test function and hence the rest of this section will be devoted to proving the Lipschitz regularity of $\vlh$ on a suitable region.

 Let us consider the following time slices:
 \begin{gather*}
 S_1 := \lbr[\{] t \in \RR : |t - t_0| \leq \scalet{\al_0}{z_0} \left(\rho_a + \frac19 (\rho_b - \rho_a)\right)^2 \rbr[\}] \label{S_1}, \\
 S_2 := \lbr[\{] t \in \RR : |t - t_0| \leq \scalet{\al_0}{z_0} \left(\rho_a + \frac29 (\rho_b - \rho_a)\right)^2 \rbr[\}] \label{S_2}. 
 \end{gather*}
 
 The rest of this section will be devoted to showing $\vlh \in C^{0,1}(B_{\rho_2}^{\al_0}(x_0) \times S_1)$. In this regard, let us further define
 \begin{gather}
 \Th := \{ i \in \NN : \spt(\psi_i) \cap S_1 \neq \emptyset\} \label{theta}, \\
 \Th_1 := \{ i \in \Th: \htq_i \subset \RR^n \times S_2\} \label{theta_1}, \\
 \Th_2 := \Th \setminus \Th_1, \nonumber
 \end{gather}
where $\htq_i$ is  defined as in \eqref{def_qihat}.

 \begin{remark}
 \label{bound_s}
 Let $2Q_i$ be a given Whitney-type cylinder for some $i \in \Th$. If $i \in \Th_2$, then we have $Q_i \cap S_1 \neq \emptyset$ and $\htq_i \cap S_2^c \neq \emptyset$. Let us define
 \begin{equation}
 \label{def_s}
 s:= \scalet{\al_0}{z_0} (\rho_b - \rho_a) \rho.
 \end{equation}
Then we have the following sequence of estimates:
\begin{equation*}
\label{est_low_s}
\begin{array}{rcl}
\hat{c} \scalet{\la}{z_i} r_i^2 & \geq & \scalet{\al_0}{z_0} \left[  \lbr \rho_a + \frac29 (\rho_b-\rho_a) \rbr^2 - \lbr \rho_a + \frac19 (\rho_b-\rho_a) \rbr^2\right] \\
& \geq & \scalet{\al_0}{z_0} \frac19 (\rho_b - \rho_a) \left[  2\rho_a + \frac39 (\rho_b-\rho_a)\right] \\
& \overset{\eqref{rho_def}}{\geq} & \scalet{\al_0}{z_0} \frac19 (\rho_b - \rho_a) \rho = \frac{s}{9}.
\end{array}
\end{equation*}
Thus for $i \in \Th_2$, there holds
\begin{equation}
\label{bound_low_s}
\scalet{\la}{z_i} r_i^2 \geq \frac{s}{9\hat{c}}.
\end{equation}

 \end{remark}

 \subsection{Preliminary estimates for the test function}
 
\begin{lemma}
\label{lemma3.6_pre}
Let $ z\in \qfur \setminus \elam$, then from \eqref{2.2.36-1}, we have that $z \in 2Q_i$ for some $i \in \mci_z$. For any $1 \leq \tht \leq \frac{p^-}{q}$,  there holds
\begin{gather}
|\tuh^i|^{\tht} \leq  \fiint_{2Q_i} |\tuh(\tz)|^{\tht} \ d\tz  \apprle_{(n,\plog,m_e)} \lbr \scalex{\al_0}{z_0} \rho\rbr^{\tht} \la^{\frac{\tht}{p(z_i)}},  \label{lemma3.6_pre_one}\\
\fiint_{2Q_i}|\nabla \tuh(\tz)|^{\tht} \ d\tz \apprle_{(n,\plog,m_e)} \lbr \frac{16\rho}{\rho_b - \rho_a} \rbr^{\tht} \la^{\frac{\tht}{p(z_i)}}.  \label{lemma3.6_pre_two}
\end{gather}
\end{lemma}
\begin{proof}
\begin{description}[leftmargin=*]
\item[Proof of \eqref{lemma3.6_pre_one}:] We prove this estimate as follows:
\begin{equation*}
\begin{array}{rcl}
|\tuh^i|^{\tht} & \leq & \lbr \scalex{\al_0}{z_0} \rho\rbr^{\tht} \fiint_{2Q_i} \abs{\frac{\tuh}{\scalex{\al_0}{z_0}\rho}}^{\tht} \ d\tz  \leq \lbr \scalex{\al_0}{z_0} \rho\rbr^{\tht} \fiint_{2Q_i} \abs{\frac{u(\tz)}{\scalex{\al_0}{z_0}\rho}}^{\tht} \ d\tz \\
& \apprle &\lbr \scalex{\al_0}{z_0} \rho\rbr^{\tht} \lbr \fiint_{8\hat{c}Q_i} \left[ 1 + \abs{\frac{u(\tz)}{\scalex{\al_0}{z_0}\rho}} \right]^{\frac{\pp}{q}}\ d\tz\rbr^{\frac{\tht q}{p^-_{2Q_i}}} \\
& \overset{\eqref{elam}}{\apprle} &\lbr \scalex{\al_0}{z_0} \rho\rbr^{\tht} \la^{\frac{\tht}{p^-_{2Q_i}}} \\
& \overset{\eqref{2.2.28-1}}{\apprle} &\lbr \scalex{\al_0}{z_0} \rho\rbr^{\tht} \la^{\frac{\tht}{p(z_i)}}.
\end{array}
\end{equation*}

\item[Proof of \eqref{lemma3.6_pre_two}:] We see that $\nabla \tuh (\tx,\tlt) = \zeta(\tlt)\eta(\tx) \nabla u(\tx,\tlt) + \zeta(\tlt) u(\tx,\tlt) \nabla \eta (\tx)$ which combined with \eqref{cut_off_function} implies
\begin{equation}
\label{lemma3.6_pre_1}
|\nabla \tuh| \leq |\nabla u| + \lbr \frac{16\rho}{\rho_b - \rho_a} \rbr \abs{\frac{u}{\scalex{\al_0}{z_0} 16\rho }}.
\end{equation}
We obtain from \eqref{rho_def} that $\lbr \frac{16\rho}{\rho_b - \rho_a} \rbr \geq 1$, which implies
\begin{equation*}
\begin{array}{rcl}
\fiint_{2Q_i} |\nabla \tuh|^{\tht} \ d\tz & \leq &  \lbr \frac{16\rho}{\rho_b - \rho_a} \rbr^{\tht}  \fiint_{2Q_i} |\nabla u|^{\tht} \ d\tz  + \lbr \frac{16\rho}{\rho_b - \rho_a} \rbr^{\tht} \fiint_{2Q_i} \abs{\frac{u(\tz)}{\scalex{\al_0}{z_0} 16\rho }}^{\tht}\ d\tz \\
& \apprle & \lbr \frac{16\rho}{\rho_b - \rho_a} \rbr^{\tht}  \lbr \fiint_{8\hat{c}Q_i} \left[|\nabla u| +1 + \abs{\frac{u(\tz)}{\scalex{\al_0}{z_0} 16\rho }}\right]^{\frac{\pp}{q}}\ d\tz  \rbr^{\frac{\tht q}{p^-_{2Q_i}}}\\
& \overset{\eqref{elam}}{\apprle} & \lbr \frac{16\rho}{\rho_b - \rho_a} \rbr^{\tht} \la^{\frac{\tht}{p^-_{2Q_i}}} \\
& \overset{\eqref{2.2.28-1}}{\apprle} & \lbr \frac{16\rho}{\rho_b - \rho_a} \rbr^{\tht} \la^{\frac{\tht}{p(z_i)}}.
\end{array}
\end{equation*}

\end{description}

\end{proof}

 \begin{corollary}
 \label{lemma3.6}
 For any $z \in \qfur \setminus \elam$,  from \eqref{2.2.36-1},  we have $z \in 2Q_i$ for some $i \in \mci_z$. Then the following bound holds:
 \begin{equation*}
 |\vlh(z)| \apprle_{(n,\plog,m_e)} \scalex{\al_0}{z_0} \rho \la^{\frac{1}{p(z_i)}}.
 \end{equation*}
 \end{corollary}
 
\begin{proof}
From \eqref{lipschitz_function}, we see that for $z \in \qfur \setminus \elam$, there holds $\vlh(z) = \sum_{j \in \mci_z} \psi_j(z) \tuh^j$, where $\mci_z$ is as defined in \eqref{mcii}. Making use of \descref{W2}{W2}, we see that in order to prove the lemma, it is sufficient to bound $|\tuh^j|$ which is proved in Lemma \ref{lemma3.6_pre_1} with $\tht =1$. This completes the proof of the corollary.
% \begin{equation}
% \begin{array}{ll}
% |\tuh^j| & \leq \scalex{\al_0}{z_0} \rho \fiint_{2Q_i} \abs{\frac{\tuh}{\scalex{\al_0}{z_0}\rho}} \ dz  \leq \scalex{\al_0}{z_0} \rho \fiint_{2Q_i} \abs{\frac{u(z)}{\scalex{\al_0}{z_0}\rho}} \ dz \\
% & \leq \scalex{\al_0}{z_0} \rho \lbr \fiint_{8\hat{c}Q_i} \left[ 1 + \abs{\frac{u(z)}{\scalex{\al_0}{z_0}\rho}} \right]^{\frac{\pp}{q}}\ dz\rbr^{\frac{q}{p^-_{2Q_i}}} \\
% & \leq \scalex{\al_0}{z_0} \rho \la^{\frac{1}{p^-_{2Q_i}}} \\
% & \overset{\eqref{2.2.28-1}}{\apprle} \scalex{\al_0}{z_0} \rho \la^{\frac{1}{p(z_i)}}.
% \end{array}
% \end{equation}
% 
% 
% 
% \begin{description}
% \item[Case  $\scalex{\al_0}{z_0} \rho\leq \scalex{\la}{z_i}r_i$:] In this case, we have the estimate:
% 
% 
% \item[Case $\scalex{\al_0}{z_0} \rho \geq  \scalex{\la}{z_i}r_i$:]
% \end{description}
% 
\end{proof}

\begin{lemma}
\label{improved_est}
Let $2Q_i$ be a parabolic Whitney type cylinder and $i \in \Th_1$, where $\Th_1$ is as defined in \eqref{theta_1}. Then for any $1 \leq \tht \leq \frac{p^-}{q}$, there holds
\begin{equation*}
\fiint_{2Q_i} |\tuh(\tz)- \tuh^i|^{\tht} \  d\tz \apprle_{(n,\plog,\lamot,m_e)}\lbr \frac{16\rho}{\rho_b - \rho_a} \rbr^{\tht} \min\left\{ \scalex{\al_0}{z_0}\rho, \scalex{\la}{z_i}r_i\right\}^{\tht} \la^{\frac{\tht}{p(z_i)}}.
\end{equation*}
\end{lemma}
\begin{proof}
Let us consider the following two cases:
\begin{description}[leftmargin=*]
\item[Case $\scalex{\al_0}{z_0}\rho \leq  \scalex{\la}{z_i}r_i$:] In this case, we can use triangle inequality along with \eqref{lemma3.6_pre_one} to get
\begin{equation}
\label{6.18}
\fiint_{2Q_i} |\tuh(\tz)- \tuh^i|^{\tht} \  d\tz \apprle 2 \fiint_{2Q_i} |\tuh(\tz)|^{\tht} \ d\tz \overset{\eqref{lemma3.6_pre_one}}{\apprle} \lbr \scalex{\al_0}{z_0} \rho\rbr^{\tht} \la^{\frac{\tht}{p(z_i)}}.
\end{equation}

\item[Case $\scalex{\al_0}{z_0}\rho \geq  \scalex{\la}{z_i}r_i$:] We apply  Lemma \ref{lemma_crucial_2} with a test function $\mu \in C_c^{\infty}(2B_i)$ satisfying  the bounds $|\mu(x)| \apprle \frac{1}{\lbr \scalex{\la}{z_i} r_i\rbr^n}$ and $|\nabla \mu(x)| \apprle \frac{1}{\lbr \scalex{\la}{z_i} r_i\rbr^{n+1}}$, we get
\begin{equation}
\label{6.19}
\fiint_{2Q_i} |\tuh(\tz)- \tuh^i|^{\tht} \,  d\tz \leq \lbr \scalex{\la}{z_i} r_i\rbr^{\tht} \fiint_{2Q_i} |\nabla \tuh|^{\tht} \, d\tz + \sup_{t_1,t_2 \in 2I_i} |\avgs{\tuh}{\mu}(t_2) - \avgs{\tuh}{\mu}(t_1)|^{\tht}.
\end{equation}
The first term on the right-hand side of \eqref{6.19} can be estimated using \eqref{lemma3.6_pre_two}, which gives
\begin{equation}
\label{est_J_1}
\lbr \scalex{\la}{z_i} r_i\rbr^{\tht} \fiint_{2Q_i} |\nabla \tuh|^{\tht} \ d\tz \apprle \lbr \scalex{\la}{z_i} r_i\rbr^{\tht}   \lbr \frac{16\rho}{\rho_b - \rho_a} \rbr^{\tht} \la^{\frac{\tht}{p(z_i)}}.
\end{equation}

To estimate the second term on the right-hand side of \eqref{6.19}, we make use of Lemma \ref{lemma_crucial_2} with $\phi(x) = \mu(x)$ and $\varphi(t) \equiv 1$ (since $i \in \Th_1$) to get 
\begin{equation}
\label{est_J_2_one}
\begin{array}{rcl}
|\avgs{\tuh}{\mu}(t_2) - \avgs{\tuh}{\mu}(t_1)| & \apprle & \frac{|2Q_i|}{\lbr \scalex{\la}{z_i} r_i\rbr^{n+1}} \fiint_{2Q_i} (1 + |\nabla u|)^{p(\tz) -1} \ d\tz \\
& \apprle & \frac{\scalet{\la}{z_i}r_i^2}{\scalex{\la}{z_i} r_i} \lbr \fiint_{8\hat{c}Q_i} (1 + |\nabla u|)^{\frac{p(\tz)}{q}} \ d\tz \rbr^{\frac{q(p^+_{2Q_i} -1)}{p^-_{2Q_i}}}\\
& \overset{\eqref{elam}}{\apprle} & \la^{-1 + \frac{1}{p(z_i)} + \frac{d}{2}} r_i \la^{\frac{p^+_{2Q_i} -1}{p^-_{2Q_i}}}.
\end{array}
\end{equation}
Now making use of \eqref{2.2.28-1} along with the fact that $\la \geq 1$ and $p^-_{2Q_i} \leq p(z_i)$, we get
\begin{equation}
\label{6.22}
\la^{-1+ \frac{1}{p(z_i)} + \frac{p^+_{2Q_i}}{p^-_{2Q_i}} - \frac{1}{p^-_{2Q_i}}} = \la^{\frac{p^+_{2Q_i}-p^-_{2Q_i}}{p^-_{2Q_i}}} \la^{\frac{p^-_{2Q_i}-p(z_i)}{p(z_i)p^-_{2Q_i}}} \leq \la^{\frac{p^+_{2Q_i}-p^-_{2Q_i}}{p^-_{2Q_i}}} \overset{\eqref{2.2.28-1}}{\apprle} C_{(n,\plog,m_e)}. 
\end{equation}
Substituting \eqref{6.22} into \eqref{est_J_2_one}, we have
\begin{equation}
\label{est_J_2}
|\avgs{\tuh}{\mu}(t_2) - \avgs{\tuh}{\mu}(t_1)| \apprle \lbr \scalex{\la}{z_i} r_i \rbr \la^{\frac{1}{p(z_i)}}.
\end{equation}

Thus, combining \eqref{est_J_1},\eqref{est_J_2} and \eqref{6.19} followed by  making use of the fact that $\frac{16\rho}{\rho_b - \rho_a} \geq 1$, we finally obtain
\[
\fiint_{2Q_i} |\tuh(\tz)- \tuh^i|^{\tht} \  d\tz \apprle_{(n,\plog,\lamot,m_e)} \lbr \scalex{\la}{z_i} r_i\rbr^{\tht}   \lbr \frac{16\rho}{\rho_b - \rho_a} \rbr^{\tht} \la^{\frac{\tht}{p(z_i)}},
\]
which proves the lemma.

\end{description}

\end{proof}

\begin{corollary}
\label{corollary3.7}
For any $i \in \Th_1$ and any $j \in \mci_i$, there holds
\[
|\tuh^i - \tuh^j| \apprle_{(n,\plog,\lamot,m_e)}\lbr \frac{16\rho}{\rho_b - \rho_a} \rbr \min\left\{ \scalex{\al_0}{z_0}\rho, \scalex{\la}{z_i}r_i\right\} \la^{\frac{1}{p(z_i)}}.
\]
\end{corollary}
\begin{proof}
From \eqref{theta_1}, we see that $j \in \Th_1$ for every $j \in \mci_i$ with $i \in \Th_1$. Thus we can split the proof into two cases:
\begin{description}[leftmargin=*]
\item[Case $\scalex{\al_0}{z_0}\rho \leq  \scalex{\la}{z_i}r_i$:] In this case, using triangle inequality along with $i, j\in \Th_1$, we get
\[
|\tuh^i - \tuh^j| \leq |\tuh^i| + |\tuh^j| \overset{\eqref{6.18}}{\apprle} \lbr \scalex{\al_0}{z_0} \rho\rbr^{\tht} \la^{\frac{\tht}{p(z_i)}}.
\]

\item[Case $\scalex{\al_0}{z_0}\rho \geq  \scalex{\la}{z_i}r_i$:]  From \descref{W5}{W5}, we have $Q_j \subset \hat{c}Q_i$, thus we get from triangle inequality
\begin{equation*}
\begin{array}{rcl}
|\tuh^i - \tuh^j| & \leq &  \fiint_{2Q_i} |\tuh(\tz) - \avgs{\tuh}{\hat{c}Q_i}| \ d\tz + \fiint_{2Q_j} |\tuh(\tz) - \avgs{\tuh}{\hat{c}Q_i}| \ d\tz \\
& \leq & \fiint_{\hat{c}Q_i} |\tuh(\tz) - \avgs{\tuh}{\hat{c}Q_i}| \ d\tz + \frac{|\hat{c}Q_i|}{|2Q_j|} \fiint_{\hat{c}Q_i} |\tuh(\tz) - \avgs{\tuh}{\hat{c}Q_i}| \ d\tz \\
& \overset{\text{\descref{W3}{W3}}}{\apprle} & \fiint_{\hat{c}Q_i} |\tuh(\tz) - \avgs{\tuh}{\hat{c}Q_i}| \ d\tz.
\end{array}
\end{equation*}

Applying Lemma \ref{lemma_crucial_2} with  $\mu \in C_c^{\infty}(\hat{c}B_i)$ satisfying $|\mu(x)| \apprle \frac{1}{\lbr \scalex{\la}{z_i} \hat{c} r_i\rbr^n}$ and $|\nabla \mu(x)| \apprle \frac{1}{\lbr \scalex{\la}{z_i} \hat{c} r_i\rbr^{n+1}}$, we get
\begin{equation}
\label{6.19-two}
\fiint_{\hat{c}Q_i} |\tuh(\tz)- \avgs{\tuh}{\hat{c}Q_i}| \,  d\tz \apprle \lbr \scalex{\la}{z_i} \hat{c} r_i\rbr \fiint_{\hat{c}Q_i} |\nabla \tuh| \, d\tz + \sup_{t_1,t_2 \in \hat{c}I_i} |\avgs{\tuh}{\mu}(t_2) - \avgs{\tuh}{\mu}(t_1)|.
\end{equation}
Using \eqref{lemma3.6_pre_two}, we estimate the first term on the right-hand side of \eqref{6.19-two} to get
\begin{equation}
\label{est_J_1-two}
 \scalex{\la}{z_i} r_i \fiint_{\hat{c}Q_i} |\nabla \tuh|^{\tht} \ d\tz \apprle \lbr \scalex{\la}{z_i} r_i\rbr   \lbr \frac{16\rho}{\rho_b - \rho_a} \rbr \la^{\frac{1}{p(z_i)}}.
\end{equation}

To estimate the second term on the right-hand side of \eqref{6.19-two}, we make proceed analogous to  \eqref{est_J_2}, which gives
% of Lemma \ref{lemma_crucial_2} with $\phi(x) = \mu(x)$ and $\varphi(t) \equiv 1$ (since $i \in \Th_1$) to get 
% \begin{equation}
% \label{est_J_2_one}
% \begin{array}{rcl}
% |\avgs{\tuh}{\mu}(t_2) - \avgs{\tuh}{\mu}(t_1)| & \apprle & \frac{|2Q_i|}{\lbr \scalex{\la}{z_i} r_i\rbr^{n+1}} \fiint_{2Q_i} (1 + |\nabla u|)^{p(\tz) -1} \ d\tz \\
% & \apprle & \frac{\scalet{\la}{z_i}r_i^2}{\scalex{\la}{z_i} r_i} \lbr \fiint_{8\hat{c}Q_i} (1 + |\nabla u|)^{\frac{p(\tz)}{q}} \ d\tz \rbr^{\frac{q(p^+_{2Q_i} -1)}{p^-_{2Q_i}}}\\
% & \overset{\eqref{elam}}{\apprle} & \la^{-1 + \frac{1}{p(z_i)} + \frac{d}{2}} r_i \la^{\frac{p^+_{2Q_i} -1}{p^-_{2Q_i}}}
% \end{array}
% \end{equation}
% Now making use of \eqref{2.2.28-1} along with the fact that $\la \geq 1$ and $p^-_{2Q_i} \leq p(z_i)$, we get
% \begin{equation}
% \label{6.22}
% \la^{-1+ \frac{1}{p(z_i)} + \frac{p^+_{2Q_i}}{p^-_{2Q_i}} - \frac{1}{p^-_{2Q_i}}} = \la^{\frac{p^+_{2Q_i}-p^-_{2Q_i}}{p^-_{2Q_i}}} \la^{\frac{p^-_{2Q_i}-p(z_i)}{p(z_i)p^-_{2Q_i}}} \leq \la^{\frac{p^+_{2Q_i}-p^-_{2Q_i}}{p^-_{2Q_i}}} \overset{\eqref{2.2.28-1}}{\apprle} C_{(n,\plog,m_e)}. 
% \end{equation}
% Substituting \eqref{6.22} into \eqref{est_J_2_one}, we get
\begin{equation}
\label{est_J_2-two}
|\avgs{\tuh}{\mu}(t_2) - \avgs{\tuh}{\mu}(t_1)| \apprle \lbr \scalex{\la}{z_i} \hat{c} r_i \rbr \la^{\frac{1}{p(z_i)}}.
\end{equation}
Combining \eqref{est_J_1-two} and \eqref{est_J_2-two} into \eqref{6.19-two} and making use of the fact that $\frac{16\rho}{\rho_b - \rho_a} \geq 1$, we have
\begin{equation*}
\label{6.31_bnd}
\fiint_{\hat{c}Q_i} |\tuh(\tz) - \avgs{\tuh}{\hat{c}Q_i}| \ d\tz \apprle_{(n,\plog,\lamot,m_e)} \lbr \scalex{\la}{z_i} r_i\rbr   \musym \la^{\frac{1}{p(z_i)}}.
\end{equation*}

\end{description}
This completes the proof of the corollary. 
\end{proof}

\subsubsection{Bounds on \texorpdfstring{$\vlh$}\  and \texorpdfstring{$\nabla \vlh$}.}

\begin{lemma}
\label{lemma6.7-1}
Let $Q_i$ be a parabolic Whitney type cylinder. Then for any $z \in 2Q_i$, we have the following bound in the case $i \in \Th_1$ or $i\in\Th_2$ with $\scalex{\al_0}{z_0}(\rho_b-\rho_a) \leq \scalex{\la}{z_i} 15 r_i$:
\begin{equation}
\label{lemma6.7-1_est}
\frac{1}{\scalex{\al_0}{z_0} \rho } |\vlh(z)| + |\nabla \vlh(z)| \apprle_{(n,\plog,\lamot,m_e)} \musym^2 \la^{\frac{1}{p(z_i)}}.
\end{equation}
\end{lemma}
\begin{proof}
The bound for $\vlh(z)$ follows directly from Corollary \ref{lemma3.6} and the fact that $\frac{16\rho}{\rho_b-\rho_a} \geq 1$.

In order to bound $\nabla \vlh(z)$, we consider the two cases:
\begin{description}[leftmargin=*]
\item[Case $i \in \Th_1$:] Since $\sum_{j\in\NN} \psi_j(z)  = 1$, we must have $\sum_{j \in \NN} \nabla \psi_j(z) = 0$ which combined with \descref{W2}{W2}  and Lemma \ref{partition_unity} gives the following sequence of estimates:
\begin{equation*}
\begin{array}{rcl}
|\nabla \vlh(z)| & = & \abs{\sum_{j \in I_i}\lbr \tuh^j - \tuh^i\rbr \nabla \psi_j(z) } \\
& \overset{\text{Corollary \ref{corollary3.7}}}{\apprle} & \lbr \frac{16\rho}{\rho_b - \rho_a} \rbr \frac{ \min\left\{\scalex{\la}{z_i} r_i, \scalex{\al_0}{z_0} \rho \right\}}{\scalex{\la}{z_i} r_i} \la^{\frac{1}{p(z_i)}}\\
& \apprle & \musym^2 \la^{\frac{1}{p(z_i)}}.
\end{array}
\end{equation*}

\item[Case $i \in \Th_2$ and $\scalex{\al_0}{z_0}\rho \leq \scalex{\la}{z_i} r_i$:] In this case, we can again make use of \descref{W2}{W2} along with the bound \eqref{lemma3.6_pre_one} applied with $\tht=1$ and Lemma \ref{partition_unity}, we get
\begin{equation*}
\begin{array}{rcl}
|\nabla \vlh(z)| & \leq \sum_{j \in I_i} |\tuh^j| |\nabla \psi_j(z)| \apprle \frac{\scalex{\al_0}{z_0} \rho \la^{\frac{1}{p(z_i)}}}{\scalex{\la}{z_i} r_i} \apprle \lbr \frac{16\rho}{\rho_b - \rho_a} \rbr \la^{\frac{1}{p(z_i)}}.
\end{array}
\end{equation*}
\end{description}
This completes the proof of the lemma. 
\end{proof}

 \begin{corollary}
 \label{corollary6.7-2}
Let $z \in \qfur \setminus \elam$, then $z \in 2Q_i$ for some $i \in \NN$. Suppose $i \in \Th_1$, then  for any $\de \in (0,1]$, there holds 
 \begin{gather}
 \frac{1}{\scalex{\la}{z_i} r_i } |\vlh(z)| \apprle_{{(n,\plog,\lamot,m_e)}} \lbr \frac{16\rho}{\rho_b - \rho_a} \rbr \frac{\la^{\frac{1}{p(z_i)}}}{\de} +  \frac{\de}{\lbr \scalex{\la}{z_i}r_i\rbr^2 \la^{\frac{1}{p(z_i)}}} |\tuh^i|^2, \label{bound+6.31}\\
  |\nabla \vlh(z)| \apprle_{{(n,\plog,\lamot,m_e)}} \lbr \frac{16\rho}{\rho_b - \rho_a} \rbr \frac{\la^{\frac{1}{p(z_i)}}}{\de}. \label{bound+6.31_two}
 \end{gather}

 \end{corollary}

 \begin{proof}
 Let us prove each of the estimates as follows:
 \begin{description}[leftmargin=*]
 \item[Proof of \eqref{bound+6.31}:] Since $\sum_{j \in I_i}\psi(z) = 1$, we make use of \descref{W2}{W2} and prove the desired estimate as follows:
 \begin{equation*}
 \begin{array}{rcl}
 |\vlh(z)| & \leq & \abs{\sum_{j \in I_i} \lbr \tuh^j - \tuh^i \rbr \psi_j(z)  } + |\tuh^i| \\
 & \overset{\text{Corollary \ref{corollary3.7}}}{\apprle} & \lbr \frac{16\rho}{\rho_b - \rho_a} \rbr \scalex{\la}{z_i} r_i \la^{\frac{1}{p(z_i)}} + |\tuh^i| \\
 & = & \lbr \frac{16\rho}{\rho_b - \rho_a} \rbr \scalex{\la}{z_i} r_i \la^{\frac{1}{p(z_i)}} + |\tuh^i|\sqrt{\frac{\de}{\scalex{\la}{z_i}r_i\la^{\frac{1}{p(z_i)}}}}\sqrt{\frac{\scalex{\la}{z_i}r_i\la^{\frac{1}{p(z_i)}}}{\de}} \\
 & \overset{\redlabel{6.31.a}{a}}{\apprle} & \lbr \frac{16\rho}{\rho_b - \rho_a} \rbr \frac{\scalex{\la}{z_i} r_i}{\de} \la^{\frac{1}{p(z_i)}} + |\tuh^i|^2 {\frac{\de}{\scalex{\la}{z_i}r_i\la^{\frac{1}{p(z_i)}}}}.
 \end{array}
 \end{equation*}
 To obtain \redref{6.31.a}{a}, we made use of Young's inequality along with the fact that $\frac{16\rho}{\rho_b - \rho_a}  \geq 1$. 

 \item[Proof of \eqref{bound+6.31_two}:] From \eqref{lemma6.7-1_est}, we have 
 \begin{equation*}
 |\nabla \vlh(z)| \apprle_{(n,\plog,\lamot,m_e)} \lbr \frac{16\rho}{\rho_b - \rho_a} \rbr \la^{\frac{1}{p(z_i)}} \leq \lbr \frac{16\rho}{\rho_b - \rho_a} \rbr \frac{\la^{\frac{1}{p(z_i)}}}{\de}.
 \end{equation*}
 \end{description}
 This completes the proof of the corollary. 
 \end{proof}

 \begin{lemma}
 \label{lemma6.7-3}
 Let $z \in \qfur \setminus \elam$, then $z \in 2Q_i$ for some $i \in \NN$. Suppose $i \in \Th_2$, then there holds
 \begin{gather}
 |\vlh(z)| \apprle_{(n,\plog,\lamot,m_e)} \scalex{\la}{z_i} r_i \la^{\frac{1}{p(z_i)}} + \frac{\scalex{\la}{z_i}r_i  \la^{{\frac{1-p(z_i)}{p(z_i)}}}}{s} \fiint_{\htq_i} |\tuh(\tz)|^2 \ d\tz \label{bound_6.7-3-1}, \\
 |\nabla \vlh(z)| \apprle_{(n,\plog,\lamot,m_e)} \la^{\frac{1}{p(z_i)}} + \frac{\la^{\frac{1-p(z_i)}{p(z_i)}}}{s} \fiint_{\htq_i} |\tuh(\tz)|^2 \ d\tz \label{bound_6.7-3-2}.
 \end{gather}
 \end{lemma}
 \begin{proof}
 Note that since $i\in \Th_2$, there must hold $\scalet{\la}{z_i} r_i^2 \apprge s$ from \eqref{bound_low_s}. Let us first obtain a rough estimate of the form
 \begin{equation}
 \label{bound6.7-3-one}
 |\tuh^i| \apprle \fiint_{2Q_i} |\tuh(\tz)| \ d\tz \apprle \scalex{\la}{z_i} r_i\la^{\frac{1}{p(z_i)}} + \frac{1}{\scalex{\la}{z_i} r_i\la^{\frac{1}{p(z_i)}}} \fiint_{2Q_i} |\tuh(\tz)|^2 \ d\tz.
 \end{equation}
 
 We now prove both the assertions of the lemma as follows:
 \begin{description}[leftmargin=*]
 \item[Estimate \eqref{bound_6.7-3-1}:] Using \eqref{bound6.7-3-one}, we get
\begin{equation*}
\begin{array}{rcl}
|\vlh(z)| & \leq & \sum_{j \in I_i} |\tuh^j| |\psi_j(z)| \\
& \overset{\text{\descref{W2}{W2}},\eqref{bound6.7-3-one}}{\apprle} & \scalex{\la}{z_i} r_i\la^{\frac{1}{p(z_i)}} + \frac{1}{\scalex{\la}{z_i} r_i\la^{\frac{1}{p(z_i)}}} \fiint_{\htq_i} |\tuh(\tz)|^2 \ d\tz \\
& \overset{\eqref{bound_low_s}}{\apprle} & \scalex{\la}{z_i} r_i\la^{\frac{1}{p(z_i)}} + \frac{\scalex{\la}{z_i} r_i \la^{\frac{1-p(z_i)}{p(z_i)}}}{s} \fiint_{\htq_i} |\tuh(\tz)|^2 \ d\tz.
\end{array}
\end{equation*}

 \item[Estimate \eqref{bound_6.7-3-2}:] Again making use of \eqref{bound6.7-3-one} along with Lemma \ref{partition_unity}, we have
\begin{equation*}
\begin{array}{rcl}
|\nabla \vlh(z)| & \leq & \sum_{j \in I_i} |\tuh^j| |\nabla \psi_j(z)| \\
& \overset{\text{\descref{W2}{W2}},\eqref{bound6.7-3-one}}{\apprle} & \frac{1}{\scalex{\la}{z_i} r_i} \lbr \scalex{\la}{z_i} r_i\la^{\frac{1}{p(z_i)}} + \frac{1}{\scalex{\la}{z_i} r_i\la^{\frac{1}{p(z_i)}}} \fiint_{\htq_i} |\tuh(\tz)|^2 \ d\tz \rbr\\
& \overset{\eqref{bound_low_s}}{\apprle} & \frac{1}{\scalex{\la}{z_i} r_i} \lbr \scalex{\la}{z_i} r_i\la^{\frac{1}{p(z_i)}} + \frac{\scalex{\la}{z_i} r_i \la^{\frac{1-p(z_i)}{p(z_i)}}}{s} \fiint_{\htq_i} |\tuh(\tz)|^2 \ d\tz \rbr \\
& \apprle & \la^{\frac{1}{p(z_i)}} + \frac{ \la^{\frac{1-p(z_i)}{p(z_i)}}}{s} \fiint_{\htq_i} |\tuh(\tz)|^2 \ d\tz.
\end{array}
\end{equation*}
 \end{description}
 \end{proof}

 \subsubsection{Bounds on \texorpdfstring{$\pa_t \vlh$}.}

 \begin{lemma}
 \label{time_vlh}
 Let $ z \in \qfur$, then $z \in 2Q_i$ for some $i \in \Th$. We then have the following estimates for the time derivative of $\vlh$: in the case $i \in \Th_1$, there holds
 \begin{equation}
 \label{bound_time_vlh_one}
 |\pa_t \vlh(\tz)| \apprle_{(n,\plog,\lamot,m_e)} \musym \la^{\frac{p(z_i)-1}{p(z_i)}} \frac{1}{\lbr \scalex{\la}{z_i} r_i\rbr^2} \min \left\{\scalex{\la}{z_i}r_i, \scalex{\al_0}{z_0} \rho \right\},
 \end{equation}
and in the case $i \in \Th_2$, there holds
\begin{equation}
\label{bound_time_vlh_two}
 |\pa_t \vlh(\tz)| \apprle_{(n,\plog,\lamot,m_e)} \musym  \frac{\scalex{\al_0}{z_0} \rho}{s} \la^{\frac{1}{p(z_i)}}.
\end{equation}

 \end{lemma}
 \begin{proof}
 Let us prove each of the assertions as follows:
 \begin{description}[leftmargin=*]
 \item[Estimate \eqref{bound_time_vlh_one}:] From the fact that $\sum_{j \in I_i} \psi_j(z) = 1$, we see that $\sum_{j \in I_i} \pa_t \psi_j(z) = 0$, which along with Lemma \ref{partition_unity} gives the following sequence of estimates
 \begin{equation*}
 \begin{array}{rcl}
 |\pa_t\vlh(z)| & = & \abs{\sum_{j \in I_i} \lbr \tuh^j - \tuh^i \rbr \pa_t \psi_j(z)  } \\
 & \overset{\text{Corollary \ref{corollary3.7}}}{\apprle} & \frac{1}{\scalet{\la}{z_0} r_i^2} \lbr \frac{16\rho}{\rho_b - \rho_a} \rbr \min\left\{ \scalex{\al_0}{z_0}\rho, \scalex{\la}{z_i}r_i\right\} \la^{\frac{1}{p(z_i)}} \\
 & = & \lbr \frac{16\rho}{\rho_b - \rho_a} \rbr  \frac{\la^{\frac{p(z_i)-1}{p(z_i)}}}{\lbr \scalex{\la}{z_i} r_i\rbr^2} \min\left\{ \scalex{\al_0}{z_0}\rho, \scalex{\la}{z_i}r_i\right\}.
 \end{array}
 \end{equation*}

 \item[Estimate \eqref{bound_time_vlh_two}:] In this case, we make use of \eqref{bound_low_s} to obtain 
 \begin{equation*}
 \begin{array}{rcl}
 |\pa_t \vlh(z)| & \leq & \sum_{j \in I_i} |\tuh^j| |\pa_t\psi_j(z)| 
 \overset{\eqref{lemma3.6_pre_one}}{\apprle}   \frac{1}{\scalet{\la}{z_i} r_i^2} \scalex{\al_0}{z_0} \rho \la^{\frac{1}{p(z_i)}} \\
 & \overset{\eqref{bound_low_s}}{\apprle} & \frac{1}{s} \scalex{\al_0}{z_0} \rho \la^{\frac{1}{p(z_i)}} 
  \apprle  \frac{1}{s} \musym   \scalex{\al_0}{z_0} \rho \la^{\frac{1}{p(z_i)}}.
 \end{array}
 \end{equation*}

 \end{description}

 \end{proof}
 
 \subsection{Some important estimates for the test function}
 
 \begin{lemma}
 \label{lemma6.8}
 Let $Q_i$ be a Whitney-type parabolic cylinder, then $i \in \Th$. Then for any $\vt \in [1,2]$, there holds
 \begin{equation}
 \label{lemma6.8-one}
 \iint_{\qfur \setminus \elam} |\vlh(z)|^{\vt} \ dz \apprle_{(n,\plog,\lamot,m_e)} \iint_{\qthrs \setminus \elam} |\tuh(z)|^{\vt} \ dz.
 \end{equation}
\end{lemma}
 \begin{proof}
 Since $\qfur \setminus \elam \subset \bigcup_{i =1}^{\infty} \htq_i$, using \descref{W2}{W2} and \descref{W5}{W5}, we infer that $\{ \htq_i\}_{i = 1}^{\infty}$ has finite overlap. This gives
 \begin{equation*}
 \begin{array}{rcl}
 \iint_{\qfur\setminus \elam} |\vlh(z)|^2 \ dz & \apprle & \sum_{i \in \NN} \iint_{Q_i} |\psi_j(z)|^2 |\tuh^j|^2 \ dz 
  \leq  \sum_{i \in \NN}  \iint_{2Q_i} |\tuh(z)|^2 \ dz \\
 & \apprle &  \iint_{\RR^{n+1} \setminus \elam} |\tuh(z)|^2 \ dz \overset{\eqref{cut_off_function}}{=} \iint_{\qthrs \setminus \elam} |\tuh(z)|^2 \ dz.
 \end{array}
 \end{equation*}

 \end{proof}

 \begin{lemma}
 \label{lemma6.8-1}
 Let $Q_i$ be a Whitney-type parabolic cylinder, then $i \in \Th$. Suppose that $i \in \Th_1$, we have
 \begin{equation}
 \label{lemma6.8-two}
 \fiint_{Q_i} |\vlh(z) - \tuh(z)| \ dz \apprle_{(n,\plog,\lamot,m_e)} \musym \min\left\{ \scalex{\la}{z_i} r_i, \scalex{\al_0}{z_0} \rho \right\} \la^{\frac{1}{p(z_i)}}.
 \end{equation}

 \end{lemma}
 \begin{proof}
 Since $i \in \Th_1$, using triangle inequality along with the fact that $\{\htq_i\}_{i \in \NN}$ has finite overlap, we get
 \begin{equation*}
 \begin{array}{rcl}
 \fiint_{Q_i} |\vlh(z) - \tuh(z)| \ dz & \leq & \sum_{j \in I_i} \fiint_{Q_i} \psi_j(z) \lbr \tuh(z) - \tuh^j \rbr \ dz \\
 & \apprle & \fiint_{\htq_i} |\tuh(z) - \avgs{\tuh}{\htq_i} | \ dz \\
 & \overset{\text{Lemma \ref{improved_est}}}{\apprle} & \lbr \frac{16\rho}{\rho_b - \rho_a} \rbr \min\left\{ \scalex{\al_0}{z_0}\rho, \scalex{\la}{z_i}r_i\right\} \la^{\frac{1}{p(z_i)}}.
 \end{array}
 \end{equation*}
This proves the lemma. 
 \end{proof}

 \begin{lemma}
 \label{lemma6.8-2}
 Let $Q_i$ be a Whitney-type parabolic cylinder, then $i \in \Th$. Then there holds
 \begin{multline}
 \label{lemma6.8-three}
 \iint_{\bfur \times S_1 } \left| \pa_t \vlh(z) \lbr \vlh(z) - \tuh(z)\rbr \right| dz \\
  \apprle_{(n,\plog,\lamot,m_e)}  \musym^2 \la |\RR^{n+1} \setminus \elam| + \frac{1}{s} \iint_{\qthrs} |\tuh(z)|^2 \, dz,
 \end{multline}
 where $s:= \scalet{\al_0}{z_0} (\rho_b - \rho_a) \rho$ is from \eqref{def_s}.
 \end{lemma}

\begin{proof}
We write the term on the left-hand side of \eqref{lemma6.8-three} as follows:
\begin{align*}
\nonumber \iint_{\bfur \times S_1} \left| \pa_t \vlh(z) \lbr \vlh(z) - \tuh(z)\rbr \right| dz & \leq  \sum_{i \in\Th_1}\iint_{Q_i } \abs{\pa_t \vlh(z) \lbr \vlh(z) - \tuh(z)\rbr} dz \\
\nonumber &  \quad + \sum_{i \in \Th_2}\iint_{Q_i} \abs{\pa_t \vlh(z) \lbr \vlh(z) - \tuh(z)\rbr} dz\\
 & := J_1 + J_2.
\end{align*}
We shall now estimate each of the terms as follows:
\begin{description}[leftmargin=*]
\item[Estimate for $J_1$:]  Since $i \in \Th_1$, making use of \descref{W2}{W2} along with \eqref{bound_time_vlh_one} and \eqref{lemma6.8-two}, we have
\begin{equation*}
J_1 \apprle \sum_{i \in \Th_1} \musym^2 \la^{\frac{p(z_i)-1}{p(z_i)}}  \la^{\frac{1}{p(z_i)}} |Q_i| \apprle \musym^2 \la |\RR^{n+1} \setminus \elam|.
\end{equation*}

\item[Estimate for $J_2$:] Using triangle inequality along with \eqref{bound_low_s}, for some $z \in 2Q_i$, there holds
\begin{equation}
\label{lm6.8-1.3}
|\pa_t \vlh(z)| \apprle \frac{1}{\scalet{\la}{z_i} r_i^2} \fiint_{\htq_i} |\tuh(\tz)| \, d\tz \overset{\eqref{bound_low_s}}{\apprle} \frac{1}{s} \fiint_{\htq_i} |\tuh(\tz)| \, d\tz. 
\end{equation}
Using the above estimate, we obtain
\begin{align*}
J_2 & \apprle \frac{1}{s}\sum_{i \in \Th_2} \lbr \fiint_{\htq_i} |\tuh(\tz)| \, d\tz \rbr \lbr \iint_{Q_i} |\vlh(\tz)| + |\tuh(z)| \ d\tz \rbr \\
& \apprle \frac{1}{s} \sum_{i \in \Th_2} |\htq_i| \lbr \fiint_{\htq_i} |\tuh(\tz)| \ d\tz\rbr^2 \\
& \apprle \frac{1}{s} \iint_{\qthrs} |\tuh(\tz)|^2 \, d\tz.
\end{align*}

\end{description}

Combining both the estimates proves the lemma. 

\end{proof}

Recall that $\htq_i = \hat{c}Q_i = Q_{\hat{c}r_i}^{\la} = Q_{\htr_i}^{\la}$ where we have set $\htr_i = \hat{c} r_i$ with $\hat{c}$ from \descref{W4}{W4}. 
\begin{corollary}
\label{corollary6.9}
Let $Q_i$ be a Whitney-type cylinder with $i \in \Th$. Then with $s:= \scalet{\al_0}{z_0} (\rho_b-\rho_a) \rho$, there holds
\begin{multline}
\label{lm6.9-one}
\iint_{2Q_i} \lbr |\nabla u| + 1 \rbr^{p(z) -1}  \lbr \frac{ \lsb{\chi}{\Th_1}} {\scalex{\al_0}{z_0} \rho} + \lsb{\chi}{\Th_2}\max\left\{\frac{1}{\scalex{\la}{z_i} r_i}, \frac{1}{\scalex{\al_0}{z_0} \rho} \right\}\rbr|\vlh(z)| \ dz \\
\hfill \apprle_{(n,\plog,\lamot,m_e)} \musym \la |2Q_i| + \frac{\lsb{\chi}{\Th_2} }{s} \iint_{\htq_i} |\tuh(z)|^2 \ dz.
\end{multline}
There also holds
\begin{equation}
\label{lm6.9-two}
\iint_{2Q_i} \lbr |\nabla u| + 1 \rbr^{p(z) -1}  |\nabla \vlh(z)| \, dz \apprle_{(n,\plog,\lamot,m_e)} \musym \la |2Q_i| + \frac{\lsb{\chi}{\Th_2} }{s} \iint_{\htq_i} |\tuh(z)|^2 \, dz.
\end{equation}
Here we have used the notation $\lsb{\chi}{\Th_1} = 1$ if and only if $i \in \Th_1$ and $\lsb{\chi}{\Th_1} = 0$ else. Similarly $\lsb{\chi}{\Th_2} =1$ if and only if $i \in \Th_2$ and $\lsb{\chi}{\Th_2} = 0$ else. 
\end{corollary}

\begin{proof}
Using \descref{W4}{W4} and Definition \ref{restriction_be_0} which restricts $1 < q < \frac{p^+-1}{p^-}$, we get
\begin{equation}
\label{lm6.9-1}
\fiint_{2Q_i} ( |\nabla u| + 1)^{p(z) - 1} \ dz  \leq  \lbr \fiint_{2Q_i} (|\nabla u|+1 )^{\frac{p(z)}{q}} \ dz \rbr^{\frac{q(p^+_{2Q_i} - 1)}{p^-_{2Q_i}}}  \overset{\eqref{elam}}{\apprle}  \la^{\frac{p^+_{2Q_i} - 1}{p^-_{2Q_i}}}.
\end{equation}
Let us prove each of the estimate as follows:

\begin{description}[leftmargin=*]
\item[Estimate \eqref{lm6.9-one} when $i \in \Th_1$:]  Since $i \in \Th_1$, we make use of Corollary \ref{lemma3.6} along with \eqref{lm6.9-1} to get
\begin{equation*}
\begin{array}{rcl}
\iint_{2Q_i} \lbr |\nabla u| + 1 \rbr^{p(z) -1}   \frac{ 1} {\scalex{\al_0}{z_0} \rho} |\vlh(z)| \ dz & \overset{\text{Corollary \ref{lemma3.6}}}{\leq} & \la^{\frac{1}{p(z_i)}} \iint_{2Q_i} \lbr |\nabla u| + 1 \rbr^{p(z) -1}  \ dz \\
& \overset{\eqref{lm6.9-1}}{\apprle} & |Q_i| \la^{\frac{1}{p(z_i)}} \la^{\frac{p^+_{2Q_i} - 1}{p^-_{2Q_i}}} \\
& \overset{\eqref{2.2.28-1}}{\apprle} & \musym \la  |Q_i|.
\end{array}
\end{equation*}

\item[Estimate \eqref{lm6.9-one} when $i \in \Th_2$:] In the case $i \in \Th_2$ with $\scalex{\al_0}{z_0} \rho \leq \scalex{\la}{z_i}r_i$, it follows from \eqref{lemma6.7-1_est} that
\begin{equation*}
\begin{array}{rcl}
\iint_{2Q_i} \lbr |\nabla u| + 1 \rbr^{p(z) -1}   \frac{ 1} {\scalex{\al_0}{z_0} \rho} |\vlh(z)| \ dz & \overset{\eqref{lemma6.7-1_est}}{\leq} & \la^{\frac{1}{p(z_i)}} \iint_{2Q_i} \lbr |\nabla u| + 1 \rbr^{p(z) -1}  \ dz \\
& \overset{\eqref{lm6.9-1}}{\apprle} & |Q_i| \la^{\frac{1}{p(z_i)}} \la^{\frac{p^+_{2Q_i} - 1}{p^-_{2Q_i}}} \\
& \overset{\eqref{2.2.28-1}}{\apprle} & \musym \la  |Q_i|.
\end{array}
\end{equation*}

In the case $i \in \Th_2$ with $\scalex{\al_0}{z_0} \rho \geq \scalex{\la}{z_i}r_i$, we make use of \eqref{bound_6.7-3-1} along with \eqref{lm6.9-1} to have
\begin{equation*}
\begin{array}{rcl}
\iint_{2Q_i} \lbr |\nabla u| + 1 \rbr^{p(z) -1}   \frac{ 1} {\scalex{\la}{z_i} r_i} |\vlh(z)| \, dz & \apprle & \lbr \la^{\frac{1}{p(z_i)}} + \frac{\la^{\frac{1-p(z_i)}{p(z_i)}}}{s} \fiint_{2Q_i} |\tuh(z)|^2 \, dz \rbr \la^{\frac{p^+_{2Q_i} - 1}{p^-_{2Q_i}}} |Q_i| \\
& \overset{\eqref{2.2.28-1}}{\apprle} & \musym \la  |Q_i| + \frac{1}{s} \fiint_{\htq_i} |\tuh(z)|^2 \, dz.
\end{array}
\end{equation*}

\item[Estimate \eqref{lm6.9-two}:] In the case $i \in \Th_1$, we make use of \eqref{bound+6.31_two} and in the case $i\in \Th_2$, we make use of \eqref{bound_6.7-3-2} to obtain
\begin{equation*}
\begin{array}{rcl}
\iint_{2Q_i} \lbr |\nabla u| + 1 \rbr^{p(z) -1}  |\nabla \vlh(z)| \, dz  & \apprle  & \lbr \la^{\frac{1}{p(z_i)}} + \frac{\la^{\frac{1-p(z_i)}{p(z_i)}}}{s} \fiint_{2Q_i} |\tuh(z)|^2 \, dz \rbr \la^{\frac{p^+_{2Q_i} - 1}{p^-_{2Q_i}}} |Q_i| \\
& \overset{\eqref{2.2.28-1}}{\apprle} & \musym \la  |Q_i| + \frac{1}{s} \fiint_{\htq_i} |\tuh(z)|^2 \, dz.
\end{array}
\end{equation*}
\end{description}

This completes the proof of the corollary. 
\end{proof}

\begin{corollary}
\label{corollary6.10}
Let $Q_i$ be a Whitney-type cylinder with $i \in \Th$ and furthermore suppose  that the following restriction  $\scalex{\la}{z_i}\htr_i \leq \frac{\scalex{\al_0}{z_0}(\rho_b-\rho_a)}{15}$ is true, then for any $\de \in (0,1)$, there holds
\begin{multline*}
\iint_{2Q_i} \lbr |\nabla u| + 1 \rbr^{p(z) -1} \left[ \frac{1}{\scalex{\la}{z_i} r_i}|\vlh(z)| + |\nabla \vlh(z)| \right] dz \\
\apprle_{(n,\plog,\lamot,m_e)} \musym \frac{\la}{\de} |2Q_i| +  \lsb{\chi}{\Th_1} \de |B_i| |\tuh^i|^2 + \frac{\lsb{\chi}{\Th_2} }{s} \iint_{\htq_i} |\tuh(z)|^2 \, dz.
\end{multline*}
\end{corollary}

\begin{proof}
Note that the bound for the term containing $|\nabla \vlh|$ is the same as \eqref{lm6.9-two} and hence we only have to consider the term containing $|\vlh|$. 
It follows from the hypothesis that
\begin{equation*}\label{cor_6.9-one}
\scalex{\la}{z_i}\htr_i \leq  \frac{\scalex{\al_0}{z_0}(\rho_b-\rho_a)}{15} \leq \scalex{\al_0}{z_0} \rho. 
\end{equation*}
Let us split the proof into two cases as follows:
\begin{description}[leftmargin=*]
\item[Case $i \in \Th_1$:] In this case, we make use of \eqref{bound+6.31} along with \eqref{lm6.9-1} to have
\begin{equation*}
\begin{array}{rcl}
& & \hspace{-2cm} \iint_{2Q_i} \lbr |\nabla u| + 1 \rbr^{p(z) -1}   \frac{ |\vlh(z)|} {\scalex{\la}{z_i} r_i} \ dz \\
& \apprle & \lbr \lbr \frac{16\rho}{\rho_b - \rho_a} \rbr \frac{\la^{\frac{1}{p(z_i)}}}{\de} +  \frac{\de}{\lbr \scalex{\la}{z_i}r_i\rbr^2 \la^{\frac{1}{p(z_i)}}} |\tuh^i|^2 \rbr\la^{\frac{p^+_{2Q_i} - 1}{p^-_{2Q_i}}} |Q_i| \\
& \overset{\eqref{2.2.28-1}}{\apprle} & \lbr \frac{16\rho}{\rho_b - \rho_a} \rbr  \frac{\la}{\de} |2Q_i| + \frac{\de}{\lbr \scalex{\la}{z_i}r_i\rbr^2} \la^{\frac{p^+_{2Q_i} - 1}{p^-_{2Q_i}}-\frac{1}{p(z_i)}} |Q_i| |\tuh^i|^2\\
& \apprle  & \lbr \frac{16\rho}{\rho_b - \rho_a} \rbr  \frac{\la}{\de} |2Q_i| + \de |B_i| \la^{-1+\frac{p^+_{2Q_i} - 1}{p^-_{2Q_i}}+\frac{1}{p(z_i)}} |\tuh^i|^2\\
& \overset{\eqref{2.2.28-1}}{\apprle}  & \lbr \frac{16\rho}{\rho_b - \rho_a} \rbr  \frac{\la}{\de} |2Q_i| + \de |B_i|  |\tuh^i|^2.
\end{array}
\end{equation*}

\item[Case $i \in \Th_2$:] In this case, we make use of \eqref{bound_6.7-3-1} along with \eqref{lm6.9-1} to get
\begin{equation*}
\begin{array}{rcl}
\iint_{2Q_i} \lbr |\nabla u| + 1 \rbr^{p(z) -1}   \frac{ |\vlh(z)|} {\scalex{\la}{z_i} r_i} \, dz & \apprle & \lbr \la^{\frac{1}{p(z_i)}} + \frac{\la^{\frac{1-p(z_i)}{p(z_i)}}}{s} \fiint_{2Q_i} |\tuh(z)|^2 \, dz \rbr \la^{\frac{p^+_{2Q_i} - 1}{p^-_{2Q_i}}} |Q_i| \\
& \overset{\eqref{2.2.28-1}}{\apprle} & \musym \frac{\la}{\de}  |Q_i| + \frac{1}{s} \fiint_{\htq_i} |\tuh(z)|^2 \ dz.
% & {\apprle} & \musym \frac{\la}{\de}  |Q_i| + \frac{1}{s} \fiint_{\htq_i} |\tuh(z)|^2 \, dz.
\end{array}
\end{equation*}
\end{description}
\end{proof}

\begin{corollary}
\label{cor6.10}
Under the assumptions of Corollary \ref{corollary6.9} and Corollary \ref{corollary6.10}, there holds
\begin{equation*}
\begin{array}{l}
\iint_{\qfur \setminus \elam} ( |\nabla u| + 1)^{p(z) -1} \left[ \frac{1}{\scalex{\al_0}{z_0} (\rho_b-\rho_a)} |\vlh| + |\nabla \vlh| \right] dz\\
\qquad \qquad \qquad   \apprle_{(n,\plog,\lamot,m_e)} \musym \sum_{i=1}^{\infty} \iint_{2Q_i} (|\nabla u| + 1)^{p(z) -1} \left[ \frac{1}{\scalex{\al_0}{z_0} \rho} |\vlh| + |\nabla \vlh| \right] dz  \\
\qquad \qquad \qquad   \apprle_{(n,\plog,\lamot,m_e)} \musym \la |\RR^{n+1} \setminus \elam| + \frac{1}{s} \iint_{\qfur \setminus \elam} |\tuh|^2 \, dz.
\end{array}
\end{equation*}

\end{corollary}

\subsection{Crucial estimates for the test function}
In this subsection, we shall prove three crucial estimates that will be necessary in Section \ref{section_five}. Note that by the time these estimates are applied, we would have taken $h \searrow 0$ in the Steklov average. 

\begin{lemma}
\label{cruc_1}
Let $\la \geq c_e \al_0$, then for any $i \in \Th_1$, $\de \in (0,1]$ and a.e. $t \in S_1$, there holds
\begin{equation}
\label{cruc_est_1}
\abs{\int_{\bfur} \lbr \tu(x,t) - \tu^i \rbr \vl(x,t) \psi_i(x,t)\ dx} \apprle_{(n,\plog,\lamot,m_e)}  \musym^2 \lbr \frac{\la}{\de} |Q_i| + \de \lsb{\chi}{3} |B_i| |\tu^i|^2\rbr,
\end{equation}
where $\lsb{\chi}{3} = 1$ if and only if $\scalex{\la}{z_i}\htr_i \leq \frac{\scalex{\al_0}{z_0}(\rho_b-\rho_a)}{15}$ and $\lsb{\chi}{3} = 0$ else. 

In the case  $i \in \Th_2$, for a.e. $t \in S_1$, there holds
\begin{equation}
\label{cruc_est_2}
\abs{\int_{\bfur} \tu(x,t)\vl(x,t) \psi_i(x,t)\ dx} \apprle_{(n,\plog,\lamot,m_e)}  \musym \la |Q_i| + \frac{1}{s} \iint_{\htq_i} |u|^2 \lsb{\chi}{\qfur}  \ dz.
\end{equation}
\end{lemma}

%%%%%%%%%%%%%%%%%%%%%%%%%%%%%%%%%%%%%%%%%%%%%%%%%%%%%%%%%%%%%%%%%%%%%%%%%%%%%%%%%%%%%%%%%%%%%%%%%%%%%%%%%%%%%%%%%%%%%%%%%%%%%%%%%%%%%%%%%%%%%%%%%%%%%%%%%%%%%%%%%%%%%%%%%%%%%%%%%%

\begin{proof}
Let us fix any $t \in S_1$,  $i \in \Th$  and take $\eta(y) \zeta(\tau) \om_i(y,\tau) \vlh(y,\tau)$ as a test function in \eqref{very_weak_solution}. Further integrating the resulting expression over $ \left(t_i - \scalet{\la}{z_i}r_i^2 , t\right)$ along with making use of  the fact that $\psi_i(y,t_i - \scalet{\la}{z_i} r_i^2) = 0$, we get for  any $a\in \RR$, the equality
% \begin{equation}
%  \label{3.122}
% \int_{t_i - \ga \left(\frac34 r_i\right)^2}^t \int_{\Om} \pa_t [u]_h \cdot \eta \zeta \om_i \vlh  + \iprod{[\aa(y,\tau,\nabla u)]_h}{\nabla (\eta \zeta \om_i \vlh)} \, dy \, d\tau = 0 
% \end{equation}
% 
% Using \eqref{3.122} along with the observation that $\spt(\om_i) \subset \frac34Q_i$,  for any constant $a \in \RR$, we obtain the following:
\begin{equation}
 \label{3.123}
 \begin{array}{ll}
  \int_{\Om} \lbr[(]  (\tuh - a)  \psi_i \vlh \rbr (y,t) \ dy & = \int_{t_i - \scalet{\la}{z_i} r_i^2}^t \int_{\Om} \pa_t \left(  (\tuh - a) \psi_i \vlh \right) (y,\tau) \, dy \, d\tau \\
%   & = \int_{t_i - \ga \left(\frac34 r_i\right)^2}^t \int_{\Om} \pa_t \left( [u]_h \eta \zeta \om_i \vlh  - a \om_i \vlh \right) (y,\tau) \, dy \, d\tau \\
%   & = \int_{t_i - \ga \left(\frac34 r_i\right)^2}^t \int_{\Om} \pa_t  [u]_h \lbr \eta \zeta \om_i \vlh \right) (y,\tau) \, dy \, d\tau + \\
%   & \qquad + \int_{t_i - \ga \left(\frac34 r_i\right)^2}^t \int_{\Om}   [u]_h  \ \pa_t\lbr \eta \zeta \om_i \vlh \right) (y,\tau) \, dy \, d\tau  +\\
%   & \qquad - \int_{t_i - \ga \left(\frac34 r_i\right)^2}^t \int_{\Om} a \pa_t \lbr \om_i \vlh\rbr \, dy \, d\tau \\
  & = \int_{t_i - \scalet{\la}{z_i} r_i^2	}^t \int_{\Om} \iprod{[\aa(y,\tau,\nabla u)]_h}{\nabla (\eta \zeta \psi_i \vlh)} \, dy \, d\tau  \\
  & \qquad +\  \int_{t_i - \scalet{\la}{z_i} r_i^2}^t \int_{\Om}   [u]_h  \ \pa_t\lbr \eta \zeta \psi_i \vlh \right) (y,\tau) \, dy \, d\tau  \\
  & \qquad -\  \int_{t_i - \scalet{\la}{z_i} r_i^2}^t \int_{\Om} a \pa_t \lbr \psi_i \vlh\rbr \, dy \, d\tau.
 \end{array}
\end{equation}

We can estimate $|\nabla (\eta \zeta \psi_i \vl)|$ using the chain rule, \eqref{cut_off_function} and Lemma \ref{partition_unity}, to get
\begin{equation}
 \label{3.126}
 \begin{array}{ll}
 |\nabla (\eta \zeta \psi_i \vl)| %&\apprle |\nabla \eta| |\zeta \om_i \vl| + |\nabla \om_i| |\zeta \eta \vl| + |\nabla \vl| |\eta \zeta \om_i| \\
 & \apprle \frac{1}{\scalex{\al_0}{z_0}(\rho_b-\rho_a)} |\vl| + \frac{1}{\scalex{\la}{z_i}r_i} |\vl| + |\nabla \vl|.
 \end{array}
\end{equation}
Similarly, we can estimate $\left|\pa_t\lbr \eta \zeta \psi_i \vl \right)\right|$ using the chain rule, \eqref{cut_off_function}, \eqref{def_s}  and Lemma \ref{partition_unity}, to get
\begin{equation}\label{3.127}  \begin{array}{ll}
 \left|\pa_t\lbr \eta \zeta \psi_i \vl \right)\right|  &\apprle \frac{\lsb{\chi}{\Th_2}}{\scalet{\al_0}{z_0} (\rho_b^2-\rho_a^2)} |\vl| + \frac{1}{\scalet{\la}{z_i} r_i^2} |\vl| + |\pa_t \vl| \\
 & \apprle \frac{\lsb{\chi}{\Th_2}}{s} |\vl| + \frac{1}{\scalet{\la}{z_i} r_i^2} |\vl| + |\pa_t \vl|, 
 \end{array}\end{equation}
 \begin{align}
 \left| \pa_t \lbr \psi_i \vl\rbr\right| & \apprle  \frac{1}{\scalet{\la}{z_i} r_i^2} |\vl| + |\pa_t \vl|,\label{3.128}
\end{align}
where we have set $\lsb{\chi}{\Th_2} = 1$ if  $i \in \Th_2$ and $\lsb{\chi}{\Th_2}=0$ otherwise. 
Let us now prove each of the assertions of the lemma.
\begin{description}[leftmargin=*]
 \item[Proof of \eqref{cruc_est_1}:] Note that $i \in \Th_1$, which implies $\zeta(t) \equiv 1$ on $\hat{c}I_i$, thus taking $a=\tuh^i$ in the \eqref{3.123} followed by letting $h \searrow 0$ and making use of \eqref{3.126} and \eqref{abounded},  we get
%  combining with \eqref{3.126}, \eqref{3.127} and \eqref{3.128} to get the following estimate:
 \begin{equation}
  \label{first_1}
  \begin{array}{ll}
   \left| \int_{8B} \lbr[(] (\tu - \tu^i) \om_i \vl \rbr (y,t) \ dy \right| & \apprle J_1 + J_2 + J_3,
%    \iint_{8Q} \lbr |\nabla u|^{p-1}+|\th|^{p-1} \rbr \lbr \frac{1}{\min\{\rho, r_i\}}|\vl| + |\nabla \vl| \rbr  \lsb{\chi}{\frac34Q_i\cap 8Q} \, dy \, d\tau+ \\
%    & \qquad + \iint_{8Q} |\tu-\tu^i| | \pa_t (\om_i \vl)|  \lsb{\chi}{\frac34Q_i\cap 8Q} \, dy \, d\tau\\
%    & \qquad + |\tu_{4Q_i}| \iint_{8Q} \lbr \frac{1}{\ga r_i^2} |\vl| + |\pa_t \vl| \rbr \lsb{\chi}{\frac34Q_i\cap 8Q}\, dy \, d\tau \\
%    &:= I_1 + I_2 + I_3.
  \end{array}
 \end{equation}
 where we have set 
 \begin{align}
  J_1& := \frac{1}{\min\{\scalex{\al_0}{z_0}\rho, \scalex{\la}{z_i}r_i\}} \iint_{\qfur} \lbr |\nabla u|+1\rbr^{p(z)-1}  |\vl|   \lsb{\chi}{2Q_i\cap \qfur} \ dz, \nonumber \\
  J_2& :=  \iint_{\qfur} \lbr |\nabla u|+1\rbr^{p(z)-1}  |\nabla \vl|   \lsb{\chi}{2Q_i\cap \qfur} \ dz,\nonumber \\
  J_3&:= \iint_{\qfur} |\tu-\tu^i| | \pa_t (\psi_i \vl)|  \lsb{\chi}{2Q_i\cap \qfur} \ dz. \nonumber%\label{bound_J_3_1}
 \end{align}

 Let us now estimate each of the terms as follows: %Let us denote $\lsb{\chi}{i_i \leq \rho}$ to be $1$ when $r_i \leq \rho$ and $0$ otherwise. 
 \begin{description}
  \item[Bound for $J_1$:] Since $i \in \Th_1$, we can directly use \eqref{lm6.9-one} if $\scalex{\al_0}{z_0} \rho \leq \scalex{\la}{z_i} r_i$ or Corollary \ref{corollary6.10} when $\scalex{\al_0}{z_0} \rho \geq \scalex{\la}{z_i} r_i$ to get for any $\de \in (0,1]$, the bound
\begin{equation}
 \label{bound_I_1}
 J_1 \apprle \musym \frac{\la}{\de} |\htq_i| +  \de |\htb_i| |\tu^i|^2.
\end{equation}
% where we have set $\lsb{\chi}{r_i \leq \rho} = 1$ if  $r_i \leq \rho$ and $\lsb{\chi}{r_i \leq \rho}=0$ else. 

  \item[Bound for $J_2$:] In this case, we can directly use \eqref{lm6.9-two} to get for any $\de \in (0,1]$, the bound
  \begin{equation}
   \label{bound_I_22}
%    \begin{array}{ll}
    J_2 %&:= \iint_{8Q} |\nabla u|^{p-1} |\nabla \vl| \lsb{\chi}{\frac34Q_i \cap 8Q} \, dy \, d\tau \\
%      \apprle \frac{\la}{\ve} \iint_{4Q_i}|\nabla u|^{p-1}+ |\tht|^{p-1} \, dy \, d\tau 
   \apprle \frac{\la}{\de} |\htq_i|.
%    \end{array}
  \end{equation}
%   To obtain \redr{a}, we made use of \eqref{lemma3.10_bound8} and to obtain \redr{b}, we made use of \whit{3} of Lemma \ref{whitney_decomposition}. 
% \hrule \hrule \hrule 
  \item[Bound for $J_3$:] 
  In order to estimate $J_3$, we further split the proof into the following two subcases:
  \begin{description}[leftmargin=*]
  \item[Subcase $ \scalex{\la}{z_i} 15\htr_i \geq \scalex{\al_0}{z_0} {(\rho_b-\rho_a)}$:] 
  Let us first obtain the following bound:
  \begin{equation}
  \label{6.74}
%   \begin{array}{rcl}
  \iint_{\qfur} |\tu(z) - \tu^i| \ dz  \apprle  |2Q_i| \fiint_{2Q_i} |\tu(z)| \ dz 
   \overset{\eqref{lemma3.6_pre_one}}{\apprle}  \scalex{\al_0}{z_0} \rho \la^{\frac{1}{p(z_i)}}|\htq_i|.
%   \end{array}
  \end{equation}
  
  Recall that $\htr_i = \hat{c} r_i$ where $\hat{c}$ is from \descref{W4}{W4}. In this case, we make use of Lemma \ref{lemma3.6}, \eqref{bound_time_vlh_one} and \eqref{3.128} to get
  \begin{equation}
  \label{6.75}
  \begin{array}{rcl}
     \left| \pa_t \lbr \psi_i \vl\rbr\right| & \apprle & \frac{1}{\scalet{\la}{z_i} r_i^2} \scalex{\al_0}{z_0} \rho \la^{\frac{1}{p(z_i)}} + \musym \frac{\la^{\frac{p(z_i)-1}{p(z_i)}}}{\scalex{\la}{z_i} r_i} \\
     & = & \musym \frac{\scalex{\al_0}{z_0} (\rho_b-\rho_a)}{\scalet{\la}{z_i}  \lbr \scalex{\la}{z_i}r_i\rbr^2}  \la^{\frac{1}{p(z_i)}} \lbr \scalex{\la}{z_i} \rbr^2 \\
     && +\musym \la^{\frac{p(z_i)-1}{p(z_i)}}\frac{1}{\scalex{\la}{z_i} r_i} \\
     & \overset{\text{hypothesis}}{\apprle} & \musym \la^{\frac{p(z_i)-1}{p(z_i)}}\frac{1}{\scalex{\la}{z_i} r_i}.
  \end{array}
  \end{equation}

  Combining \eqref{6.74} and \eqref{6.75} and using the hypothesis $ \scalex{\la}{z_i} 15\htr_i \geq \scalex{\al_0}{z_0} {(\rho_b-\rho_a)}$, we get
  \begin{equation}
  \label{6.76}
  J_3 \apprle \musym \frac{\la}{\de} |\htq_i|.
  \end{equation}

  \item[Subcase $ \scalex{\la}{z_i} 15\htr_i \leq \scalex{\al_0}{z_0} (\rho_b-\rho_a)$:]  Using Lemma \ref{lemma_crucial_1} applied with $\mu \in C_c^{\infty}(2Q_i)$ satisfying $|\nabla \mu| \apprle \frac{1}{\lbr \scalex{\la}{z_i}r_i \rbr^{n+1}}$ and $\vt =1$, we get
  \begin{equation}
  \label{6.77}
  \begin{array}{rcl}
  \fiint_{2Q_i} |\tu(z) - \tu^i| \, dz & \apprle & \scalex{\la}{z_i} r_i \fiint_{2Q_i} |\nabla \tu| \, dz + \sup_{2I_i} \abs{\avgs{\tu}{\mu}(t_2) - \avgs{\tu}{\mu}(t_1) }.
  \end{array}
  \end{equation}
  The first term on the right-hand side of \eqref{6.76} can be estimated as follows:
  \begin{equation}
  \label{6.78}
  \begin{array}{rcl}
   \fiint_{2Q_i} |\nabla \tu| \, dz  & \apprle & \musym  \lbr \fiint_{2Q_i} \left[ |\nabla u| + \abs{\frac{u}{\scalex{\al_0}{z_0} \rho}}+1 \right]^{\frac{p(z)}{q}} dz\rbr^{\frac{q}{p^-_{2Q_i}}} \\
   & \overset{\eqref{elam}}{\apprle} & \musym \la^{\frac{1}{p^-_{2Q_i}}} \overset{\eqref{2.2.28-1}}{\apprle} \musym \la^{\frac{1}{p(z_i)}}.
  \end{array}
  \end{equation}
  \end{description}

  The second term on the right-hand side of \eqref{6.76} can be estimated using Lemma \ref{lemma_crucial_2} to get
  \begin{equation}
  \label{6.79}
  \begin{array}{rcl}
  \abs{\avgs{\tu}{\mu}(t_2) - \avgs{\tu}{\mu}(t_1) } & \apprle & \frac{|Q_i|}{\lbr \scalex{\la}{z_i} r_i\rbr^{n+1}} \fiint_{2Q_i} (|\nabla u| + 1 )^{p(z)-1} \, dz \\
  & \overset{\eqref{elam}}{\apprle} &\frac{\scalet{\la}{z_i} r_i^2}{\scalex{\la}{z_i} r_i} \la^{\frac{p^+_{2Q_i}-1}{p^-_{2Q_i}}} \overset{\eqref{2.2.28-1}}{\apprle} \la^{\frac{d}{2}} r_i.
  \end{array}
  \end{equation}

  Combining \eqref{6.78},\eqref{6.79} and \eqref{6.77}, we get
  \begin{equation}
  \label{6.80}
  \iint_{2Q_i} |\tu(z) - \tu^i| \, dz \apprle \musym \scalex{\la}{z_i} r_i \la^{\frac{1}{p(z_i)}}|2Q_i|.
  \end{equation}

  In order to estimate $|\pa_t(\psi_i\vl)|$, we make use of Lemma \ref{partition_unity}, \eqref{bound+6.31} and \eqref{bound_time_vlh_one} to get for any $\de \in (0,1]$,
  \begin{equation}
  \label{6.81}
  \begin{array}{rcl}
  |\pa_t(\psi_i\vl)| & \apprle & \frac{1}{\scalet{\la}{z_i} r_i^2} \lbr \musym \frac{\la^{\frac{d}{2}}r_i}{\de} + \frac{\de}{\la^{\frac{d}{2}} r_i} |\tu^i|^2 \rbr + \musym \frac{\la^{1-\frac{d}{2}}}{\de r_i}.
  \end{array}
  \end{equation}
  
  Combining \eqref{6.80} and \eqref{6.81}, we get
  \begin{equation}
  \label{6.82}
  \begin{array}{rcl}
  J_3 & \apprle & \musym \la^{\frac{d}{2}} r_i |Q_i| \lbr\musym \frac{\la^{1-\frac{d}{2}}}{\de r_i} + \frac{\la^{1-\frac{3d}{2}}}{r_i^3}|\tu^i|^2 \rbr \\
  & \apprle & \musym^2 \frac{\la}{\de} |Q_i| + \de \frac{\la^{1-d}}{r_i^2} \scalet{\la}{z_i} r_i^2 |B_i| |\tu^i|^2 \\
  & = & \musym^2 \frac{\la}{\de} |Q_i| + \de |B_i| |\tu^i|^2.
  \end{array}
  \end{equation}

 \end{description}

 Combining the estimates \eqref{6.76} in the case $ \scalex{\la}{z_i} 15\htr_i \geq \scalex{\al_0}{z_0} (\rho_b-\rho_a)$ or \eqref{6.82} in the case $ \scalex{\la}{z_i} 15\htr_i \leq \scalex{\al_0}{z_0} (\rho_b-\rho_a)$  with \eqref{bound_I_1}, \eqref{bound_I_22} and \eqref{first_1}, we obtain the proof of  \eqref{cruc_est_1}. 
 
 \item[Proof of \eqref{cruc_est_2}:] Note that in this case, we have $\scalet{\la}{z_i} r_i^2 \apprge s$ from \eqref{bound_low_s}.  Setting $a =0$ in \eqref{3.123} along with making use of the the bounds  \eqref{3.126} and \eqref{3.127}, we get the following estimate:
 \begin{equation}
  \label{second_1}
  \begin{array}{ll}
   \left| \int_{8B} \lbr[(] \tu  \psi_i \vl \rbr (y,t) \ dy \right| & \apprle II_1 + II_2+II_3+II_4,
%    \iint_{8Q} |\nabla u|^{p-1} \lbr \frac{1}{\min\{\rho, r_i\}}|\vl| + |\nabla \vl| \rbr  \lsb{\chi}{\frac34Q_i\cap 8Q} \, dy \, d\tau+ \\
%    & \qquad + \iint_{8Q} |u| \lbr \frac1{\min\{s,\ga r_i^2\}}|\vl|  + |\pa_t \vl|\rbr  \lsb{\chi}{\frac34Q_i\cap 8Q} \, dy \, d\tau \\
%    &:= II_1 + II_2 + II_3 + II_4.
  \end{array}
 \end{equation}
where we have set 
\begin{equation*}
 \begin{array}{ll}
  II_1 & := \musym \frac{1}{\min\{\scalex{\al_0}{z_0}\rho, \scalex{\la}{z_i}r_i\}} \iint_{\qfur}\lbr  |\nabla u|+1\rbr^{p(z)-1} |\vl|   \lsb{\chi}{2Q_i\cap \qfur} \ dz, \\
  II_2 & :=  \iint_{\qfur} \lbr |\nabla u|+1\rbr^{p(z)-1}  |\nabla \vl|   \lsb{\chi}{2Q_i\cap \qfur} \ dz, \\
  II_3 & :=  \frac1{s} \iint_{\qfur} |u| |\vl|  \lsb{\chi}{2Q_i\cap \qfur} \, dy \, d\tau,\\
  II_4 & :=   \iint_{\qfur} |u| |\pa_t\vl|  \lsb{\chi}{2Q_i\cap \qfur} \, dy \, d\tau.
 \end{array}
\end{equation*}
We shall now estimate each of the terms as follows:
\begin{description}[leftmargin=*]
 \item[Bound for $II_1$:] Since $i \in \Th_2$, we can directly use \eqref{lm6.9-one} to bound this term to get
 \begin{equation}
  \label{II_1_1}
%    \begin{array}{ll}
    II_1 \apprle \musym \la |2Q_i| + \frac{\lsb{\chi}{\Th_2} }{s} \iint_{\htq_i} |\tuh(z)|^2 \ dz.
%    \end{array}
  \end{equation}

 \item[Bound for $II_2$:] To bound this term, we make use of \eqref{lm6.9-two} to get
 \begin{equation}
  \label{II_2_1}
  \begin{array}{ll}
   II_2 & \apprle \musym \la |2Q_i| + \frac{\lsb{\chi}{\Th_2} }{s} \iint_{\htq_i} |\tuh(z)|^2 \ dz.
%       & \overset{\text{\redr{b}}}{\apprle} |4Q_i| \la^{p-1} \lbr \la + \frac{1}{\la r_i^2} \fiint_{4Q_i} |\tu|^2 \, dy \, d\tau \rbr \\
%            & \apprle \la^p |4Q_i| + \frac{1}{s} \iint_{4Q_i} |u|^2 \lsb{\chi}{8Q}\, dy \, d\tau.
  \end{array}
 \end{equation}
%  To obtain \redr{a}, we made use of \eqref{lemma3.10_bound6} with $\ve=1$, to get \redr{b}, we made use of \whit{3} of Lemma \ref{whitney_decomposition} and finally to get \redr{c}, we used $\ga = \la^{2-p}$ along with $\ga r_i^2 \apprge s$ since $i \in \Th_2$. 

 \item[Bound for $II_3$:] For bounding this term, we  use  \descref{W4}{W4}, $|\tu|\leq|u| \lsb{\chi}{\qfur}$ and \eqref{lipschitz_function}  to  get
 \begin{equation}
  \label{II_3_1}
  \begin{array}{ll}
   II_3 & \apprle \frac{1}{s} |\htq_i|\lbr \fiint_{\htq_i} |u| \lsb{\chi}{\qfur} \ dz\rbr\lbr  \fiint_{\htq_i} |\tu| \ dz \rbr \\
%    & \apprle \frac{1}{s} |4Q_i|\lbr \fiint_{4Q_i} |u| \lsb{\chi}{8Q} \, dy \, d\tau\rbr\lbr  \fiint_{4Q_i} |u|\lsb{\chi}{8Q} \, dy \, d\tau \rbr \\
%    & = \frac{1}{s} |4Q_i|\lbr \fiint_{4Q_i} |u| \lsb{\chi}{8Q} \, dy \, d\tau\rbr^2 \\
   & \apprle \frac{1}{s} \iint_{\htq_i} |u|^2 \lsb{\chi}{\qfur} \ dz.
  \end{array}
 \end{equation}
%  Note that we have used the trivial bound $|\tu| \leq |u|$ above. 

 \item[Bound for $II_4$:] In this case, we make use of \descref{W4}{W4} along with $|\tu|\leq|u| \lsb{\chi}{\qfur}$, \eqref{bound_low_s} and \eqref{lipschitz_function}, to  get
 \begin{equation}
  \label{II_4_1}
  \begin{array}{ll}
   II_4 & \apprle  |\htq_i|\lbr \fiint_{\htq_i} |u| \lsb{\chi}{\qfur} \ dz\rbr\lbr \frac{1}{\scalet{\la}{z_i} r_i^2} \fiint_{\htq_i} |\tu| \ dz \rbr \\
%    & \apprle \frac{1}{s} |4Q_i|\lbr \fiint_{4Q_i} |u| \lsb{\chi}{8Q} \, dy \, d\tau\rbr\lbr  \fiint_{4Q_i} |u|\lsb{\chi}{8Q} \, dy \, d\tau \rbr \\
%    & = \frac{1}{s} |4Q_i|\lbr \fiint_{4Q_i} |u| \lsb{\chi}{8Q} \, dy \, d\tau\rbr^2 \\
   & \apprle \frac{1}{s} \iint_{\htq_i} |u|^2 \lsb{\chi}{\qfur} \ dz.
  \end{array}
 \end{equation}
% Note that we have used the trivial bound $|\tu| \leq |u|$ above. 
\end{description}
We finally combine the estimates \eqref{II_1_1}, \eqref{II_2_1}, \eqref{II_3_1}, \eqref{II_4_1} and \eqref{second_1} to  obtain \eqref{cruc_est_2}.
\end{description}
This completes the proof of the lemma. 
\end{proof}
%%%%%%%%%%%%%%%%%%%%%%%%%%%%%%%%%%%%%%%%%%%%%%%%%%%%%%%%%%%%%%%%%%%%%%%%%%%%%%%%%%%%%%%%%%%%%%%%%%%%%%%%%%%%%%%%%%%%%%%%%%%%%%%%%%%%%%%%%%%%%%%%%%%%%%%%%%%%%%%%%%%%%%%%%%%%%%%%%%
% \begin{lemma}
% \label{cruc_2}
% Let $\la \geq c_e \al_0$, then for any $i \in \Th_2$ and a.e. $t \in S_1$, there exists a constant $C = C{(\plog,\lamot,n)}$ such that there holds
% \begin{equation}
% \abs{\int_{\bfur} \tu(x,t)\vl(x,t) \psi_i(x,t)\ dx} \leq C \musym \la |Q_i| + \frac{C}{s} \iint_{\htq_i} |\tu|^2  \ dz.
% \end{equation}
% \end{lemma}

\begin{lemma}
\label{cruc_3}
Let $\la \geq c_e \al_0$, then for a.e. $t \in S_1$, there exists a constant $C = C_{(n,\plog,\lamot,m_e)}$ such that there holds
\begin{equation}
\label{cruc_est_3}
\int_{\bfur\setminus \elam(t)} \lbr |\tu|^2 - |\tu - \vl|^2 \rbr \ dx \geq C \musym^2 \lbr -\la |\RR^{n+1} \setminus \elam| - \frac{1}{s} \iint_{\qfur} |\tu|^2  \ dz\rbr.
\end{equation}
\end{lemma}
% \begin{proof}
% The proof of this Lemma is very similar to the proof of \cite[Lemma 3.22]{AB2} and \cite[(6.22) of Lemma 6.11]{li2017very} and hence will be omited. 
% \end{proof}
% \hrule \hrule
%%%%%%%%%%%%%%%%%%%%%%%%%%%%%%%%%%%%%%%%%%%%%%%%%%%%%%%%%%%%%%%%%%%%%%%%%%%%%%%%%%%%%%%%%%%%%%%%%%%%%%%%%%%%%%%%%%%%%%%%%%%%%%%%%%%%%%%%%%%%%%%%%%%%%%%%%%%%%%%%%%%%%%%%%%%%%%%%%%
\begin{proof}
 Let us fix any $t\in S_1$ and any point $x \in \qfur \setminus \elam(t)$.  Now define
 \begin{equation*}
  \tTh := \left\{  i \in \Th: \spt(\psi_i) \cap \bfur \times \{t\} \neq \emptyset, \  |\tu| + |\vl| \neq 0 \  \text{on}\ \spt(\psi_i)\cap (\bfur \times \{t\}) \right\}.
 \end{equation*}

 From \eqref{theta}, we see that if $\spt(\psi_i) \cap \qfur \times \{t\} \neq \emptyset$, then  $i \in \Th$.  If $i \neq \tTh$, then $\tu = \vl = 0$ on $\spt(\psi_i) \cap \bfur \times \{t\}$, which  implies
 \[
  \int_{\spt(\psi_i) \cap \bfur \times\{t\}} |\tu|^2 - |\tu - \vl|^2 \ dx = 0.
 \]
 Hence we only need to consider $i \in \tTh$.  We now decompose  $\tTh = \tTh_1 \cup \tTh_2$, where $\tTh_1 := \tTh \cap \Th_1$ and $\tTh_2 := \tTh \cap \Th_2$. Noting that $\sum_{i \in \tTh} \om_i(\cdot, t) \equiv 1$ on $\RR^n \cap \elam(t)$, we can rewrite the left-hand side of \eqref{cruc_est_3} as 
 \begin{equation}
  \label{3124}
  \int_{\bfur \setminus \elam(t)} (|\tu|^2 - |\tu - \vl|^2)(x,t) \ dx = J_1 + J_2,
 \end{equation}
where we have set
\begin{gather*}
 J_1:= \sum_{i \in \tTh_1} \int_{\bfur} \psi_i (|\tu|^2 - |\tu - \vl|^2) \ dx, \qquad 
%  \label{def_I} 
 J_2:= \sum_{i \in \tTh_2} \int_{\bfur} \psi_i (|\tu|^2 - |\tu - \vl|^2) \ dx. 
%  \label{def_II}.
\end{gather*}
We shall now estimate each of the terms as follows:
\begin{description}[leftmargin=*]
 \item[Estimate of $J_1$:] We can further rewrite this term as follows:
 \begin{equation}
  \label{I_1}
  \begin{array}{ll}
  J_1 & = \sum_{i \in \tTh_1} \int_{\bfur} \psi_i(z) \lbr |\tu^i|^2  + 2 \vl (\tu - \tu^i) \rbr \ dx - \sum_{i \in \tTh_1} \int_{\qfur} \psi_i(z) |\vl - \tu^i|^2 \ dx\\
  & := J_1^1 + J_1^2.
  \end{array}
 \end{equation}
 \begin{description}
  \item[Estimate of $J_1^1$:] Using \eqref{cruc_est_1}, we get
  \begin{equation}
   \label{I_1_1}
   J_1^1 \apprge \sum_{i \in \tTh_1} \int_{\bfur} \om_i(z)  |\tu^i|^2 \ dz - \de\musym^2 \sum_{i \in \tTh_1} \lsb{\chi}{3}|\htb_i| |\tu^i|^2 - \musym^2\sum_{i \in \tTh_1} \frac{\la}{\de} |\htq_i|.
  \end{equation}
  From \eqref{def_tuh}, we have $\tu^i = 0$ whenever $\spt(\psi_i) \cap \lbr \bthrs \cap \Om \times \RR\rbr \neq \emptyset$. Hence we only have to sum over all those $i \in \tTh_1$ for which $\spt(\psi_i) \subset \bthrs \cap \Om \times \RR$.  In this case, we make use of a suitable choice for $\de \in (0,1]$, and use \descref{W4}{W4}  to estimate \eqref{I_1_1} from below. We get
  \begin{equation}
 \label{bound_I1}
 J_1^1 \apprge -\musym^2 \la |\RR^{n+1} \setminus \elam|. 
\end{equation}

  \item[Estimate of $J_1^2$:] For any $x \in \qfur \setminus \elam(t)$, we have from Lemma \ref{partition_unity} that $\sum_{j} \psi_j(x,t) = 1$, which gives
  \begin{equation}
   \label{I_2_1}
   \begin{array}{ll}
    \psi_i(z) |\vl(z) - \tu^i|^2  
    & \apprle   \sum_{j \in I_i} |\psi_j(z)|^2  \lbr \tu^j - \tu^i\rbr^2 \\
    & \overset{\redlabel{3.104.a}{a}}{\apprle} \musym^2 \min\{ \scalex{\al_0}{z_0}\rho, \scalex{\la}{z_i} r_i\}^2 \la^{\frac{2}{p(z_i)}}.
   \end{array}
  \end{equation}
To obtain \redref{3.104.a}{a} above, we made use of Corollary \ref{corollary3.7} (recall $i \in \tTh_1 \subset \Th_1$) along with \descref{W3}{W3}.  Substituting \eqref{I_2_1} into the expression for $J_1^2$ and using $|Q_i| = |B_i| \times \scalet{\la}{z_i} r_i^2$, we get
\begin{equation}
 \label{bound_I_2}
J_1^2  \apprle \musym^2 \sum_{i \in \tTh_1} |\bfur \cap 2B_i| {\lbr \scalex{\la}{z_i} r_i\rbr^2} \la^{\frac{2}{p(z_i)}}  \apprle \musym^2\la |\RR^{n+1} \setminus \elam|.
\end{equation}
 \end{description}
Substituting \eqref{bound_I1} and \eqref{bound_I_2} into \eqref{I_1}, we get
\begin{equation}
 \label{bound_I}
 J_1 \apprge - \musym^2 \la |\RR^{n+1} \setminus \elam|.
\end{equation}

 \item[Estimate of $J_2$:] From the identity $|\tu|^2 - |\tu-\vl|^2 = 2\tu \vl - |\vl|^2$, we get
 \begin{equation}
  \label{II1}
  \begin{array}{ll}
  \sum_{i \in \tTh_2} \left| \int_{\bfur} \psi_i (|\tu|^2 - |\tu - \vl|^2) \ dx\right| & \apprle \sum_{i \in \tTh_2} \left| \int_{\bfur} \psi_i \tu \vl \ dx\right| + \sum_{i \in \tTh_2} \left| \int_{\bfur} \psi_i |\vl|^2 \ dx\right| \\
  & := J_2^1 + J_2^2.
  \end{array}
 \end{equation}
\begin{description}
 \item[Estimate for $J_2^1$:]  This term can be easily estimated using \eqref{cruc_est_2} followed by summing over $i \in \tTh_2$ and using \descref{W1}{W1}, to get
 \begin{equation}
  \label{II1_1}
  \begin{array}{ll}
   \sum_{i \in \tTh_2} \left| \int_{\bfur} \psi_i \tu \vl \ dx\right| & \apprle \musym \sum_{i \in \tTh_2}\lbr  \la |Q_i| + \frac{1}{s} \iint_{\htq_i} |u|^2 \lsb{\chi}{\qfur}\ dz\rbr \\
   & \apprle \musym \la |\RR^{n+1} \setminus \elam| + \frac{1}{s} \iint_{\qfur} |u|^2 \ dz.
%    & \apprle \la^p |\RR^{n+1} \setminus \elam| + \frac{1}{s} \iint_{8Q} |u|^2 \ dz.
  \end{array}
 \end{equation}

 \item[Estimate for $J_2^2$:] To estimate this term, we make use of \eqref{lipschitz_function} along with $|Q_i| = |B_i| \times \scalet{\la}{z_i} r_i^2$, \eqref{bound_low_s}, the bound $|\tu| \leq |u| \lsb{\chi}{\qfur}$ and \descref{W1}{W1}, to get
 \begin{equation}
  \label{II2_1}
  \begin{array}{ll}
   \sum_{i \in \tTh_2} \left| \int_{\bfur} \psi_i |\vl|^2 \ dx\right| & \apprle \sum_{i \in \tTh_2}  \lbr \fiint_{\htq_i} |\tu| \ dz \rbr^2 |\qfur \cap 2B_i| 
    \apprle \frac{1}{s} \iint_{\qfur} |u|^2 \ dz.
  \end{array}
 \end{equation}
\end{description}
 We combine \eqref{II1_1}, \eqref{II2_1} and \eqref{II1} to obtain
 \begin{equation}
  \label{bound_II}
  J_2\apprle \musym \la |\RR^{n+1} \setminus \elam| + \frac{1}{s} \iint_{\qfur} |u|^2 \ dz.
 \end{equation}
\end{description}
Thus, from \eqref{bound_I}, \eqref{bound_II} and \eqref{3124}, the proof of the lemma follows.
\end{proof}

%%%%%%%%%%%%%%%%%%%%%%%%%%%%%%%%%%%%%%%%%%%%%%%%%%%%%%%%%%%%%%%%%%%%%%%%%%%%%%%%%%%%%%%%%%%%%%%%%%%%%%%%%%%%%%%%%%%%%%%%%%%%%%%%%%%%%%%%%%%%%%%%%%%%%%%%%%%%%%%%%%%%%%%%%%%%%%%%%%

\subsection{Lipschitz regularity for the test function}

 We will now show that the function $\vlh$  constructed in \eqref{lipschitz_function} is Lipschitz continuous on $\bfur \times S_1$. To do this, we shall use the integral characterization of Lipschitz continuous functions obtained in \cite[Theorem 3.1]{Prato}. This technique for proving the Lipschitz regularity of \eqref{lipschitz_continuity} was first used in \cite{bogelein2013regularity}.
 \begin{lemma}[Lipschitz characterization]
 \label{deprato}
 Let $\tz \in \qfur$ and $r >0$ be given. Define the parabolic cylinder $Q_r(\tz) := B_r(\tx) \times (\tlt - r^2, \tlt+r^2)$, i.e., $Q_r(\tz) := \{ z \in \RR^{n+1}: d_p(z,\tz) \leq r\}$ where $d_p$ is as defined in \eqref{parabolic_metric}. Furthermore suppose that the  following expression is bounded independent of $\tz \in \bfur\times S_1$ and $r>0$
 \[
 I_r(\tz) := \frac{1}{\abs{\bfur \times S_1 \cap Q_r(\tz)}} \iint_{\bfur\times S_1 \cap Q_r(\tz)} \abs{\frac{\vlh(z) - \avgs{\vlh}{\bfur\times S_1\cap Q_r(\tz)}}{r}} \ dz < \infty,
 \]
then $\vlh \in C^{0,1}(\bfur \times S_1)$.
 \end{lemma}

 \begin{remark}
 From \eqref{exp_d} and the fact that $\la \geq 1$,  for any $z_1,z_2 \in \RR^{n+1}$ and any $z \in \RR^{n+1}$, we get
 \begin{equation}
 \label{equiv_dist}
 \begin{array}{rcl}
 d_p(z_1,z_2) & \overset{\eqref{parabolic_metric}}{:=} & \max\{|x_1-x_2|, \sqrt{|t_1-t_2|} \} \\
 & \leq&  \max\{\nscalex{\la}{z}|x_1-x_2|, \sqrt{\nscalet{\la}{z} |t_1-t_2|} \}\overset{\eqref{loc_parabolic_metric}}{ =:} d_z(z_1,z_2)\\
 & \leq &  \la^{{\frac{1}{p^-}-\frac{d}{2}}}\la^{\frac12-\frac{d}{2}}\max\{|x_1-x_2|, \sqrt{|t_1-t_2|} \} \leq C_{(\la,p^-,d)} d_p(z_1,z_2).
 \end{array}
 \end{equation}
 
 This shows that for any $z \in \RR^{n+1}$, we have $d_p \approx_{(\la,p^-,d)} d_z$. 
 \end{remark}

 In this subsection, we want to apply Lemma \ref{deprato}, hence we only need to ensure the constants involved are independent of $r>0$ and $\tz$ only. \emph{Only for this subsection, we will use the notation $o(1)$ to denote a constant which  can depend on $\la,\al_0,\plog,\lamot,n,\|\tuh\|_{L^1}, \|u\|_{L^1}$ but {\bf NOT} on  $r>0$ and the point $\tz$.}

 \begin{lemma}
 \label{lipschitz_continuity}
 Let $\la \geq c_e \al_0$ with $c_e$ from Lemma \ref{fix_ce}, then for any $\tz \in \bfur \times S_1$ and $r>0$, there exists a constant $C>0$ independent of $\tz$ and $r$ such that
\begin{equation}\label{da_pr}
 I_r(\tz) := \frac{1}{\abs{\bfur \times S_1 \cap Q_r(\tz)}} \iint_{\bfur\times S_1 \cap Q_r(\tz)} \abs{\frac{\vlh(z) - \avgs{\vlh}{\bfur\times S_1\cap Q_r(\tz)}}{r}} \ dz \leq C.
 \end{equation}
 In particular, this implies for any $z_1, z_2 \in \bfur \times S_1$, there exists a constant $K>0$ such that
 \[
 |\vlh(z_1) - \vlh(z_2)| \leq K d_p(z_1,z_2).
 \]
 \end{lemma}
\begin{proof}
In order to show \eqref{da_pr}, we will consider the following four cases:
\begin{gather}
2Q_r(\tz) \subset \bfve \times S_2 \setminus \elam, \label{case_1}\\
2Q_r(\tz) \cap \elam \neq \emptyset, \qquad 2Q_r(\tz) \subset \bfve \times S_2 \txt{and} r < \frac{1}{3} \la^{-\frac{1}{p^+} + \frac{d}{2}} (\rho_b-\rho_a),\label{case_2}\\
2Q_r(\tz) \cap \elam \neq \emptyset, \qquad 2Q_r(\tz) \subset \bfve \times S_2 \txt{and} r \geq \frac{1}{3} \la^{-\frac{1}{p^+} + \frac{d}{2}} (\rho_b-\rho_a),\label{case_3}\\
2Q_r(\tz) \setminus \bfve \times S_2 \neq \emptyset. \label{case_4}
\end{gather}
 and show \eqref{da_pr} holds in each of them. 
 
 From the fact that $\tz \in \bfur \times S_1$, it is easy to see that 
 \begin{equation}
 \label{meas_z_r_tz}
 |Q_r(\tz) \cap \bfur \times S_1| \geq c(n) r^{n+2}.
 \end{equation}

 \begin{description}[leftmargin=*]
 \item[Case \eqref{case_1}:] In this case, using \eqref{meas_z_r_tz}, we get
 \begin{equation}
 \label{6.107}
 \begin{array}{rcl}
 I_r(\tz) & \apprle & \frac{1}{r} \fiint_{Q_r(\tz) \cap \bfur \times S_1} \fiint_{Q_r(\tz) \cap \bfur \times S_1} |\vlh(\mfz_1)-\vlh(\mfz_2)| \ d\mfz_1 \ d\mfz_2 \\
 & \apprle &\sup_{z \in \bfur \times S_1} |\nabla \vlh(z)| + r |\pa_t \vlh(z)|.
 \end{array}
 \end{equation}

 In order to bound \eqref{6.107}, let us fix an $z \in \bfur \times S_1$, then $z \in Q_i$ for some $i \in \Th$.  Let $\mfz_i \in \elam$ be such that $d_{z_i} (z_i, \elam) = d_{z_i} (z_i,\mfz_i)$. Thus using the fact that $2Q_r(\tz) \cap \elam = \emptyset$, we get
 \begin{equation}
 \label{bound_r}
 r \leq d_p(z,\elam) \leq d_p(z,z_i) + d_p(z_i,\mfz_i) \overset{\eqref{equiv_dist}}{\leq} d_{z_i}(z,z_i) + d_{z_1}(z_i,\mfz_i) \leq 5\hat{c} r_i.
 \end{equation}

 If $i \in \Th_1$, then we can use \eqref{bound+6.31_two} and \eqref{bound_time_vlh_one} to get
 \begin{equation*}
 \label{6.108}
 \begin{array}{rcl}
 |\nabla \vlh(z)| + r |\pa_t\vlh(z)| & \apprle & \musym \la^{\frac{1}{p(z_i)}} + r \musym \frac{\la^{1-\frac{d}{2}}}{r_i} \\
 & \overset{\eqref{bound_r}}{\apprle}& \musym \la^{\frac{1}{p^-}} + \musym \la^{1-\frac{d}{2}} \leq o(1).
 \end{array}
 \end{equation*}
 
 If $i \in \Th_2$, using \eqref{bound_low_s}, we see that 
 \begin{equation*}
 \label{6.110}
 |\htq_i| = r_i^{n+2} \scalexn{\la}{z_i} \scalet{\la}{z_i} = \lbr \scalet{\la}{z_i} r_i^2\rbr^{\frac{n+2}{2}} \la^{\frac{n}{2} - \frac{n}{p(z_i)}} \overset{\eqref{bound_low_s}}{\apprge} s^{\frac{n+2}{2}} \la^{\frac{n}{2} - \frac{n}{p^+}}.
 \end{equation*}

 then we will use \eqref{bound_6.7-3-2} and \eqref{lm6.8-1.3} along  to get
 \begin{equation*}
 \label{6.111}
 \begin{array}{rcl}
 |\nabla \vlh(z)| + r |\pa_t\vlh(z)| & \apprle & \la^{\frac{1}{p^-}} + \frac{\la^{\frac{1}{p^-} -1}}{s|\htq_i|} \iint_{\qfur}|\tuh|^2 \ dz  + r \frac{1}{s |\htq_i|} \iint_{\qfur} |\tuh| \ dz\\
 & \apprle & C_{(\la,\plog,s,n, \|u\|_{L^2})}  +C_{(s,n,\plog)} \frac{r}{\scalexn{\la}{z_i}\scalet{\la}{z_i} r_i^{n+2}} \|u\|_{L^1}\\
 & \overset{\eqref{bound_r},\eqref{bound_low_s}}{\apprle} & o(1).
 \end{array}
 \end{equation*}

 \item[Case \eqref{case_2}:] Noting that \eqref{meas_z_r_tz} must also hold in this case, we apply triangle inequality and estimate $I_r(\tz)$ as follows:
 \begin{equation}
  \label{6.112}
   I_r(z) \leq 2J_1 + J_2,
 \end{equation}
 where we have set
 \begin{equation*}
 \label{def_J_1_2}
 \begin{array}{rcl}
 J_1&:=& \fiint_{Q_r(\tz) \cap \bfur \times S_1} \left| \frac{\vlh(z) - \tuh(z)}{r}\right| dz, \\
 J_2 &:=& \fiint_{Q_r(\tz) \cap \bfur \times S_1} \left| \frac{\tuh(z) - \avgs{\tuh}{Q_r(\tz) \cap \bfur \times S_1}}{r}\right|dz.
 \end{array}
 \end{equation*}

 \begin{description}[leftmargin=*]
 \item[Estimate for $J_1$:] If $2Q_r(\tz) \subset \elam$, then $\vlh = \tuh$ which implies $J_1=0$. Hence without loss of generality,  we can  assume $2Q_r(\tz) \cap \elam^c \neq \emptyset$. 
Using the construction of \eqref{lipschitz_function}, we get
\begin{equation}
 \label{3.81.1}
 \begin{array}{ll}
  J_1   & \apprle \sum_{i\in \Th} \frac{1}{|Q_r(\tz) \cap \bfur \times S_1|} \iint_{Q_r(\tz) \cap \bfur \times S_1 \cap 2Q_i} \left| \frac{\tuh(z) - \tuh^i}{r}\right| dz.
 \end{array}
\end{equation}

Let us fix an $i \in \Th$ and take two points $\mfz_1 \in 2Q_i \cap Q_r(\tz)$ and $\mfz_2 \in \elam \cap 2Q_r(\tz)$. Then from \eqref{equiv_dist}, we have
\[
r_i \leq \frac{1}{\hat{c}} d_{z_i} (z_i, \mfz_2) \leq \frac{1}{\hat{c}} \left[ d_{z_i}(z_i, \mfz_1) + d_{z_i} (\mfz_1, \mfz_2)\right] \leq \frac{1}{\hat{c}} \left[ 2r_i + \la^{\frac{1}{p^-} - d + \frac12}d_p(\mfz_1,\mfz_2) \right].
\]
Since $\hat{c} \geq 9$ (see \descref{W4}{W4}), we thus get
\begin{equation}
\label{bound_r_i}
r_i \leq C_{(\la,\plog,n)} r. 
\end{equation}
It follows from the triangle inequality, \eqref{meas_z_r_tz} and \eqref{bound_r_i} that
% Note that \eqref{3.70} holds and thus summing over all $i \in \Th$ such that  $\tQ \cap2\mch\cap \frac34Q_i \neq \emptyset$ in \eqref{3.81.1} and making use of  \eqref{3.83.1}, we get
\begin{equation*}
 \label{3.84.1}
 J_1 \apprle \sum_{\substack{i\in\Th}} \frac{|\htq_i|}{r^{n+2}} \fiint_{\htq_i} \left| \frac{\tuh(z) - \avgs{\tuh}{\htq_1}}{r_i}\right| d\tz \apprle \sum_{i\in \Th}  \fiint_{\htq_i} \left| \frac{\tuh(z) - \avgs{\tuh}{\htq_i}}{r_i}\right| dz.
\end{equation*}
To bound \eqref{3.84.1},  in the case $i \in \Th_1$, we apply Lemma \ref{improved_est} to get 
\begin{equation}
\label{6.116}
 \fiint_{\htq_i} \left| \frac{\tuh(z) - \avgs{\tuh}{\htq_i}}{r_i}\right| dz \apprle \musym \la^{\frac{d}{2}} \leq o(1),
\end{equation}
and in the case $i \in \Th_1$, we use \eqref{bound_r_i} and \eqref{bound_low_s} to obtain
\begin{equation}
\label{6.117}
 \fiint_{\htq_i} \left| \frac{\tuh(z) - \avgs{\tuh}{\htq_i}}{r_i}\right| dz \apprle_{(\la,\plog,n)} \frac{1}{r_i^{n+3}} \|u\|_{L^1} \leq o(1). 
\end{equation}

Thus, combining  \eqref{6.116} and \eqref{6.117} with \eqref{meas_z_r_tz} and \eqref{3.81.1} gives 
\begin{equation}
\label{3.84.2}
 J_1 \leq C_{(\la,\plog,n,\al_0,\|u\|_{L^1})} = o(1).
\end{equation}

 \item[Estimate for $J_2$:] From triangle inequality, we see that 
 \begin{equation}
 \label{3.82}
   J_2  \apprle \fiint_{Q_r(\tz)}  \left| \frac{\tuh(z) - \avgs{\tuh}{Q_r(\tz)}}{r}\right| dz.
 \end{equation}
  If  $Q_r(\tz) \subset \Om \cap \bfur \times \RR$, then we estimate \eqref{3.82} by first applying Lemma \ref{lemma_crucial_1} for some  $\mu \in C_c^{\infty} (B_r(\tx))$ satisfying $|\mu| \leq \frac{C(n)}{r^{n}}$ and  $|\nabla \mu| \leq \frac{C(n)}{r^{n+1}}$ to get
 \begin{equation}
  \label{3.83}
   J_2  \apprle  \fiint_{Q_r(\tz)} |\nabla \tuh(z)|\, dz + \sup_{t_1,t_2 \in I_r(\tlt)}\left|\frac{\avgs{\tuh}{\mu}(t_2)-\avgs{\tuh}{\mu}(t_1)}{r} \right|.
 \end{equation}
 From hypothesis, we have $2Q_r(\tz) \cap \elam \neq \emptyset$, thus we can estimate the first term on the right-hand side of \eqref{3.83} as follows:
 \begin{equation}
 \label{6.121}
 \fiint_{Q_r(\tz)} |\nabla \tuh(z)|\ dz \overset{\eqref{cut_off_function}}{\apprle}  \fiint_{Q_r(\tz)} |\nabla u| + \left| \frac{u}{\scalex{\al_0}{z_0} \rho}\right|\ dz \overset{\eqref{elam}}{\apprle}  \la^{\frac{1}{p^-}}.
 \end{equation}

 To estimate the second term on the right-hand side of \eqref{3.83}, we observe from Lemma \ref{lemma_crucial_2} that
 \begin{equation}
 \label{6.122}
 \abs{\avgs{\tuh}{\mu}(t_2)-\avgs{\tuh}{\mu}(t_1)} \apprle \frac{|Q_r(\tz)|}{r^{n+1}} \fiint_{Q_r(\tz)} \lbr[[](1 + |\nabla u|)^{p(z) -1 }\rbr[]]_h \, dz  \apprle r \la^{\frac{p^+-1}{p^-}}.
 \end{equation}
 
 Thus in the case $Q_r(\tz) \subset \Om \cap \bfur \times \RR$, we can combine \eqref{6.122},\eqref{6.121} and \eqref{3.83} to get \begin{equation}\label{6.123}J_2 \leq o(1).\end{equation}

 On the other hand,  if $Q_r(\tz) \nsubseteq \Om \cap \bfur \times \RR$, then we can directly apply Poincar\'e's inequality  to get
 \begin{equation}
 \label{3.89}
  \begin{array}{ll}
   J_2 & \apprle \fiint_{Q_r(\tz)} \left| \frac{\tuh(z)}{r}\right| \ dz \apprle  \fiint_{Q_r(\tz)} |\nabla \tuh(z)| \ dz \overset{\eqref{6.121}}{\apprle}  \la^{\frac{1}{p^-}} = o(1).
  \end{array}
 \end{equation}
 \end{description}%%%%%%%%%%%%%%%%%%%%%%%%%%%%%%%%%%%%%%%%%%%%%%%%%%%%%%%%%%%%%%%%%%%%%%%%%%%%%%%%%%%%%%%%%%%%%%%%%%%%

Therefore, combining \eqref{6.123} or \eqref{3.89} along with \eqref{3.84.2} and  \eqref{6.112} shows that $I_r(\tz) \leq o(1)$ in the case \eqref{case_2} holds. 
 
 \item[Case \eqref{case_3}:] In this case, using \eqref{meas_z_r_tz} and the bound $r \geq \frac{1}{3} \la^{-\frac{1}{p^+} + \frac{d}{2}} (\rho_b-\rho_a) = C_{(\la,\plog,d,n,\rho_b,\rho_a)}$, 
we observe that \begin{equation}\label{6.300}|Q_r(\tz) \cap \bfur \times S_1| \overset{\eqref{meas_z_r_tz}}{\geq} C_{(n,\plog,\rho_a,\rho_b)}.\end{equation} 
We then obtain from \eqref{6.300} that
 \begin{equation}
 \label{6.301}
 I_r(\tz) \leq C_{(\la,\plog,d,n,\rho_b,\rho_a)} \iint_{Q_r(\tz) \cap \bfur \times S_1} |\vlh| \, dz \overset{\eqref{lemma6.8-one}}{\leq}  o(1). 
 \end{equation}

 \item[Case \eqref{case_4}:] Similar to the case \eqref{case_3}, we again obtain \eqref{6.300}, thus we can proceed exactly as in \eqref{6.301} to bound $I_r(\tz) \leq o(1)$.
 \end{description}
This completes the proof of the Lipschitz continuity. 
\end{proof}

 \section{Caccioppoli type inequality}
 \label{section_five}
 
 \begin{lemma}
 \label{Caccioppoli}
There exist a small constant $\be_0 = \be_0(n,\plog,\lamot,m_e) \in (0,1)$ and a positive constant $C_{cac} = C_{cac}(n,\plog,\lamot,m_e)$ such that the following holds:
Let $Q := B \times I = Q_{\rho}^{\al_0}(z_0)$ be a parabolic cylinder for some $\al_0\geq 1$ satisfying \eqref{est_bnd_al_0}. Then we have
 \begin{multline}
  \label{conclusion_1}
  \al_0^{1-\be}|\qone| +  \sup_{t \in \ione} \int_{\bone} \left| {u(x,t)} \right|^2 \tilde{\mathcal{M}}(x,t)^{-\be} \, dx \\
  \leq C_{cac} \left[ \iint_{\qfur}  \al_0^{1-\frac{2}{p(z_0)}-\be} \left| \frac{u}{\scalex{\al_0}{z_0}\rho_b} \right|^2 \ dz + \iint_{\qfur}\left| \frac{u}{\scalex{\al_0}{z_0} \rho_b} \right|^{p(z)(1-\be)} dz \right],
 \end{multline}
 where we have set $\tilde{\mathcal{M}}(x,t):=\max\{g(x,t)^{\frac{1}{1-\be}}, c_e\al_0\}$.

\end{lemma}

\begin{proof}
Let $t_1 \in S_1 \setminus \itwos $ and $t \in \ione$ with $t_1 < t$.   We  shall make use of  $\eta(x)\vlh(x,\tau)$ where $\vlh$ is from \eqref{lipschitz_function} and $\eta$ is from \eqref{cut_off_function} as a test function in \eqref{very_weak_solution} (this is possible since $\spt(\vlh) \subset \Om \cap \bfur \times \RR$ and $\vlh \in C^{0,1}(\bfur \times S_1)$ from Lemma \ref{lipschitz_continuity}). Thus, after integrating the resulting expression of \eqref{very_weak_solution} over $(t_1,t)$, we get
%  
% , then using the definition of very weak solution, we get 
 \begin{equation}
 \label{cac1}
  L_1 + L_2:=\int_{t_1}^t \left[ \int_{\Om_{8\rho}} \frac{d{[u]_h}}{d\tau} \eta(y) \vlh(y,\tau) + \iprod{[\aa(y,\tau,\nabla u)]_h}{\nabla (\eta \vlh)} \ dy\right]\ d\tau = 0.
 \end{equation}

 \begin{description}[leftmargin=*]
  \item[Estimate of $L_1$:] Note that $\zeta(\tau) = 1$ for all $\tau \in (t_1,t)$.
Using the standard hole filling technique, we have
  \begin{equation}
   \label{6.38}
   \begin{array}{ll}
    L_1 %& = \int_{t_1}^t \int_{\Om_{8\rho}} \dds{[u]_h(y,s)} \eta(y) \zeta(s)  \vlh(y,s) \ dy\ ds \\
    & = \int_{t_1}^t \int_{\Om_{8\rho}} \dds{\tu_h(y,\tau)}  \vlh(y,s) \ dy\ d\tau \\
%     & = \int_{t_1}^t \int_{\Om_{8\rho}}  \dds{\tu_h}  (\vlh-\tu_h+\tu_h) \ dy \ ds \\
%     & = \int_{t_1}^t \int_{\Om_{8\rho}}  \dds{\tu_h}  (\vlh-\tu_h) \ dy \ ds +   \int_{t_1}^t \int_{\Om_{8\rho}}  \dds{\tu_h} \tu_h \ dy \ ds \\
%     & = \int_{t_1}^t \int_{\Om_{8\rho}\setminus \elam^s}  \dds{\tu_h}  (\vlh-\tu_h) \ dy \ ds +   \int_{t_1}^t \int_{\Om_{8\rho}}  \dds{(\tu_h)^2}  \ dy \ ds \\
%     & = \int_{t_1}^t \int_{\Om_{8\rho}\setminus \elam^s}  \dds{(\tu_h+\vlh -\vlh)}  (\vlh-\tu_h) \ dy \ ds +   \int_{t_1}^t \int_{\Om_{8\rho}}  \dds{(\tu_h)^2}  \ dy \ ds \\
%     & = \int_{t_1}^t \int_{\Om_{8\rho}\setminus \elam^s}  \dds{\vlh}  (\vlh-\tu_h) \ dy \ ds +   \int_{t_1}^t \int_{\Om_{8\rho} \setminus \elam^s} \dds{(\tu_h - \vlh)} (\vlh-\tu_h) \ dy \ ds + \\
%     & \qquad \qquad + \int_{t_1}^t \int_{\Om_{8\rho}}  \dds{(\tu_h)^2}  \ dy \ ds \\
    & = \int_{t_1}^t \int_{\Om_{8\rho}\setminus \elam^{\tau}}  \dds{\vlh}  (\vlh-\tu_h) \, dy \, d\tau +    \int_{t_1}^t \int_{\Om_{8\rho}}  \frac{d{\lbr (\tu_h)^2 - (\vlh - \tu_h)^2 \rbr }}{d\tau}  \, dy \, d\tau \\
    & := J_2 + J_1(t) - J_1(t_1),
   \end{array}
  \end{equation}
  where we have set 
  \begin{equation*}
  \label{def_i_1}J_1(\tau) := \frac12 \int_{\Om_{8\rho}} ( (\tu_h)^2 - (\vlh - \tu_h)^2 ) (y,\tau) \ dy.
  \end{equation*}

  \begin{description}
   \item[Estimate for $J_2$:] Taking absolute values and making use of Lemma \ref{lemma6.8-2}, we get
   \begin{equation}
    \label{6.39}
     |J_2|  \apprle \iint_{\qfur\setminus \elam}   \left| \dds{\vlh}  (\vlh-\tu_h)\right| dy \ d\tau \apprle\musym^2\la |\RR^{n+1} \setminus \elam| + \frac{1}{s} \iint_{\qfur} |[u]_h|^2 \, dy \ d\tau. 
   \end{equation}
%    To obtain \redr{a} above, we made use of Lemma \ref{lemma3.14} with $\vartheta = 1$ and the trivial bound $|\tuh| \leq [u]_h$.

%    \item[Estimate for $I_1(t)$] 
   \item[Estimate for $J_1(t_1)$:] We first claim that we can choose $t_1\in S_1 \setminus \ione$ such that
   \begin{equation}
    \label{6.40}
    |J_1(t_1)|  
%     = \left| \int_{\Om_{8\rho}} \lbr |\tu_h|^2 - |\vlh - \tu_h|^2 (x,t_1)\rbr \ dx \right|
    \leq \frac{3}{s} \int_{S_1 \setminus \ione} \left| \int_{\Om \cap \bfur} \lbr |\tu_h|^2 - |\vlh - \tu_h|^2 (x,\tau)\rbr dy\right| d\tau
   \end{equation}
holds   with $s$ defined as in \eqref{def_s}. 
Suppose not, then for any $t_1 \in S_1 \setminus \ione$, we would have 
\begin{multline}
\label{J_1_t_11}
 \left| \int_{\Om \cap \bfur } \lbr |\tu_h|^2 - |\vlh - \tu_h|^2\rbr (y,t_1) \, dy \right| \\
 > \frac{3}{s} \int_{S_1 \setminus \ione} \left| \int_{\Om\cap \bfur} \lbr |\tu_h|^2 - |\vlh - \tu_h|^2 \rbr(y,\tau) \, dy\right| d\tau.
\end{multline}
But the gap between $S_1$ and $\ione$ is given by
$$\scalet{\al_0}{z_0} \left( 2\rho_a + \frac19 (\rho_b-\rho_a) \right) \frac{(\rho_b-\rho_a)}{9} \geq \frac{s}{3},$$
which gives
\begin{equation}
\label{J_1_t_22}
  \hint_{S_1 \setminus \ione} \left| \int_{\Om \cap \bfur} \lbr |\tu_h|^2 - |\vlh - \tu_h|^2 (y,\tau)\rbr dy\right| d\tau \geq \min_{\tau \in S_1\setminus \ione} |J_1(\tau)|.
\end{equation}
Combining \eqref{J_1_t_11} and \eqref{J_1_t_22}, for any $t_1 \in S_1 \setminus \ione$, we get $|J_1(t_1)| >  \min_{\tau \in S_1 \setminus \ione} |J_1(\tau)|$,
which is absurd. Hence \eqref{6.40}  must be true for some $t_1 \in S_1 \setminus \ione$.

From the construction \eqref{lipschitz_function},  we have $\vlh = \tu_h$ on $\elam$. Furthermore,  $\spt(\tu_h) \subset \qthr$ and $|\tu_h| \leq |[u]_h|\lsb{\chi}{\qfur}$ holds. For $t_1 \in S_1 \setminus \ione$ satisfying \eqref{6.40}, we then get
\begin{equation}
 \label{6.44}
 \begin{array}{ll}
  |J_1(t_1)| %& \leq \frac{1}{s} \iint_{8Q} \left| |\tu_h|^2 - |\vlh - \tu_h|^2 \right| dz \\
& \leq   \frac{3}{s} \iint_{\qfur\cap \elam } \left| |\tu_h|^2 - |\vlh - \tu_h|^2 \right| dz + \frac{3}{s} \iint_{\qfur\setminus \elam} \left| |\tu_h|^2 - |\vlh - \tu_h|^2 \right| dz \\
% & =   \frac{1}{s} \iint_{8Q\cap \elam }  |\tu_h|^2  dz + \frac{1}{s} \iint_{8Q\setminus \elam} \left| |\tu_h|^2 - |\vlh - \tu_h|^2 \right| dz \\
& \apprle   \frac{1}{s} \iint_{\qfur\cap \elam }  |\tu_h|^2  dz + \frac{1}{s} \iint_{\qfur\setminus \elam}  |\tu_h|^2 + |\vlh |^2  dz \\
% & =    \frac{1}{s} \iint_{8Q }  |\tu_h|^2  dz + \frac{1}{s} \iint_{8Q\setminus \elam}   |\vlh |^2  dz \\
& \overset{\redlabel{4.8.a}{a}}{\apprle} \frac{1}{s} \iint_{\qfur} |[u]_h|^2 \ dz.
 \end{array}
\end{equation}
To obtain \redref{4.8.a}{a}, we used Lemma \ref{lemma6.8} (applied with $\vartheta = 2$).

  \end{description}

  \item[Estimate for $L_2$:] 
   We decompose the expression as 
\begin{align}
\nonumber     L_2 & = \int_{t_1}^t \int_{\Om \cap \bfur \cap \elam(\tau)} \iprod{[\aa(y,\tau,\nabla u)]_h}{\nabla (\eta \vlh)} \, dy\ d\tau  \\
\nonumber     & \qquad \quad + \int_{t_1}^t \int_{\Om \cap \bfur \setminus \elam(\tau)} \iprod{[\aa(y,\tau,\nabla u)]_h}{\nabla (\eta \vlh)} \, dy\ d\tau \\
\label{4.9}     & := L_2^1 + L_2^2.
\end{align}

   \begin{description}
    \item[Estimate for $L_2^2$:] Using the chain rule, \eqref{abounded}, \eqref{cut_off_function} along with Corollary \ref{corollary6.9}, we get
\begin{align}
\nonumber L^2_2  & \leq \int_{t_1}^t \int_{\Om \cap \bfur  \setminus \elam(\tau)} {[(|\nabla u|+1)^{p(z)-1}]_h}{|\nabla (\eta \vlh)|} \, dy\ d\tau  \\
\nonumber    & \apprle \iint_{(\bfur \times S_1) \setminus \elam} {[(|\nabla u|+1)^{p(z)-1}]_h} \lbr \frac{|\vlh|}{\scalex{\al_0}{z_0}\rho} + |\nabla \vlh|\rbr \ dy\ d\tau  \\%+ \frac{1}{\rho}\int_{(8B\times 2\tm) \setminus \elam} {[|\nabla u|^{p-1}]_h}{|\vlh|} dy\ d\tau  \\
\label{4.10} &\apprle \musym \la |\RR^{n+1} \setminus \elam| + \frac{1}{s} \iint_{\qfur} |\tuh|^2 \, dy\ d\tau.
\end{align}
% To obtain \redr{a} above, we made use of Lemma \ref{corollary3.20}. 
\end{description} 
\end{description}

Substituting \eqref{4.10} into \eqref{4.9} and \eqref{6.39}, \eqref{6.44} into \eqref{6.38},  and finally making use of \eqref{cac1} along with the bound $|\tuh| \leq |[u]_h|\lsb{\chi}{\qfur}$, we obtain
 \begin{multline}
 \label{6.45}
    \frac12 \int_{\Om\cap \bfur} | (\tu_h)^2 - (\vlh - \tu_h)^2 | (y,t) \ dy +  \int_{t_1}^t \int_{\elam(\tau)} \iprod{[\aa(y,\tau,\nabla u)]_h}{\nabla (\eta \vlh)} \, dy \, d\tau   \\
%    \hspace*{6cm}\apprle \la^p |\RR^{n+1} \setminus \elam| + \frac{1}{s} \iint_{8Q} |\tu_h|^2 \, dy \, d\tau \\
   \hspace*{6cm} \apprle \mu^2 \la |\RR^{n+1} \setminus \elam| + \frac{1}{s} \iint_{\qfur} |[u]_h|^2 \, dy \, d\tau.
 \end{multline}
 Since the estimate in \eqref{6.45} is independent of $h$, we can let $h \searrow 0$ to get
 \begin{multline}
 \label{6.46.1}
    \frac12 \int_{\Om \cap \bfur} | (\tu)^2 - (\vl - \tu)^2 | (y,t) \ dy +  \int_{t_1}^t \int_{\elam(\tau)} \iprod{\aa(y,\tau,\nabla u)}{\nabla (\eta \vl)} \, dy \, d\tau   \\
%    \hspace*{6cm}\apprle \la^p |\RR^{n+1} \setminus \elam| + \frac{1}{s} \iint_{8Q} |\tu_h|^2 \, dy \, d\tau \\
   \hspace*{6cm} \apprle \mu^2 \la |\RR^{n+1} \setminus \elam| + \frac{1}{s} \iint_{\qfur} |u|^2 \, dy \, d\tau.
 \end{multline}
%  
%  
%  \begin{equation}
%   \label{6.46.1}
%   M_1 + M_2 \apprle M_3,
%  \end{equation}
% where we have set
% \begin{equation}
%  \label{def_M}
%  \begin{array}{ll}
%   M_1 & := \frac12 \int_{\Om_{8\rho}} | (\tu)^2 - (\vl - \tu)^2 | (x,t) \ dx,\\
%   M_2 & := \int_{t_1}^t \int_{E_{\tau}(\la)} \iprod{\aa(x,\tau,\nabla u)}{\nabla (\eta \vl)} \ dx\ d\tau,\\
%   M_3 & := \la^p |\RR^{n+1} \setminus \elam| + \frac{1}{s} \iint_{8Q} |u|^2 \ dz .
%  \end{array}
% \end{equation}
% 
  Using the fact that $\tu = \vl$ on $\elam$ and \eqref{cruc_est_3},  we have
 \begin{equation}
 \label{4.13}
  \begin{array}{ll}
    \int_{\bfur} | (\tu)^2 - (\vl - \tu)^2 | (y,t) \ dy   %& = \int_{\elam^t} | \tu (x,t)|^2 \ dx + \int_{\Om_{8\rho}\setminus \elam^t} | (\tu)^2 - (\vl - \tu)^2 | (x,t) \ dx \\
%    & = \int_{\elam^t} | \tu (x,t)|^2 \ dx + \int_{8B\setminus \elam^t} | (\tu)^2 - (\vl - \tu)^2 | (x,t) \ dx \\
   & \apprge \int_{\elam(t)} | \tu (x,t)|^2 \ dx   - \musym^2 \la |\RR^{n+1} \setminus \elam| \\
   & \qquad - \musym^2 \frac{1}{s} \iint_{\qfur} |u|^2 \ dz. 
  \end{array}
 \end{equation}
%  To obtain \redr{a} above, we made use of Lemma \ref{crucial_lemma}. 
%  
%  
%  Thus substituting \eqref{4.13} into \eqref{6.46.1}, we get
% %  Thus, \eqref{6.46} becomes 
%  \begin{equation}
%  \label{6.47}
%   \begin{array}{l}
%    \frac12 \int_{\elam^t} | \tu(x,t)|^2 \ dx + M_2 \apprle M_3.
% %    J_1 + J_2 \leq J_3. 
%   \end{array}
%  \end{equation}
Making use of the bounds \eqref{4.13} and \eqref{6.46.1}, followed by multiplying the resulting expression with $\la^{-1-\be}$ and integrating over the interval $(c_e\al_0,\infty)$ with respect to $\la$ (recall that  $c_e$ is as in Lemma \ref{fix_ce}), for almost every $t \in S_1$, we get 
 \begin{equation}
 \label{K_expression}
K_1 + K_2 \apprle K_3 + K_4,
 \end{equation}
 where we have set
 \begin{equation*}
  \begin{array}{@{}r@{}c@{}l@{}}
   K_1 \ &:=&\  \frac12\int_{c_e \al_0}^{\infty} \la^{-1-\be} \int_{\elam(t)} | \tu(y,t)|^2 \ dy \ d\la, \\%\label{def_K_1}\\
  K_2 \ &:=& \ \int_{c_e \al_0}^{\infty} \la^{-1-\be}\int_{t_1}^t \int_{\elam(\tau)} \iprod{\aa(y,\tau,\nabla u)}{\nabla (\eta u)} \ dy\ d\tau \ d\la, \\%\label{def_K_2}\\
  K_3\  &:=& \ \musym^2\int_{c_e \al_0}^{\infty} \la^{-1-\be}  \la |\RR^{n+1} \setminus \elam| \  d\la, \\%\label{def_K_3}\\
  K_4 \ &:=& \ \musym^2 \frac{1}{s} \int_{c_e \al_0}^{\infty} \la^{-1-\be}   \iint_{\qfur} |u|^2 \ dy  \ d\tau \ d\la .%\label{def_K_4}.
  \end{array}
 \end{equation*}

 We now define the truncated Maximal function 
 \begin{equation}\label{def_mtilde}\tilde{\mathcal{M}}(z) := \max \{ c_e \al_0, \lbr g(z)\rbr^{\frac{1}{1-\be}}\}\end{equation} and then estimate each of the $K_i$ for $i \in \{1,2,3,4\}$ as follows:
 \begin{description}
  \item[Estimate for $K_1$:]  By applying Fubini, we obtain
  \begin{equation}
  \label{4.19}
    K_1 \geq \int_{\bfur} |\tu(y,t)|^2 \int_{c_e g(y,t)}^{\infty} \la^{-1-\be} \, d\la \, dy 
     \apprge \frac{1}{\be c_e^{\be}}  \int_{\bfur} \tilde{\mathcal{M}}(y,t)^{-\be} | \tu(y,t)|^2 \, dy .
  \end{equation}

  \item[Estimate for $K_2$:] Again applying Fubini, we get
  \begin{equation*}
  \label{4.20}
    K_2 = \frac{1}{\be c_e^{\be}} \int_{t_1}^s \int_{\Om \cap \bfur } \tilde{\mathcal{M}}(y,\tau)^{-\be} \iprod{\aa(y,\tau,\nabla u)}{\nabla (\eta u)} \, dy \, d\tau. 
  \end{equation*}
  Applying the chain rule along with \eqref{abounded}, we get
  \begin{equation}
\label{4.25}
 \begin{array}{ll}
%   \int_{t_1}^t \int_{\Om_{8\rho}} g(y,\tau)^{-\be} \iprod{\aa(y,\tau,\nabla u)}{\nabla \tu} \eta  \, dy \, d\tau  + \int_{t_1}^t \int_{\Om_{8\rho}} g(y,\tau)^{-\be} \iprod{\aa(y,\tau,\nabla u)}{\nabla \eta } \tu \, dy \, d\tau  \\
  K_2 &= \int_{t_1}^s \int_{\Om \cap \bfur} \tilde{\mathcal{M}}(y,\tau)^{-\be} \iprod{\aa(y,\tau,\nabla u)}{\nabla u} \eta^2  \, dy \, d\tau   \\
  & \qquad + \int_{t_1}^s \int_{\Om \cap \bfur} \tilde{\mathcal{M}}(y,\tau)^{-\be} \iprod{\aa(y,\tau,\nabla u)}{\nabla \eta^2 } u \, dy \, d\tau  \\
  &\apprge \int_{t_1}^s \int_{\Om \cap \bfur} \tilde{\mathcal{M}}(y,\tau)^{-\be} |\nabla u|^{p(z)} \eta^2  \, dy \, d\tau  \\
  & \qquad - \int_{t_1}^s \int_{\Om \cap \bfur} \tilde{\mathcal{M}}(y,\tau)^{-\be} \lbr |\nabla u|+1\rbr^{p(z)-1} \frac{|u|}{\scalex{\al_0}{z_0}\rho} \, dy \, d\tau  \\
  & \apprge  \iint_{\qone} \tilde{\mathcal{M}}(y,\tau)^{-\be} |\nabla u|^{p(z)} \eta^2  \, dy \, d\tau  \\
  & \qquad - \iint_{\qfur} \tilde{\mathcal{M}}(y,\tau)^{-\be} \lbr |\nabla u|+1\rbr^{p(z)-1} \frac{|u|}{\scalex{\al_0}{z_0}\rho} \, dy \, d\tau  \\
  & := \aa_1 + \aa_2.
 \end{array}
\end{equation}

\begin{description}
  \item[Estimate for $\aa_1$:] Let $S :=\{ z \in \qone: |\nabla u(z)|^{p(z)} \geq \be \tilde{\mathcal{M}}(z)\},$ then we get
 \begin{equation}
 \label{4.26}
  \begin{array}{lcl}
   \iint_{\qone} |\nabla u|^{p(z)(1-\be)} \ dz  & =& \iint_S |\nabla u|^{p(z)(1-\be)} \ dz + \iint_{\qone\setminus S} |\nabla u|^{p(z)(1-\be)} \ dz \\
%    & \leq \be^{-\be} \iint_S g(z)^{-\be} |\nabla u|^p \ dz + \iint_{Q\setminus S} |\nabla u|^{p-\be} \ dz \\
   & \leq &\be^{-\be} \iint_Q \tilde{\mathcal{M}}(z)^{-\be} |\nabla u|^{p(z)} \ dz + \be^{1-\be} \iint_{\qfur\setminus S} \tilde{\mathcal{M}}(z)^{1-\be} \ dz \\
%       & \leq \be^{-\be} \iint_Q g(z)^{-\be} |\nabla u|^p \ dz + \be^{p-\be} \iint_{Q} g(z)^{p-\be} \ dz \\
%       & \overset{\text{\redr{a}}}{\apprle} \be^{-\be} \iint_Q g(z)^{-\be} |\nabla u|^p \ dz + \be^{p-\be} \iint_{8Q} |\nabla \tu|^{p-\be}  + |\nabla u|^{p-\be} \ dz \\
%       & \overset{\text{\redr{b}}}{\apprle} \be^{-\be} \iint_Q g(z)^{-\be} |\nabla u|^p \ dz + \be^{p-\be} \iint_{8Q} \lbr |\nabla u|^{p-\be} +  \left| \frac{u}{\rho}\right|^{p-\be}\rbr\ dz \\
%       & \overset{\text{\redr{c}}}{\apprle} \be^{-\be} \iint_Q g(z)^{-\be} |\nabla u|^p \ dz + \be^{p-\be} \iint_{8Q} |\nabla u|^{p-\be} \ dz \\
            & \apprle  & \iint_{\qone} \tilde{\mathcal{M}}(z)^{-\be} |\nabla u|^{p(z)} \ dz   + \be^{1-\be} |\qfur| \al_0^{1-\be} \\
            && + \be^{1-\be} \iint_{\qfur} g(z) \ dz \\
            & \overset{\text{Lemma \ref{max_bnd}}}{\apprle}  & \iint_{\qone} \tilde{\mathcal{M}}(z)^{-\be} |\nabla u|^{p(z)} \ dz   + \be^{1-\be} |\qfur| \al_0^{1-\be} \\
            && + \be^{1-\be} \iint_{\qfur} |\nabla u|^{p(z)(1-\be)} + \abs{\frac{u}{\scalex{\al_0}{z_0}\rho}}^{p(z)(1-\be)} \ dz.
      \end{array}
 \end{equation}
%  To obtain the last inequality, we made use of the fact that $\be^{-\be} \leq 2$ for $\be >0$ and $\be^{p-\be} \leq \be$ since $p-\be \geq 1$. 
%  To obtain \redr{a}, we made use of \eqref{def_g} along with Lemma \ref{max_bound}, to get \redr{b}, we applied chain rule along with the bound $|\nabla \rho| \apprle \frac{1}{\rho}$ and finally to obtain \redr{c}, we made use of Theorem \ref{sobolev-poincare}. 
 
 \item[Estimate for $\aa_2$:]  We use the bound  $\lsb{\chi}{\qfur}( |\nabla u(z)| +1)^{p(z)}  \leq \tilde{\mathcal{M}}(z)$ for a.e $z \in \RR^n$, along with Young's inequality, to get
 \begin{equation}
 \label{4.27}
  \begin{array}{rcl}
   \aa_2 & \apprle& \iint_{\qfur}   (|\nabla u|+ 1) ^{p(z)(1-\be)-1} \frac{|u|}{\scalex{\al_0}{z_0}\rho} \ dz\\
%    & \overset{\text{\redr{a}}}{\apprle }& \varepsilon \iint_{8Q} |\nabla u|^{p-\be} \ dz + C(\varepsilon) \iint_{8Q} \left| \frac{u}{\rho}\right|^{p-\be}\ dz \\
  & \apprle &  \ve \iint_{\qfur}(|\nabla u|+1)^{p(z)(1-\be)} + C(\ve) \iint_{\qfur} \left| \frac{u}{\scalex{\al_0}{z_0}\rho}\right|^{p(z)(1-\be)}\ dz \\
  & \apprle & \ve |\qfur| \al_0^{1-\be} + C(\ve) \iint_{\qfur} \left| \frac{u}{\scalex{\al_0}{z_0}\rho}\right|^{p(z)(1-\be)}\ dz.
  \end{array}
 \end{equation}
\end{description}

  \item[Estimate for $K_3$:] Applying the layer-cake representation (see for example \cite[Chapter 1]{Grafakos}), we get
  \begin{equation}
  \label{4.21}
   \begin{array}{rcl}
    K_3  & =& \musym^2 \frac{1}{1-\be} \iint_{\RR^{n+1}} \tilde{\mathcal{M}}(z)^{1-\be} \ dz \\
%     & \overset{\text{\redr{a}}}{\apprle}& \frac{1}{p-\be} \iint_{\RR^{n+1}} \mm(|\nabla \tu|^q)^{\frac{p-\be}{q}}(y,s) \ dy \ ds \\
%     & \overset{\text{\redr{b}}}{\apprle}& \frac{1}{p-\be} \iint_{8Q} |\nabla \tu|^{p-\be}(y,s) \ dy \ ds + \frac{1}{p-\be} \iint_{8Q} |\nabla u|^{p-\be}(y,s) \ dy \ ds \\
%     & {\apprle}& \frac{1}{p-\be} \iint_{8Q} \lbr |\nabla u|^{p-\be}(y,s) + \left| \frac{u}{\rho}\right|^{p-\be} \rbr \ dy \ ds \\
     & \apprle & \musym^2 \left[ \frac{1}{1-\be} \iint_{\qfur}  \lbr |\nabla u|+1\rbr^{p(z)(1-\be)} +\al_0^{1-\be}  + \left| \frac{u}{\scalex{\al_0}{z_0}\rho}\right|^{p(z)(1-\be)}\ dz \right] \\
    & \apprle & \musym^2 \left[ \al_0^{1-\be} |\qfur| + \frac{1}{1-\be} \iint_{\qfur} \left| \frac{u}{\scalex{\al_0}{z_0}\rho}\right|^{p(z)(1-\be)}\ dz \right].  
   \end{array}
  \end{equation}
% To obtain \redr{a}, we used \eqref{def_g}, to obtain \redr{b}, we made use of the maximal function bound from Lemma \ref{max_bound}, to get \redr{c}, we applied Poincar\'e's inequality from Theorem \ref{sobolev-poincare} and finally to get \redr{d}, we made use of \eqref{hypothesis_1}.
  \item[Estimate for $K_4$] Again applying Fubini, we get 
  \begin{equation}
  \label{4.22}
K_4  = \musym^2 \frac{1}{s} \int_{{c_e \al_0}}^{\infty} \la^{-1-\be}   \iint_{\qfur} |u|^2 \ dz \ d\la  = \musym^2 \frac{1}{\be} \iint_{\qfur}  \al_0^{-\be}  \frac{|u|^2}{s} \ dz.
  \end{equation}
 \end{description}
% 
%  Combining \eqref{4.19}, \eqref{4.20}, \eqref{4.21} and \eqref{4.22} into \eqref{K_expression}, we get
%  \begin{equation*}
%  \label{4.24}
%   \begin{array}{l}
%     \frac{1}{2\be}  \int_{B} g(x,t)^{-\be} | \tu(x,t)|^2 \ dx + \frac{1}{\be} \int_{t_1}^t \int_{\Om_{8\rho}} g(y,\tau)^{-\be} \iprod{\aa(y,\tau,\nabla u)}{\nabla (\eta \tu)} \, dy \, d\tau  \apprle \\
%     \hspace*{6cm} \apprle \al_0^{p-\be} |Q| + \frac{1}{\be} \iint_{8Q} \al_0^{-\be}  \frac{|u|^2}{s} \, dy \, d\tau.
%   \end{array}
%  \end{equation*}
% 
Substituting \eqref{4.26} and \eqref{4.27} into \eqref{4.25} followed by combining  \eqref{4.19}, \eqref{4.21}, \eqref{4.22} and \eqref{K_expression}, we get
\begin{align*}
& \frac{1}{2\be} \int_{\bone} \tilde{\mathcal{M}}(y,t)^{-\be} | \tu(y,t)|^2 \, dy + \frac{1}{\be} \iint_{\qone} |\nabla u|^{p(z)(1-\be)} \, dz \\
& \apprle \musym^2 \left[ \frac{1}{\be} \be^{1-\be} \al_0^{1-\be} |\qfur|  + \frac{1}{\be} \varepsilon |Q| \al_0^{1-\be} + \frac{\be^{1-\be} + C(\ve)}{\be}  \iint_{\qfur} \left| \frac{u}{\scalex{\al_0}{z_0}\rho}\right|^{p(z)(1-\be)} dz \right. \\
  &  \qquad \qquad \qquad \qquad \quad +  \left. \al_0^{1-\be} |\qfur| + \frac{1}{\be} \iint_{\qfur} \al_0^{-\be}  \frac{|u|^2}{s} \ dz \right].
%       \frac{1}{2\be}  \iint_{\Om_{8\rho}} g(x,t)^{-\be} | \tu(x,t)|^2 \ dx + \frac{1}{\be} C (1-\be^{p-\be}) \int_Q |\nabla u|^{p-\be} \ \mcZ \\
%     \hspace*{2cm} \apprle \al_0^{p-\be} |Q| + \frac{1}{\be} \int_{8Q} \la_4^{-\be}  \frac{|u|^2}{s} \ d\mcZ +  \varepsilon |Q| \ka \al_0^{p-\be} + C(\varepsilon) \int_{16Q} \lbr \frac{|u|}{\rho} \rbr^{p-\be} \ d\mcZ.
\end{align*}
Let us now take $\rho_a = \rho$ and $\rho_b = 16\rho$, then $\musym = constant$. This along with \eqref{def_s} gives
% Multiplying the above expression  by $\be$ and then using the intrinsic scaling $s = \rho^2 \al_0^{2-p}$ from \eqref{def_s} along with \eqref{hypothesis_1}, we get 
\begin{multline*}
         \int_{\bnot} \tilde{\mathcal{M}}(y,t)^{-\be} | \tu(y,t)|^2 \, dy + |\qnot| \al_0^{1-\be} \\
\apprle   \iint_{\qfve} \lbr \frac{|u|}{\scalex{\al_0}{z_0}\rho} \rbr^{p(z)(1-\be)} dz + \iint_{\qfve} \al_0^{1-\be-\frac{2}{p(z_0)}}  \lbr \frac{|u|}{\scalex{\al_0}{z_0}\rho}\rbr^2 dz.
\end{multline*}
% This gives
% % Choosing $\be \in (0,\be_0)$  and $\varepsilon \in (0,1)$ small, we then make use of  \eqref{hypothesis_1} and the observation $|Q| \approx |B| s^2 = \rho^2 \al_0^{2-p}|B|$, to get
% \begin{equation*}
% \label{4.31}
% %  \begin{array}{l@{}l}
%         \sup_{t \in \tm} \al_0^{p-2} \hint_{B} \tilde{\mathcal{M}}(y,t)^{-\be} \left| \frac{ \tu(y,t)}{\rho}\right|^2 \ dy +  \al_0^{p-\be}   \ \apprle \   \fiint_{8Q}\left[ \al_0^{p-2-\be}   \lbr \frac{|u|}{\rho}\rbr^2 +  \lbr \frac{|u|}{\rho} \rbr^{p-\be} + |\th|^{p-\be} \right]\ dz.
% %  \end{array}
% \end{equation*} 
This completes the proof of the lemma. 
\end{proof}

 \section{Reverse H\"older type inequality}
 \label{section_six}

\begin{lemma}
\label{lemma_lower_order}
Let $Q := B \times I = Q_{\rho}^{\al_0}(z_0)$ be a parabolic cylinder for some $\al_0\geq 1$ satisfying \eqref{est_bnd_al_0}, and let $\bM_0$ be given as in \eqref{def_M_0}.
Then there exists $\rho_0 = \rho_0(n,\La_1,\bM_0)>0$ such that if $0 < 32\rho \leq \rho_0$, then for $\sig = \max \{ 2, p^+_{\qfve}(1-\be) \}$ there holds
\begin{equation*}
\fiint_{\qfor} \left| \frac{u}{\scalex{\al_0}{z_0} \rho} \right|^{\sig} dz \apprle_{(n,\plog,\lamot,m_e)}  \al_0^{\frac{\sigma}{p(z_0)}}.
\end{equation*}
% for some positive constant $c=c(n, \mathfrak{p}^-, \mathfrak{p}^+ ,\La_0, \La_1,m_e)$.
\end{lemma}

\begin{proof}
We set
\begin{equation*}
\tp := \frac{p^-_{\qfve}[2(1-\be)-\be \sig]}{2-\be p^-_{\qfve}}.
\end{equation*}
We shall use Lemma \ref{lemma_crucial_3} with $(\sig,\ga_1,\ga_2,\theta)$ replaced by $\left( \sig,\tp,2(1-\be),\frac{\tp}{\sig} \right)$.
To apply Lemma \ref{lemma_crucial_3}, we need to check that the condition
\begin{equation}
\label{claim_sig}
\frac{\sig}{\tp} \leq 1 + \frac{2(1-\be)}{n}
\end{equation}
holds true.
Let us consider the following two cases:
\begin{description}[leftmargin=*]
\item[Case $\sig=2$:]
In this case, the condition \eqref{claim_sig} is equivalent to
\begin{equation*}
\tp = \frac{p^-_{\qfve}(2-4\be)}{2-\be p^-_{\qfve}} \geq \frac{2n}{n+2-2\be}.
\end{equation*}
Since $p^-_{\qfve} \geq \mathfrak{p}^-$, it is enough to show that
\begin{equation*}
\label{claim_sig=2}
\frac{\mathfrak{p}^-(2-4\be)}{2-\be \mathfrak{p}^-} \geq \frac{2n}{n+2-2\be}.
\end{equation*}
Setting
\begin{equation*}
\Psi(\be) := \frac{\mathfrak{p}^-(2-4\be)}{2-\be \mathfrak{p}^-} - \frac{2n}{n+2-2\be},
\end{equation*}
we observe from $\mathfrak{p}^- > \frac{2n}{n+2}$ that $\Psi$ is continuous on the interval $\left( 0,\frac{1}{\mathfrak{p}^-} \right)$ and that $\Psi(0)>0$.
Therefore, there exists a small constant $\be_0 = \be_0(n,\mathfrak{p}^-)>0$ such that $\Psi(\be)>0$ for all $\be \in (0,\be_0)$.

\item[Case $\sig=p^+_{\qfve}(1-\be)$:]
In this case, we first observe from the definition of $\tp$ that there exists a small constant $\be_0 = \be_0(n,\plog,m_e)>0$ such that $\tp \geq 1$ for all $\be \in (0,\be_0)$.
Then we have
\begin{equation*}
\frac{\sig}{\tp} = 1 + \frac{\sig-\tp}{\tp} \leq 1+\sig-\tp
\end{equation*}
For $0 < \be \leq \frac{1}{\mathfrak{p}^+} \leq \frac{1}{p^-_{\qfve}}$, we obtain
\begin{align}
\label{8.1exp_0}  \sig - \tp = \frac{2 \left[ \sig - p^-_{\qfve}(1-\be) \right]}{2-\be p^-_{\qfve}} & = \frac{2(1-\be) \left( p^+_{\qfve}-p^-_{\qfve} \right)}{2-\be p^-_{\qfve}} 
% \nonumber & \leq 2(1-\be) \left( p^+_{\qfve}-p^-_{\qfve} \right) \\
% \nonumber & \leq 2(1-\be) \modp (32\rho) \\
  \leq 2(1-\be) \modp (\rho_0).
\end{align}
Now making use of  Restriction \ref{rthree}, 
% Choosing $\rho_0 = \rho_0(n,L)>0$ so small that
% $$ \modp(\rho_0) \leq \frac{L}{\log \left( \frac{1}{\rho_0} \right)} \leq \frac{1}{n}, $$
we see that the condition \eqref{claim_sig} holds true.
\end{description}

We now apply Lemma \ref{lemma_crucial_3} with $(\sig,\ga_1,\ga_2,\theta)$ replaced by $\left( \sig,\tp,2(1-\be),\frac{\tp}{\sig} \right)$ to discover that
\begin{multline*}
\fiint_{\qone} \left| \frac{u}{\scalex{\al_0}{z_0} \rho_a} \right|^{\sig} dz 
\leq C \fint_{\ione} \left( \fint_{\bone} |\nabla u|^{\tp} dx \right) \left( \fint_{\bone} \left| \frac{u}{\scalex{\al_0}{z_0} \rho_a} \right|^{2(1-\be)} dx \right)^{\frac{\sig-\tp}{2(1-\be)}} dt,
\end{multline*}
where $C=C_{(n,\plog,m_e)}$ is a positive constant.
Then it follows from  H\"older's inequality and \eqref{def_mtilde}, for  a.e. $t \in \ione$, there holds
\begin{align}
\nonumber & \fint_{\bone \times \{t\}} \left| \frac{u}{\scalex{\al_0}{z_0} \rho_a} \right|^{2(1-\be)} dx \\
\nonumber & = \fint_{\bone \times \{t\}} \left| \frac{u}{\scalex{\al_0}{z_0} \rho_a} \right|^{2(1-\be)} \tilde{\mathcal{M}}(x,t)^{-\be(1-\be)} \tilde{\mathcal{M}}(x,t)^{\be(1-\be)} \, dx \\
\label{8.1pf_aux} & \leq \left( \fint_{\bone \times \{t\}} \left| \frac{u}{\scalex{\al_0}{z_0} \rho_a} \right|^2 \tilde{\mathcal{M}}(x,t)^{-\be} \, dx \right)^{1-\be} \left( \fint_{\bone \times \{t\}} \tilde{\mathcal{M}}(x,t)^{1-\be} \, dx \right)^{\be}.
\end{align}
Using the Caccioppoli type inequality from Lemma \ref{Caccioppoli}, the H\"older inequality, Lemma \ref{max_bnd} and Theorem \ref{measure_density_poincare}, we get
\begin{align}
\nonumber J & := \sup_{t \in \ione} \fint_{\bone \times \{t\}} \left| \frac{u}{\scalex{\al_0}{z_0} \rho_a} \right|^2 \tilde{\mathcal{M}}(x,t)^{-\be} \, dx \\
\nonumber & \apprle \musym^{2} \al_0^{-1+d} \left[ \fiint_{\qfur}  \al_0^{1-\frac{2}{p(z_0)}-\be} \left| \frac{u}{\scalex{\al_0}{z_0}\rho_b} \right|^2 \ dz + \fiint_{\qfur}\left| \frac{u}{\scalex{\al_0}{z_0} \rho_b} \right|^{p(z)(1-\be)} dz \right] \\
\nonumber & \apprle \musym^{2} \al_0^{-1+d} \left[ \al_0^{1-\frac{2}{p(z_0)}-\be} \fiint_{\qfur} \left| \frac{u}{\scalex{\al_0}{z_0}\rho_b} \right|^2 \ dz + \fiint_{\qfurr} (|\nabla u|+1)^{p(z)(1-\be)} dz \right] \\
\nonumber & \apprle \musym^{2} \al_0^{-1+d} \left[ \al_0^{1-\frac{2}{p(z_0)}-\be} \fiint_{\qfur} \left| \frac{u}{\scalex{\al_0}{z_0} \rho_b} \right|^2 \ dz + \al_0^{1-\be} \right] \\
\label{8.1pf_est_of_J} & \apprle \musym^{2} \al_0^{d-\be} \left[ \al_0^{-\frac{2}{p(z_0)}} \left( \fiint_{\qfur}  \left| \frac{u}{\scalex{\al_0}{z_0} \rho_b} \right|^{\sig} \ dz \right)^{\frac{2}{\sig}} + 1 \right].
\end{align}
We now apply the H\"older inequality with exponents $q=\frac{2(1-\be)}{\be(\sig-\tp)}$ and $q'=\frac{2(1-\be)}{2(1-\be)-\be(\sig-\tp)}$ to obtain
\begin{align*}
\nonumber \fiint_{\qone} \left| \frac{u}{\scalex{\al_0}{z_0} \rho_a} \right|^{\sig} dz & \apprle J^{\frac{\sig-\tp}{2}} \fint_{\ione} \left( \fint_{\bone} |\nabla u|^{\tp} \, dx \right) \left( \fint_{\bone \times \{t\}} \tilde{\mathcal{M}}(x,t)^{1-\be} \, dx \right)^{\frac{\be(\sig-\tp)}{2(1-\be)}} dt \\
\nonumber & \apprle J^{\frac{\sig-\tp}{2}} \left( \fiint_{\qone} |\nabla u|^{\tp q'} \, dz \right)^{\frac{1}{q'}} \left( \fiint_{\qone} \tilde{\mathcal{M}}(z)^{1-\be} \, dz \right)^{\frac{\be(\sig-\tp)}{2(1-\be)}} \\
\nonumber & \apprle J^{\frac{\sig-\tp}{2}} \left( \fiint_{\qone} (|\nabla u|+1)^{p(z)(1-\be)} \, dz\right)^{\frac{1}{q'}} \left( \fiint_{\qone} \tilde{\mathcal{M}}(z)^{1-\be} \, dz \right)^{\frac{\be(\sig-\tp)}{2(1-\be)}} \\
 & \apprle J^{\frac{\sig-\tp}{2}} \al_0^{\frac{1-\be}{q'}} \left( \fiint_{\qone} \tilde{\mathcal{M}}(z)^{1-\be} \, dz \right)^{\frac{\be(\sig-\tp)}{2(1-\be)}},
\end{align*}
where we have used the fact that $\tp q' = p^-_{\qfve}(1-\be) \leq p(z)(1-\be)$.
Furthermore, it follows from the definition of $\tilde{\mathcal{M}}(z)$, the boundedness of the strong Maximal function in Lemma \ref{max_bnd} and Theorem \ref{measure_density_poincare} that
\begin{align}
\nonumber \fiint_{\qone} \tilde{\mathcal{M}}(z)^{1-\be} \, dz & \leq (c_e\al_0)^{1-\be} + \fiint_{\qone} g(z) \, dz \\
\nonumber & \apprle \al_0^{1-\be} +  \fiint_{\qfur} \left( (|\nabla u|+1)^{p(z)(1-\be)}  +  \left|\frac{u}{\scalex{\al_0}{z_0}\rho}\right|^{p(z)(1-\be)} \right) dz \\
\nonumber & \apprle \al_0^{1-\be} +   \fiint_{\qfurr} (|\nabla u|+1)^{p(z)(1-\be)} \, dz\\
\label{8.1pf_2} & \apprle \al_0^{1-\be}.
\end{align}
Combining \eqref{8.1pf_est_of_J}--\eqref{8.1pf_2} and Young's inequality yields
\begin{align}
\nonumber \fiint_{\qone} \left| \frac{u}{\scalex{\al_0}{z_0} \rho_a} \right|^{\sig} dz & \leq C J^{\frac{\sig-\tp}{2}} \al_0^{\frac{1-\be}{q'} + \frac{\be(\sig-\tp)}{2}} \\
\nonumber & \leq C \musym^{\sig-\tp} \al_0^{\frac{d(\sig-\tp)}{2}+ \frac{1-\be}{q'} -\frac{\sig-\tp}{p(z_0)}} \left( \fiint_{\qfur}  \left| \frac{u}{\scalex{\al_0}{z_0} \rho_b} \right|^{\sig} \ dz \right)^{\frac{\sig-\tp}{\sig}} \\
\nonumber & \qquad + C \musym^{\sig-\tp} \al_0^{\frac{d(\sig-\tp)}{2}+ \frac{1-\be}{q'}} \\
\nonumber & \leq \frac{1}{2} \fiint_{\qfur}  \left| \frac{u}{\scalex{\al_0}{z_0} \rho_b} \right|^{\sig} \ dz + C \musym^{\frac{(\sig-\tp)\sig}{\tp}} \al_0^{\left[ \frac{d(\sig-\tp)}{2} + \frac{1-\be}{q'} - \frac{\sig-\tp}{p(z_0)} \right] \frac{\sig}{\tp}} \\
\label{8.1pf_est_ab} & \qquad + C \musym^{\sig-\tp} \al_0^{\frac{d(\sig-\tp)}{2}+ \frac{1-\be}{q'}}.
\end{align}
From the definition of $q'$, we infer
\begin{align}
\nonumber \left[ \frac{d(\sig-\tp)}{2} + \frac{1-\be}{q'} - \frac{\sig-\tp}{p(z_0)} \right] \frac{\sig}{\tp} & = \left[ \frac{d(\sig-\tp)}{2} + \frac{2(1-\be)-\be(\sig-\tp)}{2} - \frac{\sig-\tp}{p(z_0)} \right] \frac{\sig}{\tp} \\
\nonumber & = \left[ 1-\be - \frac{\be(\sig-\tp)}{2} \right] \frac{\sig}{\tp} + \left( \frac{d}{2} - \frac{1}{p(z_0)} \right) (\sig-\tp) \frac{\sig}{\tp} \\
\label{8.1exp_00} & = \frac{\sig}{p^-_{\qfve}} + \left( \frac{d}{2} - \frac{1}{p(z_0)} \right) (\sig-\tp) \frac{\sig}{\tp},
\end{align}
where we have used the following identity:
$$
\left[ 1-\be - \frac{\be(\sig-\tp)}{2} \right] \frac{\sig}{\tp} = \left[ 1-\be - \frac{\be \left( \sig-p^-_{\qfve}(1-\be) \right)}{2-\be p^-_{\qfve}} \right] \frac{\sig \left( 2 - \be p^-_{\qfve}\right)}{p^-_{\qfve} (2-2\be-\be \sig)} = \frac{\sig}{p^-_{\qfve}}.
$$
\begin{description}[leftmargin=*]
\item[Case $\sig=p^+_{\qfve}(1-\be)$:]
Since $d \leq 1$, it follows from \eqref{8.1exp_0} and \eqref{claim_sig} that
\begin{align}
\nonumber \left( \frac{d}{2} - \frac{1}{p(z_0)} \right) (\sig-\tp) \frac{\sig}{\tp} & \leq (1-\be) \left( p^+_{\qfve}-p^-_{\qfve} \right) \left( 1+\frac{2(1-\be)}{n} \right) \\
\nonumber & \leq 2 \left( p^+_{\qfve}-p^-_{\qfve} \right) \\
\label{8.1exp_01} & \leq 2 \modp(32\rho).
\end{align}
Combining \eqref{8.1exp_00} and \eqref{8.1exp_01} yields
\begin{equation}
\label{8.1exp_b1}
\al_0^{\left( \frac{d(\sig-\tp)}{2} + \frac{1-\be}{q'} - \frac{\sig-\tp}{p(z_0)} \right) \frac{\sig}{\tp}} \leq \al_0^{\frac{\sig}{p^-_{\qfve}} + 2\modp(32\rho)} \leq \al_0^{\frac{\sig}{p(z_0)} + (\sig+2)\modp(32\rho)} \leq c \al_0^{\frac{\sig}{p(z_0)}}.
\end{equation}
\item[Case $\sig=2$:]
In this case, we know that $p(z_0) \leq p^+_{\qfve} \leq \frac{2}{1-\be}$.
This gives
\begin{equation}
\label{8.1exp_02}
\frac{d}{2} - \frac{1}{p(z_0)} \leq \frac{d}{2} - \frac{1-\be}{2} = \frac{\be-(1-d)}{2} \leq 0, \quad \forall \be \in (0,1-d).
\end{equation}
Therefore, it follows from \eqref{8.1exp_00} and \eqref{8.1exp_02} that
\begin{equation}
\label{8.1exp_b2}
\al_0^{\left( \frac{d(\sig-\tp)}{2} + \frac{1-\be}{q'} - \frac{\sig-\tp}{p(z_0)} \right) \frac{\sig}{\tp}} \leq \al_0^{\frac{\sig}{p^-_{\qfve}}} \leq \al_0^{\frac{\sig}{p(z_0)} + 2 \modp(32\rho)} \leq c \al_0^{\frac{\sig}{p(z_0)}}.
\end{equation}
\end{description}

On the other hand, we note from the definition of $q'$ that
\begin{align}
\nonumber \frac{d(\sig-\tp)}{2} + \frac{1-\be}{q'} & = \frac{d(\sig-\tp)}{2} + \frac{2(1-\be)-\be(\sig-\tp)}{2} \\
\label{8.1exp_1} & = 1-\be + \frac{(d-\be)(\sig-\tp)}{2}.
\end{align}
\begin{description}[leftmargin=*]
\item[Case $\sig=p^+_{\qfve}(1-\be)$:]
Since $d \leq 1$, it follows from \eqref{8.1exp_0} and \eqref{8.1exp_1} that
$$
\frac{d(\sig-\tp)}{2} + \frac{1-\be}{q'} \leq (1-\be) + \left( p^+_{\qfve}-p^-_{\qfve} \right),
$$
and hence
\begin{equation}
\label{8.1exp_a1}
\al_0^{\frac{d(\sig-\tp)}{2}+ \frac{1-\be}{q'}} \leq \al_0^{(1-\be) + \left( p^+_{\qfve}-p^-_{\qfve} \right)} \leq c \al_0^{1-\be} = c \al_0^{\frac{\sig}{p^+_{\qfve}}} \leq c \al_0^{\frac{\sig}{p(z_0)}},
\end{equation}
where we have used \eqref{bnd_al_0_rho} and $\sig=p^+_{\qfve}(1-\be)$.

\item[Case $\sig=2$:]
In this case, we have
\begin{align*}
\frac{d(\sig-\tp)}{2} + \frac{1-\be}{q'} & = 1-\be + \frac{(d-\be)(\sig-\tp)}{2} \\
& = 1-\be + \frac{(d-\be)\left(2-(1-\be)p^-_{\qfve}\right)}{2-\be p^-_{\qfve}}.
\end{align*}
We claim that
\begin{equation}
\label{8.1_claim}
1-\be + \frac{(d-\be)\left(2-(1-\be)p^-_{\qfve}\right)}{2-\be p^-_{\qfve}} \leq \frac{2}{p^-_{\qfve}}.
\end{equation}
A direct computation shows that the above inequality is equivalent to
$$
d(1-\be)\left(p^-_{\qfve}\right)^2 - 2(1+d-\be) \, p^-_{\qfve} + 4 \geq 0.
$$
Setting $\Phi (p):=d(1-\be)p^2 - 2(1+d-\be) + 4$ for $p \in \RR$, we see that the quadratic function $\Phi : \RR \to \RR$ has the minimum at $p_*=\frac{1+d-\be}{d(1-\be)}$.
We observe that
$$p_*=\frac{1+d-\be}{d(1-\be)} > \frac{2}{1-\be}, \quad \forall \beta \in (0,1-d).$$
Since $p^-_{\qfve} \leq p^+_{\qfve} \leq \frac{2}{1-\be}$, we deduce that
\begin{multline*}
d(1-\be)\left(p^-_{\qfve}\right)^2 - 2(1+d-\be) \, p^-_{\qfve} + 4 \\
\geq d(1-\be)\left(\frac{2}{1-\be}\right)^2 - 2(1+d-\be)\left(\frac{2}{1-\be}\right) + 4 = \frac{4d-4(1+d-\be)+4(1-\be)}{1-\be} = 0,
\end{multline*}
which proves the claim \eqref{8.1_claim}.
Therefore, we obtain
\begin{equation}
\label{8.1exp_a2}
\al_0^{\frac{d(\sig-\tp)}{2}+ \frac{1-\be}{q'}} \leq \al_0^{\frac{2}{p^-_{\qfve}}} \leq \al_0^{\frac{2}{p(z_0)} + 2\modp(32\rho)} \apprle \al_0^{\frac{2}{p(z_0)}} = \al_0^{\frac{\sig}{p(z_0)}}.
\end{equation}
\end{description}

We now combine \eqref{8.1exp_b1}, \eqref{8.1exp_b2}, \eqref{8.1exp_a1}, \eqref{8.1exp_a2} with \eqref{8.1pf_est_ab} to discover that
\begin{multline*}
\fiint_{\qone} \left| \frac{u}{\scalex{\al_0}{z_0} \rho_a} \right|^{\sig} dz 
\leq \frac{1}{2} \fiint_{\qfur}  \left| \frac{u}{\scalex{\al_0}{z_0} \rho_b} \right|^{\sig} \ dz + C \left[ \musym^{\frac{(\sig-\tp)\sig}{\tp}} + \musym^{\sig-\tp} \right] \al_0^{\frac{\sig}{p(z_0)}}
\end{multline*}
holds for all $4\rho \leq \rho_a < \rho_b \leq 8\rho$.
Applying Lemma \ref{iter_lemma}, we can now obtain
\begin{equation*}
\fiint_{\qfor} \left| \frac{u}{\scalex{\al_0}{z_0} \rho} \right|^{\sig} dz \apprle \al_0^{\frac{\sig}{p(z_0)}},
\end{equation*}
which completes the proof.
\end{proof}

We now prove the following Reverse-H\"older type inequality.

\begin{lemma}
\label{lemma_reverse_Holder}
Let $Q := B \times I = Q_{\rho}^{\al_0}(z_0)$ be a parabolic cylinder for some $\al_0\geq 1$ satisfying \eqref{est_bnd_al_0}, and let $\bM_0$ be given as in \eqref{def_M_0}.
Then there exist $\rho_0 = \rho_0(n,\La_1,\bM_0)>0$ and $\tq = \tq(n,\plog)>1$ such that if $0 < 32\rho \leq \rho_0$, there holds
\begin{equation*}
\fiint_{\qnot} |\nabla u|^{p(z)(1-\be)} \, dz \apprle_{(n,\plog,\lamot,m_e)}  \left( \fiint_{\qtoo} |\nabla u|^{\frac{p(z)(1-\be)}{\tq}} \, dz \right)^{\tq} + 1.
\end{equation*}
% for some positive constant $c=c(n, \mathfrak{p}^-, \mathfrak{p}^+ ,\La_0, \La_1,m_e)$.
\end{lemma}

\begin{proof}
Using the Caccioppoli type inequality \eqref{conclusion_1} with \eqref{est_bnd_al_0}, we have
\begin{equation}\label{8.2pf_0} \begin{array}{rcl}
 \fiint_{\qnot} |\nabla u|^{p(z)(1-\be)} \, dz &  \apprle &  \fiint_{\qtoo}  \al_0^{1-\frac{2}{p(z_0)}-\be} \left| \frac{u}{\scalex{\al_0}{z_0}\rho} \right|^2 dz + \fiint_{\qtoo}\left| \frac{u}{\scalex{\al_0}{z_0} \rho} \right|^{p(z)(1-\be)} dz  \\
& \apprle &  \al_0^{1-\frac{2}{p(z_0)}-\be} \fiint_{\qtoo} \left| \frac{u}{\scalex{\al_0}{z_0}\rho} \right|^2 dz + \fiint_{\qtoo}\left| \frac{u}{\scalex{\al_0}{z_0} \rho} \right|^{p^+_{\qfve}(1-\be)} dz + 1  \\
 &  \apprle &  \al_0^{1-\frac{2}{p(z_0)}-\be} I_2 + I_{p^+_{\qfve}(1-\be)} + 1,
\end{array}\end{equation}
where we have defined
$$
I_{\sig} := \fiint_{\qtoo} \left| \frac{u}{\scalex{\al_0}{z_0}\rho} \right|^{\sig} dz
$$
for $\sig = 2$ and $\sig = p^+_{\qfve}(1-\be)$.
For such $\sig$, we set $q_1 := \frac{n\sig}{n+2-2\be}$.
Then it is clear that $q_1 < \sig$ and $\frac{\sig}{q_1} = 1 + \frac{2(1-\be)}{n}$.
Applying Lemma \ref{lemma_crucial_3} with $(\sig,\ga_1,\ga_2,\theta)$ replaced by $\left( \sig,q_1,2(1-\be),\frac{q_1}{\sig} \right)$, we obtain
\begin{equation}
\label{8.2pf_1}
I_{\sig} \apprle \fint_{\itoo} \left( \fint_{\btoo} |\nabla u|^{q_1} dx \right) \left( \fint_{\btoo} \left| \frac{u}{\scalex{\al_0}{z_0} \rho} \right|^{2(1-\be)} dx \right)^{\frac{\sig-q_1}{2(1-\be)}} dt.
\end{equation}
As in \eqref{8.1pf_aux}, it follows from the H\"older inequality that for a.e. $t \in \itoo$ (recall \eqref{def_mtilde}),
\begin{equation}
\label{8.2pf_2}
\fint_{\btoo \times \{t\}} \left| \frac{u}{\scalex{\al_0}{z_0} \rho} \right|^{2(1-\be)} dx \leq J^{1-\be} \left( \fint_{\btoo \times \{t\}} \tilde{\mathcal{M}}(x,t)^{1-\be} \, dx \right)^{\be},
\end{equation}
where
$$
J := \sup_{t \in \itoo} \fint_{\btoo \times \{t\}} \left| \frac{u}{\scalex{\al_0}{z_0} \rho} \right|^2 \tilde{\mathcal{M}}(x,t)^{-\be} \, dx.
$$
We now use Lemma \ref{Caccioppoli} (with $\rho_a=2\rho$ and $\rho_b=4\rho$) and Lemma \ref{lemma_lower_order} to get
\begin{align}
\nonumber J & \apprle \al_0^{-1+d} \left[ \al_0^{1-\frac{2}{p(z_0)}-\be} \fiint_{\qfor} \left| \frac{u}{\scalex{\al_0}{z_0}\rho} \right|^2 \ dz + \fiint_{\qfor} \left| \frac{u}{\scalex{\al_0}{z_0} \rho} \right|^{p(z)(1-\be)} dz \right] \\
\nonumber & \apprle \al_0^{-1+d} \left[ \al_0^{1-\frac{2}{p(z_0)}-\be} \fiint_{\qfor} \left| \frac{u}{\scalex{\al_0}{z_0}\rho} \right|^2 \ dz + \fiint_{\qfor} \left| \frac{u}{\scalex{\al_0}{z_0} \rho} \right|^{p^+_{\qfve}(1-\be)} dz + 1 \right] \\
\label{8.2pf_3} & \apprle \al_0^{-1+d} \left[ \al_0^{1-\frac{2}{p(z_0)}-\be} \al_0^{\frac{2}{p(z_0)}} + \al_0^{\frac{p^+_{\qfve}(1-\be)}{p(z_0)}} + 1 \right] \apprle \al_0^{d-\be},
\end{align}
where we have used the following inequality:
$$
\al_0^{\lbr \frac{ p^+_{\qfve}}{p(z_0)}\rbr (1-\be) }  \leq \al_0^{\frac{\left( p^+_{\qfve} - p^-_{\qfve} \right)(1-\be)}{p(z_0)}} \al_0^{1-\be} \apprle \al_0^{1-\be}.
$$
In addition, it follows from the definition of $\tilde{\mathcal{M}}(z)$, the boundedness of the strong Maximal function from Lemma \ref{max_bnd} and Theorem \ref{measure_density_poincare} that
\begin{align}
\nonumber \fiint_{\qtoo} \tilde{\mathcal{M}}(z)^{1-\be} \, dz & \leq (c_e\al_0)^{1-\be} + \fiint_{\qtoo} g(z) \, dz \\
\nonumber & \apprle \al_0^{1-\be} +  \fiint_{\qfor} \left( (|\nabla u|+1)^{p(z)(1-\be)}  +  \left|\frac{u}{\scalex{\al_0}{z_0}\rho}\right|^{p(z)(1-\be)} \right) dz \\
\nonumber & \apprle \al_0^{1-\be} +   \fiint_{\qfiv} (|\nabla u|+1)^{p(z)(1-\be)} \, dz\\
\label{8.2pf_4} & \apprle \al_0^{1-\be}.
\end{align}
Combining \eqref{8.2pf_1}, \eqref{8.2pf_2},\eqref{8.2pf_3} and \eqref{8.2pf_4} and using  H\"older inequality with exponents $r=\frac{2(1-\be)}{\be(\sig-q_1)}$ and its conjugate $r'=\frac{2(1-\be)}{2(1-\be)-\be(\sig-q_1)}$, we deduce that
\begin{align}
\nonumber I_{\sig} & \apprle J^{\frac{\sig-q_1}{2}} \fint_{\itoo} \left( \fint_{\btoo} |\nabla u|^{q_1} \, dx \right) \left( \fint_{\btoo \times \{t\}} \tilde{\mathcal{M}}(x,t)^{1-\be} \, dx \right)^{\frac{\be(\sig-q_1)}{2(1-\be)}} dt \\
\nonumber & \apprle J^{\frac{\sig-q_1}{2}} \left( \fiint_{\qtoo} |\nabla u|^{q_1 r'} \, dz \right)^{\frac{1}{r'}} \left( \fiint_{\qtoo} \tilde{\mathcal{M}}(z)^{1-\be} \, dz \right)^{\frac{\be(\sig-q_1)}{2(1-\be)}} \\
\label{8.2pf_5} & \apprle \al_0^{\frac{(d-\be)(\sig-q_1)}{2}} \left( \fiint_{\qtoo} |\nabla u|^{q_1 r'} \, dz\right)^{\frac{1}{r'}} \al_0^{\frac{\be(\sig-q_1)}{2}} =  \al_0^{\frac{d(\sig-q_1)}{2}} \left( \fiint_{\qtoo} |\nabla u|^{q_1 r'} \, dz\right)^{\frac{1}{r'}}.
\end{align}
An easy computation shows that
$$
q_1 r' = \frac{n\sig}{n+2(1-\be)                                                                                                                                                                                                                                                                                                                                                                                                                                                                                                                                                                                        } \cdot \frac{2(1-\be)}{2(1-\be)-\be(\sig-q_1)} = \frac{n\sig}{n+2-\be(\sig+2)}.
$$
We claim that there exists $\be = \be(n,\plog) > 0$ such that
\begin{equation}
\label{8.2pf_claim}
q_1 r' \leq \frac{p^-(1-\be)}{\tq} \leq \frac{p^-_{\qfve}(1-\be)}{\tq},
\end{equation}
where $\tq := \min \left\{ \frac{1}{2} \left( \frac{(n+2)p^-}{2n} + 1 \right), \frac{(n+1)^2}{n(n+2)}\right\} > 1$.

Indeed, when $\sig=2$, we have $q_1 r' = \frac{2n}{n+2-4\be}$.
Setting $\Psi(\be) := \frac{p^-(1-\be)(n+2-4\be)}{2n}$, we observe that $\Psi$ is continuous on $\RR$ and that $\Psi(0) = \frac{(n+2)p^-}{2n} > 1$.
Therefore, there exists $\be_0 = \be_0(n,\plog) > 0$ such that $\Psi(\be) \geq \frac{1}{2} \left( \frac{(n+2)p^-}{2n} + 1 \right) \geq \tq$ for all $ \be \in (0,\be_0)$. This yields \eqref{8.2pf_claim}.
When $\sig = p^+_{\qfve}(1-\be)$, we have $q_1 r' = \frac{np^+_{\qfve}(1-\be)}{n+2-\be \left( p^+_{\qfve}(1-\be)+2 \right)}$, and hence
\begin{equation*}
\frac{p^-(1-\be)}{q_1 r'} = \frac{p^- \left[ n+2-\be \left( p^+_{\qfve}(1-\be)+2 \right) \right]}{np^+_{\qfve}} \geq \frac{p^- \left[ n+2-\be \left( p^+ + 2 \right) \right]}{np^+} \geq \frac{p^-(n+1)}{np^+}, 
\end{equation*}
if $\beta < \frac{1}{p^+ +2}$.
Since
$$
\frac{p^-}{p^+} = 1 - \frac{p^+ - p^-}{p^+} \geq 1-(p^+ - p^-) \overset{\eqref{def_de_small}}{\geq} 1 - \frac{1}{n+2} = \frac{n+1}{n+2},
$$
we obtain
$$
\frac{p^-(1-\be)}{q_1 r'} \geq \frac{n+1}{n} \cdot \frac{p^-}{p^+} \geq \frac{(n+1)^2}{n(n+2)} \geq \tq,
$$
which proves our claim \eqref{8.2pf_claim}.

Now, it follows from \eqref{8.2pf_5}, \eqref{8.2pf_claim} and the H\"older inequality that
\begin{align*}
I_{\sig} & \apprle \al_0^{\frac{d(\sig-q_1)}{2}} \left( \fiint_{\qtoo} |\nabla u|^{q_1 r'} \, dz\right)^{\frac{1}{r'}} \\
& \apprle \al_0^{\frac{d(\sig-q_1)}{2}} \left( \fiint_{\qtoo} |\nabla u|^{\frac{p^-_{\qfve}(1-\be)}{\tq}} \, dz\right)^{\frac{q_1 \tq}{p^-_{\qfve}(1-\be)}} \\
& \apprle \al_0^{\frac{d(\sig-q_1)}{2}} \left( \fiint_{\qtoo} (|\nabla u|+1)^{\frac{p(z)(1-\be)}{\tq}} \, dz\right)^{\frac{q_1 \tq}{p^-_{\qfve}(1-\be)}} \\
& \apprle \al_0^{\frac{d(\sig-q_1)}{2}} \left( \fiint_{\qtoo} (|\nabla u|+1)^{\frac{p(z)(1-\be)}{\tq}} \, dz\right)^{\tq} \left( \fiint_{\qtoo} (|\nabla u|+1)^{p(z)(1-\be)} \, dz\right)^{\frac{q_1}{p^-_{\qfve}(1-\be)}-1} \\
& \apprle \al_0^{\frac{d(\sig-q_1)}{2}} \left( \fiint_{\qtoo} (|\nabla u|+1)^{\frac{p(z)(1-\be)}{\tq}} \, dz\right)^{\tq} \al_0^{\frac{q_1}{p^-_{\qfve}}-(1-\be)} \\
& =  \al_0^{\frac{d(\sig-q_1)}{2}+\frac{q_1}{p^-_{\qfve}}-(1-\be)} \left( \fiint_{\qtoo} (|\nabla u|+1)^{\frac{p(z)(1-\be)}{\tq}} \, dz\right)^{\tq}.
\end{align*}
We note from the definition of $q_1$ that
\begin{align*}
& \frac{d(\sig-q_1)}{2}+\frac{q_1}{p^-_{\qfve}}-(1-\be) \\
& \quad = \frac{d\sig(1-\be)}{n+2-2\be}+\frac{n\sig}{n+2-2\be} \cdot \frac{1}{p^-_{\qfve}} -(1-\be) \\
& \quad \leq \frac{d\sig(1-\be)}{n+2-2\be}+\frac{n\sig}{n+2-2\be} \cdot \frac{1}{p^+_{\qfve}} -(1-\be) + \sig \left( p^+_{\qfve}-p^-_{\qfve} \right).
\end{align*}
Therefore, we have
\begin{equation}
\label{8.2pf_6}
I_{\sig} \apprle \al_0^{\frac{d\sig(1-\be)}{n+2-2\be}+\frac{n\sig}{n+2-2\be} \cdot \frac{1}{p^+_{\qfve}} -(1-\be)} \left( \fiint_{\qtoo} (|\nabla u|+1)^{\frac{p(z)(1-\be)}{\tq}} \, dz\right)^{\tq}.
\end{equation}
Combining \eqref{8.2pf_0} and \eqref{8.2pf_6} gives
\begin{align*}
 \fiint_{\qnot} |\nabla u|^{p(z)(1-\be)} \, dz & \apprle \al_0^{-\frac{2}{p(z_0)}+\frac{2d(1-\be)}{n+2-2\be}+\frac{2n}{n+2-2\be} \cdot \frac{1}{p^+_{\qfve}}} \left( \fiint_{\qtoo} (|\nabla u|+1)^{\frac{p(z)(1-\be)}{\tq}} \, dz\right)^{\tq} \\
& \qquad +  \al_0^{\frac{dp^+_{\qfve}(1-\be)^2}{n+2-2\be}+\frac{n(1-\be)}{n+2-2\be}-(1-\be)} \left( \fiint_{\qtoo} (|\nabla u|+1)^{\frac{p(z)(1-\be)}{\tq}} \, dz\right)^{\tq} + 1.
\end{align*}
Here, we see from \eqref{exp_d} that the following two bounds hold:
\begin{align*}
& -\frac{2}{p(z_0)}+\frac{2d(1-\be)}{n+2-2\be}+\frac{2n}{n+2-2\be} \cdot \frac{1}{p^+_{\qfve}} \\
&\qquad \qquad\leq -\frac{2}{p^+_{\qfve}}+\frac{2d(1-\be)}{n+2-2\be}+\frac{2n}{n+2-2\be} \cdot \frac{1}{p^+_{\qfve}} \\
& \qquad\qquad= \frac{2d(1-\be)}{n+2-2\be} \left( d-\frac{2}{p^+_{\qfve}} \right) \leq \frac{2d(1-\be)}{n+2-2\be} \left( d-\frac{2}{p^+} \right) \leq 0,
\end{align*}
and 
\begin{align*}
\frac{dp^+_{\qfve}(1-\be)^2}{n+2-2\be}+\frac{n(1-\be)}{n+2-2\be} -(1-\be) & = \frac{(1-\be)^2}{n+2-2\be} \left( dp^+_{\qfve}-2 \right) \\
& \leq \frac{(1-\be)^2}{n+2-2\be} \left( dp^+-2 \right) \leq 0.
\end{align*}
Consequently, we get
\begin{align*}
\fiint_{\qnot} |\nabla u|^{p(z)(1-\be)} \, dz & \apprle \left( \fiint_{\qtoo} (|\nabla u|+1)^{\frac{p(z)(1-\be)}{\tq}} \, dz\right)^{\tq} + 1  \\
& \apprle \left( \fiint_{\qtoo} |\nabla u|^{\frac{p(z)(1-\be)}{\tq}} \, dz\right)^{\tq} + 1 ,
\end{align*}
which is the desired estimate.
\end{proof}

 \section{Proof of main theorem}
 \label{section_seven}
 
We are now ready to prove Theorem \ref{thm_main}. Let us  consider a parabolic cylinder $Q_{2r} \equiv Q_{2r}(\mathfrak{z}_0) \subset \mbfq$ and define
\begin{equation}
\label{def_la_0}
\la_0^{-\frac{n}{p_m}+\frac{nd}{2}+d-\be} := \fiint_{Q_{2r}} (|\nabla u|+1)^{p(z)(1-\be)} \, dz \geq 1,
\end{equation}
where $p_m := \inf_{Q_{2r}} p(\cdot)$.
We note from \eqref{exp_d} that if $0 < \be < \frac{1}{2} \left( \frac{n+2}{2}d-\frac{n}{p^-} \right)$, we get
\begin{equation}
\label{sec9pf_vt}
-\frac{n}{p_m}+\frac{nd}{2}+d-\be \geq -\frac{n}{p^-}+\frac{n+2}{2}d-\be > \frac{1}{2} \left( \frac{n+2}{2}d-\frac{n}{p^-} \right) > 0.
\end{equation}
For fixed $r \leq r_1 < r_2 \leq 2r$, we consider the concentric parabolic cylinders $Q_r \subseteq Q_{r_1} \subset Q_{r_2} \subseteq Q_{2r}$, and radii $\rho$ satisfying
\begin{equation}
\label{range_rho}
\frac{r_2-r_1}{4\mfi} \leq \rho \leq r_2-r_1,
\end{equation}
where $\mfi$ is the constant from Lemma \ref{lemma_vitali}.
We remark that this choice ensures $\qnot \subset Q_{r_2}$ for $z_0 \in Q_{r_1}$.
Let $\al_0$ be any number such that
\begin{equation}
\label{def_al_0}
\al_0 \geq B \la_0, \quad \text{where} \ \, B^{-\frac{n}{p_m}+\frac{nd}{2}+d-\be} := 	\left( \frac{8\mfi r}{r_2-r_1} \right)^{n+2}.
\end{equation}
Then we have
\begin{align}
\nonumber \fiint_{\qnot} (|\nabla u|+1)^{p(z)(1-\be)} \, dz & \leq \frac{|Q_{2r}|}{|\qnot|} \fiint_{Q_{2r}} (|\nabla u|+1)^{p(z)(1-\be)} \, dz \\
\nonumber & \leq \left( \frac{2r}{\rho} \right)^{n+2} \al_0^{\frac{n}{p(z_0)}-\frac{nd}{2}+1-d} \la_0^{-\frac{n}{p_m}+\frac{nd}{2}+d-\be} \\
% \nonumber & \leq \left( \frac{2r}{\rho} \right)^{n+2} \al_0^{\frac{n}{p_m}-\frac{nd}{2}+1-d} \la_0^{-\frac{n}{p_m}+\frac{nd}{2}+d-\be} \\
\nonumber & \leq \left( \frac{2r}{\rho} \right)^{n+2} B^{\frac{n}{p_m}-\frac{nd}{2}-d+\be} \al_0^{1-\be} \\
\label{sec9pf_1} & \leq \al_0^{1-\be},
\end{align}
for any $\rho$ satisfying \eqref{range_rho}.
We now consider the upper level set
\begin{equation*}
E(\al_0; r_1) := \left\{ z \in Q_{r_1} : |\nabla u(z)|^{p(z)} > \al_0 \right\}.
\end{equation*}
Then it follows from the Lebesgue differentiation theorem that for every $z_0 \in E(\al_0; r_1)$, there holds
\begin{equation}
\label{sec9pf_2}
\lim_{\varrho \searrow 0} \fiint_{Q_{\varrho}^{\al_0}(z_0)} (|\nabla u|+1)^{p(z)(1-\be)} \, dz \geq |\nabla u(z_0)|^{p(z_0)(1-\be)} > \al_0^{1-\be}.
\end{equation}
Thus, from \eqref{sec9pf_1} and \eqref{sec9pf_2}, we see that there exists a radius $\rho_{z_0} \in \left( 0, \frac{r_2-r_1}{4\mfi} \right)$ such that
\begin{equation}
\label{sec9pf_stop1}
\fiint_{\qzero} (|\nabla u|+1)^{p(z)(1-\be)} \, dz = \al_0^{1-\be}
\end{equation}
and
\begin{equation}
\label{sec9pf_stop2}
\fiint_{\qnot} (|\nabla u|+1)^{p(z)(1-\be)} \, dz \leq \al_0^{1-\be}, \quad \forall \rho \in \left( \rho_{z_0}, r_2-r_1 \right].
\end{equation}
Observing that $Q_{4\mfi\rho_{z_0}}^{\al_0}(z_0) \subset Q_{r_2}$, we obtain
\begin{equation}
\label{sec9pf_3}
\al_0^{1-\be} = \fiint_{\qzero} (|\nabla u|+1)^{p(z)(1-\be)} \, dz \qquad \text{and} \qquad \fiint_{\qzeros} (|\nabla u|+1)^{p(z)(1-\be)} \, dz \leq \al_0^{1-\be}.
\end{equation}
Therefore, we can apply Lemma \ref{lemma_reverse_Holder} which gives
\begin{equation}
\label{sec9pf_reverse}
\fiint_{\qzero} |\nabla u|^{p(z)(1-\be)} \, dz \apprle \left( \fiint_{\qzeros} |\nabla u|^{\frac{p(z)(1-\be)}{\tq}} \, dz \right)^{\tq} + 1.
\end{equation}
We observe that for any $\nu \in (0,1)$,
\begin{align}
\nonumber & \left( \fiint_{\qzeros} |\nabla u|^{\frac{p(z)(1-\be)}{\tq}} \, dz \right)^{\tq} \\
\nonumber & \quad \leq \left( (\nu \al_0)^{\frac{1-\be}{\tq}} + \frac{1}{|\qzeros|} \iint_{\qzeros \cap E(\nu \al_0; r_2)} |\nabla u|^{\frac{p(z)(1-\be)}{\tq}} \, dz \right)^{\tq} \\
% \nonumber & \quad \leq c (\nu \al_0)^{1-\be} + c \left( \frac{1}{|\qzeros|} \iint_{\qzeros \cap E(\nu \al_0; r_2)} |\nabla u|^{\frac{p(z)(1-\be)}{\tq}} \, dz \right)^{\tq} \\
% \nonumber & \quad \leq c (\nu \al_0)^{1-\be} + \frac{c}{|\qzeros|} \iint_{\qzeros \cap E(\nu \al_0; r_2)} |\nabla u|^{\frac{p(z)(1-\be)}{\tq}} \, dz  \left( \fiint_{\qzeros} |\nabla u|^{\frac{p(z)(1-\be)}{\tq}} \, dz \right)^{\tq-1} \\
\nonumber & \quad \apprle (\nu \al_0)^{1-\be} + \frac{1}{|\qzeros|} \iint_{\qzeros \cap E(\nu \al_0; r_2)} |\nabla u|^{\frac{p(z)(1-\be)}{\tq}} \, dz  \left( \fiint_{\qzeros} |\nabla u|^{p(z)(1-\be)} \, dz \right)^{\frac{\tq-1}{\tq}} \\
\nonumber & \quad \apprle (\nu \al_0)^{1-\be} + \frac{ \al_0^{\frac{(\tq-1)(1-\be)}{\tq}}}{|\qzeros|} \iint_{\qzeros \cap E(\nu \al_0; r_2)} |\nabla u|^{\frac{p(z)(1-\be)}{\tq}} \, dz \\
\label{sec9pf_4} & \quad =  \nu^{1-\be} \fiint_{\qzero} (|\nabla u|+1)^{p(z)(1-\be)} \, dz  + \frac{1}{|\qzeros|} \iint_{\qzeros \cap E(\nu \al_0; r_2)} \al_0^{\frac{(\tq-1)(1-\be)}{\tq}} |\nabla u|^{\frac{p(z)(1-\be)}{\tq}} \, dz,
\end{align}
where we have used \eqref{sec9pf_3}.
Moreover, we have
\begin{align}
 1 & \leq (\nu \al_0)^{1-\be} + \chi_{\{ \nu \al_0 < 1\}}  \leq \nu^{1-\be} \fiint_{\qzero} (|\nabla u|+1)^{p(z)(1-\be)} \, dz + \chi_{\{ \nu \al_0 < 1\}},\label{sec9pf_4_2}
\end{align}
where $\chi_{\{ \nu \al_0 < 1 \}} := 1$ if $\nu \al_0 < 1$, and $\chi_{\{ \nu \al_0 < 1 \}} := 0$ if $\nu \al_0 \geq 1$.
Choosing $\nu \equiv \nu(n,p^-, p^+ ,\La_0, \La_1) \in (0,1)$ small enough and combining \eqref{sec9pf_reverse} with \eqref{sec9pf_4} and \eqref{sec9pf_4_2}, we discover that
\begin{equation*}
\fiint_{\qzero} |\nabla u|^{p(z)(1-\be)} \, dz \apprle \frac{1}{|\qzeros|} \iint_{\qzeros \cap E(\nu \al_0; r_2)} \al_0^{\frac{(\tq-1)(1-\be)}{\tq}} |\nabla u|^{\frac{p(z)(1-\be)}{\tq}} \, dz +  \chi_{\{ \nu \al_0 < 1\}}.
\end{equation*}
Then it follows from \eqref{sec9pf_stop1} and \eqref{sec9pf_stop2} that
\begin{equation*}
\fiint_{\qzeross} |\nabla u|^{p(z)(1-\be)} \, dz \apprle \frac{1}{|\qzeros|} \iint_{\qzeros \cap E(\nu \al_0; r_2)} \al_0^{\frac{(\tq-1)(1-\be)}{\tq}} |\nabla u|^{\frac{p(z)(1-\be)}{\tq}} \, dz +  \chi_{\{ \nu \al_0 < 1\}}.
\end{equation*}
Multiplying this estimate by $|\qzeross|$, we get
\begin{multline}
\label{sec9pf_5}
\iint_{\qzeross} |\nabla u|^{p(z)(1-\be)} \, dz 
\apprle \iint_{\qzeros \cap E(\nu \al_0; r_2)} \al_0^{\frac{(\tq-1)(1-\be)}{\tq}} |\nabla u|^{\frac{p(z)(1-\be)}{\tq}} \, dz +  \chi_{\{ \nu \al_0 < 1\}} |\qzero|.
\end{multline}
Note that the family $\left\{ \qzerof \right\}_{z_0 \in E(\al_0; r_1)}$  covers $E(\al_0; r_1)$.
Hence, by the Vitali covering lemma, there exists a countable family of disjoint sets $\left\{ \qzerofi \right\}_{i=1}^{\infty}$ with $\rho_{z_i} \in \left( 0, \frac{r_2-r_1}{4\mfi} \right)$ such that
\begin{equation*}
E(\al_0; r_1) \subset \bigcup_{i=1}^{\infty} \qzerossi \subset Q_{r_2},
\end{equation*}
up to a set of measure zero.
Since the parabolic cylinders $\left\{ \qzerofi \right\}_{i=1}^{\infty}$ are disjoint, we obtain from \eqref{sec9pf_5} that
\begin{equation}
\label{sec9pf_6}
\iint_{E(\al_0; r_1)} |\nabla u|^{p(z)(1-\be)} \, dz \apprle \iint_{E(\nu \al_0; r_2)} \al_0^{\frac{(\tq-1)(1-\be)}{\tq}} |\nabla u|^{\frac{p(z)(1-\be)}{\tq}} \, dz +  \chi_{\{ \nu \al_0 < 1\}} |Q_{r_2}|.
\end{equation}

To estimate the integral of $|\nabla u|^{p(z)}$, we recall  Fubini's theorem to discover
\begin{equation*}
\be \int_0^M \al_0^{\be-1} \iint_{E(\al_0; r_1)} |\nabla u|^{p(z)(1-\be)} \, dz \, d\al_0 = \iint_{Q_{r_1}} |\nabla u|^{p(z)(1-\be)} \left( |\nabla u|^{p(z)} \right)_M^{\be} dz,
\end{equation*}
for any $M>0$, where $\left( |\nabla u|^{p(z)} \right)_M := \min \left\{ |\nabla u|^{p(z)}, M \right\}$ is the truncated function of $|\nabla u|^{p(z)}$.
We note that the right-hand side of the above identity is finite, as $|\nabla u|^{p(z)(1-\be)} \in L^1$ and the truncated function $\left( |\nabla u|^{p(z)} \right)_M$ is bounded.
Then it follow from \eqref{def_al_0}, \eqref{sec9pf_6} and a change of variables that the following holds for any $M> \max \left\{ B\la_0, \frac{1}{\nu} \right\}$, there holds
\begin{align*}
 \iint_{Q_{r_1}} &|\nabla u|^{p(z)(1-\be)} \left( |\nabla u|^{p(z)} \right)_M^{\be} dz \\
% & = \be \int_0^{B\la_0} \al_0^{\be-1} \iint_{E(\al_0; r_1)} |\nabla u|^{p(z)(1-\be)} \, dz \, d\al_0 + \be \int_{B\la_0}^M \al_0^{\be-1} \iint_{E(\al_0; r_1)} |\nabla u|^{p(z)(1-\be)} \, dz \, d\al_0 \\
& \leq \be \int_0^{B\la_0} \al_0^{\be-1} \iint_{Q_{r_1}} |\nabla u|^{p(z)(1-\be)} \, dz \, d\al_0 \\
& \quad + C \be \int_{B\la_0}^M \al_0^{\be-1} \left[ \iint_{E(\nu \al_0; r_2)} \al_0^{\frac{(\tq-1)(1-\be)}{\tq}} |\nabla u|^{\frac{p(z)(1-\be)}{\tq}} \, dz + \chi_{\{ \nu \al_0 < 1\}} |Q_{r_2}| \right] d\al_0 \\
& \leq (B\la_0)^{\be} \iint_{Q_{r_1}} |\nabla u|^{p(z)(1-\be)} \, dz \, d\al_0 \\
& \quad + C \be \int_0^M \al_0^{\frac{\be-1}{\tq}} \iint_{E(\nu \al_0; r_2)} |\nabla u|^{\frac{p(z)(1-\be)}{\tq}} \, dz \, d\al_0 + C |Q_{r_2}| \be \int_0^M \al_0^{\be-1} \chi_{\{ \nu \al_0 < 1\}} d\al_0 \\
% & = (B\la_0)^{\be} \iint_{Q_{r_1}} |\nabla u|^{p(z)(1-\be)} \, dz \, d\al_0 \\
% & \quad + C \be \int_0^M \left( \fraC{\al_0}{\nu} \right)^{\fraC{\be-1}{\tq}} \iint_{E(\al_0; r_2)} |\nabla u|^{\fraC{p(z)(1-\be)}{\tq}} \, dz \, \fraC{1}{\nu} \, d\al_0 + C |Q_{r_2}| \be \int_0^{\fraC{1}{\nu}} \al_0^{\be-1} d\al_0 \\
% & = (B\la_0)^{\be} \iint_{Q_{r_1}} |\nabla u|^{p(z)(1-\be)} \, dz \, d\al_0 \\
% & \quad +\fraC{C \be \tq}{\tq+\be-1} \left( \fraC{1}{\nu} \right)^{\frac{\tq+\be-1}{\tq}} \iint_{Q_{r_2}} |\nabla u|^{\frac{p(z)(1-\be)}{\tq}} \left( |\nabla u|^{p(z)} \right)_M^{\frac{\tq+\be-1}{\tq}} dz + c |Q_{r_2}| \left( \frac{1}{\nu} \right)^{\be} \\
& \leq B \la_0^{\be} \iint_{Q_{r_1}} |\nabla u|^{p(z)(1-\be)} \, dz \, d\al_0 +\frac{C \be \tq}{\nu(\tq-1)} \iint_{Q_{r_2}} |\nabla u|^{\frac{p(z)(1-\be)}{\tq}} \left( |\nabla u|^{p(z)} \right)_M^{\frac{\tq+\be-1}{\tq}} dz + \frac{C}{\nu}  |Q_{r_2}| \\
& \leq B \la_0^{\be} \iint_{Q_{r_1}} |\nabla u|^{p(z)(1-\be)} \, dz \, d\al_0 +\frac{C \be \tq}{\nu(\tq-1)} \iint_{Q_{r_2}} |\nabla u|^{p(z)(1-\be)} \left( |\nabla u|^{p(z)} \right)_M^{\be} dz + \frac{C}{\nu} |Q_{r_2}|,
\end{align*}
where for the last inequality we have used the fact that $\left( |\nabla u|^{p(z)} \right)_M \leq |\nabla u|^{p(z)}$.
Since $\tq \equiv \tq(n,\plog)>1$ and $\nu \equiv \nu(n,\plog,\lamot) \in (0,1)$ are universal constants, we choose
$$ 0 < \be \leq \be_0 \equiv \be_0(n,\plog,\lamot,m_e) := \frac{\nu(\tq-1)}{2C\tq} $$
to find that
\begin{multline*}
\iint_{Q_{r_1}} |\nabla u|^{p(z)(1-\be)} \left( |\nabla u|^{p(z)} \right)_M^{\be} dz \\
\leq \frac{1}{2} \iint_{Q_{r_2}} |\nabla u|^{p(z)(1-\be)} \left( |\nabla u|^{p(z)} \right)_M^{\be} dz + \frac{C (8\mfi r)^{\frac{n+2}{\vt}} \la_0^{\be}}{(r_2-r_1)^{\frac{n+2}{\vt}}} \iint_{Q_{2r}} |\nabla u|^{p(z)(1-\be)} \, dz + C|Q_{2r}|,
\end{multline*}
where $\vt := -\frac{n}{p_m}+\frac{nd}{2}+d-\be > 0$.
Since $r \leq r_1 < r_2 \leq 2r$ are arbitrary, we apply Lemma \ref{iter_lemma} to get
\begin{equation*}
\iint_{Q_r} |\nabla u|^{p(z)(1-\be)} \left( |\nabla u|^{p(z)} \right)_M^{\be} dz \apprle C(\vt) \left[ \la_0^{\be} \iint_{Q_{2r}} |\nabla u|^{p(z)(1-\be)} \, dz + |Q_{2r}| \right].
\end{equation*}
Letting $M \to \infty$ and using Fatou's lemma, we have
\begin{equation*}
\iint_{Q_r} |\nabla u|^{p(z)} \, dz \apprle \la_0^{\be} \iint_{Q_{2r}} |\nabla u|^{p(z)(1-\be)} \, dz + |Q_{2r}|.
\end{equation*}
We remark from \eqref{sec9pf_vt} that the dependence on $\vt$ can be eliminated.
Recalling \eqref{def_la_0}, we see that
\begin{equation}
\label{sec9pf_7}
\fiint_{Q_r(\mathfrak{z}_0)} |\nabla u|^{p(z)} \, dz \apprle \left( \fiint_{Q_{2r}(\mathfrak{z}_0)} |\nabla u|^{p(z)(1-\be)} \, dz \right)^{1+\frac{\be}{\vt}} + 1	.
\end{equation}
It only remains to replace $\vt$ in \eqref{sec9pf_7} by $\vt_0 := -\frac{n}{p_0}+\frac{nd}{2}+d-\be$, where $p_0 := p(\mathfrak{z}_0)$.
Observing that
$$
\frac{\be}{\vt}-\frac{\be}{\vt_0} = \frac{\be(\vt_0-\vt)}{\vt \vt_0} = \frac{\be \left( \frac{n}{p_m}-\frac{n}{p_0} \right)}{\vt \vt_0} = \frac{n \be (p_0-p_m)}{\vt \vt_0 p_0 p_m},
$$
we have
$$
0 \leq \frac{\be}{\vt}-\frac{\be}{\vt_0} = \frac{n \be (p_0-p_m)}{\vt \vt_0 p_0 p_m} \overset{\eqref{sec9pf_vt}}{\leq} \frac{4n \be (p_0-p_m)}{\left(  \frac{n+2}{2}d-\frac{n}{p^-} \right)^2} \leq c \modp(4r),
$$
and hence
\begin{align*}
\left( \fiint_{Q_{2r}} |\nabla u|^{p(z)(1-\be)} \, dz \right)^{\frac{\be}{\vt}-\frac{\be}{\vt_0}} & \leq \left( \fiint_{Q_{2r}} (|\nabla u|+1)^{p(z)(1-\be)} \, dz \right)^{c\modp(4r)} \\
& \overset{\eqref{def_M_0}}{\apprle}  \left( \frac{1}{2r} \right)^{(n+2)c\modp(4r)} \bM_0^{c\modp(4r)} \\
& \apprle \left( \frac{1}{2r} \right)^{(n+2)c\modp(4r)} \left( \frac{1}{1024r} \right)^{c\modp(4r)} \leq C.
\end{align*}
Consequently, we obtain that
\begin{equation*}
\fiint_{Q_r(\mathfrak{z}_0)} |\nabla u|^{p(z)} \, dz \apprle  \left( \fiint_{Q_{2r}(\mathfrak{z}_0)} |\nabla u|^{p(z)(1-\be)} \, dz \right)^{1+\frac{\be}{\vt_0}} + 1,
\end{equation*}
which completes the proof of Theorem \ref{thm_main}.

% \section*{Compliance with Ethical Standards}
% \begin{itemize}
% \item The authors declare that they have no conflict of interest.
% \item This chapter does not contain any studies with human participants or animals performed by any of
% the authors.
% \item Informed consent was obtained from all individual participants included in the
% study.
% \end{itemize}

\end{document}